\documentclass[aap,MSNbibl,seceqn,dvips]{arximspdf}
\usepackage{graphicx}
\doi{10.1214/13-AAP921} 
\volume{24}
\issue{1}
\pubyear{2014}
\firstpage{235}
\lastpage{291}

\makeatletter
\newcommand{\rrVert}{\Vert}
\newcommand{\rrvert}{\vert}
\newcommand{\llVert}{\Vert}
\newcommand{\llvert}{\vert}

\newcommand{\underset}[2]{\,\mathop{#2}\limits_{#1}\,}

\newtheorem{theorem}{Theorem}[section]
\newtheorem{prop}[theorem]{Proposition}
\newtheorem{lemma}[theorem]{Lemma}
\newtheorem{corollary}[theorem]{Corollary}

\newproclaim{definition}[theorem]{Definition}
\newproclaim{hyp}{Assumption}
\newproclaim{remark}[theorem]{Remark}

\def\BB{\mathcal B}
\def\HH{\mathcal H}

\def\R{{\mathbb R}}
\def\T{{\mathbb T}}
\def\P{{\mathbb P}}
\def\E{{\mathbb E}}
\def\G{{\mathbb G}}
\def\N{{\mathbb N}}

\makeatother

\begin{document}
\begin{frontmatter}

\title{Deviation inequalities, moderate deviations
and some limit theorems for bifurcating Markov chains with application}
\runtitle{Deviation inequalities and limit theorems for BMC}

\begin{aug}
\author[A]{\fnms{S. Val\`ere} \snm{Bitseki Penda}\corref{}\ead[label=e1]{Valere.Bitsekipenda@math.univ-bpclermont.fr}},
\author[A]{\fnms{Hac\`ene} \snm{Djellout}\ead[label=e2]{Hacene.Djellout@math.univ-bpclermont.fr}}
\and\break
\author[B]{\fnms{Arnaud} \snm{Guillin}\ead[label=e3]{Arnaud.Guillin@math.univ-bpclermont.fr}}
\runauthor{S. V. Bitseki Penda, H. Djellout and A. Guillin}
\affiliation{Universit\'e Blaise Pascal}
\address[A]{S. V. Bitseki Penda\\
H. Djellout\\
Laboratoire de Math\'ematiques\\
Universit\'e Blaise Pascal\\
CNRS UMR 6620\\
24 avenue des Landais\\
BP 80026, 63177 Aubi\`ere\\
France\\
E-mail:\\
\hphantom{\quad}\printead*{e1}\\
\hphantom{\quad}\printead*{e2}} 
\address[B]{A. Guillin\\
Institut Universitaire de France\\
\quad et Laboratoire de Math\'ematiques\\
Universit\'e Blaise Pascal\\
CNRS UMR 6620\\
24 avenue des Landais\\
BP 80026, 63177 Aubi\`ere\\
France\\
E-mail:\\
\hphantom{\quad}\printead*{e3}}
\end{aug}

\received{\smonth{11} \syear{2011}}
\revised{\smonth{1} \syear{2013}}

%
\begin{abstract}
First, under a geometric ergodicity assumption, we provide some
limit theorems and some probability inequalities for the bifurcating
Markov chains (BMC). The BMC model was introduced by Guyon to detect
cellular aging from cell lineage, and our aim is thus to complete
his asymptotic results. The deviation inequalities are then applied
to derive first result on the moderate deviation principle (MDP) for
a functional of the BMC with a restricted range of speed, but with a
function which can be unbounded. Next, under a uniform geometric
ergodicity assumption, we provide deviation inequalities for the BMC
and apply them to derive a second result on the MDP for a bounded
functional of the BMC with a larger range of speed. As statistical
applications, we provide superexponential convergence in probability
and deviation inequalities (for either the Gaussian setting or the
bounded setting), and the MDP for least square estimators of the
parameters of a first-order bifurcating autoregressive process.
\end{abstract}

%

\end{frontmatter}

\section{Introduction}

Bifurcating Markov chains (BMC) are an adaptation of (usual) Markov
chains to the data of a regular binary tree; see below for a more
precise definition. In other terms, it is a Markov chain for which the
index set is a regular binary tree. They are appropriate, for example,
in the modeling of cell lineage data when each cell in one generation
gives birth to two offspring in the next. Recently, they have received
a great deal of attention because of the experiments of biologists on
aging of Escherichia Coli; see \cite{SteMadPauTad,GuAl}. E. Coli
is a rod-shaped bacterium which reproduces by dividing in the middle,
thus producing two cells, one which already existed, that we call old
pole progeny, and the other which is new, that we call new pole
progeny. The aim of their experiments was to look for evidence of aging
in E. Coli. In this section, we will introduce the model that allowed
the authors of \cite{GuAl} to study the aging of E. Coli and we refer
to their works for further motivations and insights on the data leading
to the model studied here. This model is a typical example of
bifurcating Markovian dynamics, and it has been the motivation for the
rigorous mathematical study of BMC in \cite{Guyon}. This also motivates
Sections \ref{momentscontrol} and \ref{expoprobaineq} in the sequel,
where we give a rigorous asymptotic (and nonasymptotic) study of BMC
under geometric ergodicity and uniform geometric ergodicity
assumptions.

%
\begin{figure}[b]

\includegraphics{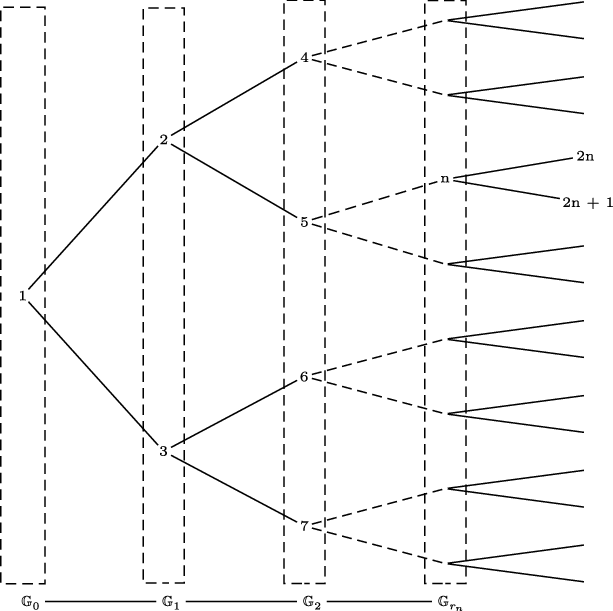}

\caption{The binary tree $\mathbb{T}$.}\label{Fiarbrebinaire}
\end{figure}
%

\subsection{The model}\label{model}

Let $\mathbb{T}$ be a binary regular tree in which each vertex is
seen as a positive integer different from 0; see Figure
\ref{Fiarbrebinaire}. For $r\in\mathbb{N}$, let
\[
\mathbb{G}_{r}= \bigl\{2^{r},2^{r}+1,\ldots,2^{r+1}-1 \bigr\},\qquad \mathbb{T}_{r}=\bigcup
_{q=0}^{r}\mathbb{G}_{q},
\]
which denote, respectively, the $r$th column and the first $(r+1)$
columns of the tree.
Then, the cardinality $|\mathbb{G}_{r}|$ of $\mathbb{G}_{r}$ is
$2^{r}$ and that of $\mathbb{T}_{r}$ is
$|\mathbb{T}_{r}|=2^{r+1}-1$. A column of a given integer $n$ is
$\mathbb{G}_{r_{n}}$ with $r_{n}=\lfloor\log_{2}n\rfloor$, where
$\lfloor x\rfloor$ denotes the integer part of the real number $x$.

The genealogy of the cells is described by this tree. In the sequel
we will thus see $\mathbb{T}$ as a given population. Then the vertex
$n$, the column $\mathbb{G}_{r}$ and the first $(r+1)$ columns
$\mathbb{T}_{r}$ designate, respectively, individual $n$, the $r$th
generation and the first $(r+1)$ generations. The initial individual
is denoted~$1$.

Guyon et al. \cite{GuAl,Guyon} proposed the following
linear Gaussian model to describe the evolution of the growth rate
of the population of cells derived from an initial individual:
%
\begin{equation}
\label{bar11} \mathcal{L}(X_{1})=\nu\quad \mbox{and}\quad \forall
n\geq1\qquad
\cases{\displaystyle X_{2n}=\alpha_{0}X_{n} +
\beta_{0} + \varepsilon_{2n},
\cr
\displaystyle
X_{2n+1} =
\alpha_{1}X_{n} + \beta_{1} +
\varepsilon_{2n+1},}
\end{equation}
where $X_{n}$ is the growth rate of individual $n$, $n$ is the
mother of $2n$ (the new pole progeny cell) and $2n+1$ (the old pole
progeny cell), $\nu$ is a distribution probability on $\mathbb{R}$,
$\alpha_{0}, \alpha_{1}\in(-1,1)$; $\beta_{0}, \beta_{1}\in
\mathbb{R}$ and $ ((\varepsilon_{2n}, \varepsilon_{2n+1}), n\geq
1 )$ forms a sequence of i.i.d. bivariate random variables with
law $\mathcal{N}_{2}(0,\Gamma)$, where
\[
\Gamma=\sigma^{2}\pmatrix{ 1 & \rho
\cr
\rho& 1},\qquad \sigma^{2}>0,\qquad
\rho\in(-1,1).
\]
The processes $(X_{n})$ defined by (\ref{bar11}) are typical
examples of BMC which are called the first-order bifurcating
autoregressive processes [BAR(1)]. The BAR(1) processes are an
adaptation of autoregressive processes, when the data have a binary
tree structure. They were first introduced by Cowan and Staudte
\cite{CS86} for cell lineage data where each individual in one
generation gives rise to two offspring in the next generation. We
will not discuss here extensions to $m$-ary tree, which follow more or
less from the same method, or Markov chains on Galton--Watson trees
that are left for an other study.

In \cite{Guyon}, Guyon, after establishing the first results on the
theory of BMC, proves laws of large
numbers and central\vspace*{2pt} limit theorem for the least-square estimators
$\hat{\theta}^{r}=(\hat{\alpha}_{0}^{r}, \hat{\beta}_{0}^{r},
\hat{\alpha}_{1}^{r}, \hat{\beta}_{1}^{r})$ of the 4-dimensional
parameter $\theta=(\alpha_{0},\beta_{0},\break \alpha_{1},\beta_{1})$; see
Section \ref{Application} for a more precise definition. He also
gives some statistical tests which allow to check if the model is
symmetric or not (roughly $\alpha_0=\alpha_1$ or not), and if the new
pole and the old pole populations
are even distinct in mean, which allows him to conclude a
statistical evidence in aging in E. Coli. Let us also mention
\cite{BerSapGeg}, where Bercu et al., using the martingale approach,
give asymptotic analysis of the least squares estimators of the
unknown parameters of a general asymmetric $p$th-order BAR
processes.

In this paper, we will give moderate deviation principle (MDP) for
this estimator and the statistical tests done by Guyon. We will also
give deviation inequalities for $\hat{\theta}^{r}-\theta$, which are
important for a rigorous (nonasymptotic) statistical study. This
will be done in two cases: the Gaussian case as described above and
the case where the noise and the initial state $X_{1}$ are assumed
to take values in a compact set. Note that the latter case implies
that the BAR(1) process defined by (\ref{bar11}) valued in compact
set.

We are now going to give a rigorous definition of BMC. We refer to
\cite{Guyon} for more detail.

\subsection{Definitions}

For an individual $n\in\mathbb{T}$, we are interested in the
quantity~$X_{n}$ (it may be the weight, the growth rate$,\ldots$) with
values in the metric space $S$ endowed with its Borel $\sigma$-field
$\mathcal{S}$.
%
\begin{definition}[($\mathbb{T}$-transition probability, see \cite{Guyon})]
We call $\mathbb{T}$-transition probability any mapping $P\dvtx  S\times
\mathcal{S}^{2}\rightarrow[0,1]$ such that:
\begin{itemize}
\item$P(\cdot,A)$ is measurable for all $A\in\mathcal{S}^{2}$;

\item$P(x,\cdot)$ is a probability measure on $(S^{2},\mathcal{S}^{2})$
for all $x\in S$.
\end{itemize}
\end{definition}
For a $\mathbb{T}$-transition probability $P$ on $S\times
\mathcal{S}^{2}$, we denote by $P_{0}$, $P_{1}$ and $Q$, respectively,
the first and the second marginal of $P$, and the mean of $P_{0}$
and $P_{1}$, that is, $P_{0}(x,B)=P(x,B\times S)$,
$P_{1}(x,B)=P(x,S\times B)$ for all\vspace*{1pt} $x\in S$ and $B\in\mathcal{S}$
and $ Q=\frac{P_{0}+P_{1}}{2}$.

For $p\geq1$, we denote by $\mathcal{B}(S^{p})$ [resp.,
$\mathcal{B}_{b}(S^{p})$], the set of all
$\mathcal{S}^{p}$-measurable (resp., $\mathcal{S}^{p}$-measurable and
bounded) mappings $f\dvtx
S^{p}\rightarrow\mathbb{R}$. For $f\in\mathcal{B}(S^{3})$, we
denote by $Pf \in\mathcal{B}(S)$ the function
\[
x\mapsto Pf(x)=\int_{S^{2}}f(x,y,z)P(x,dy,dz)\qquad\mbox{when it is
defined}.
\]

\begin{definition}[(Bifurcating Markov chains; see \cite{Guyon})]
Let $(X_{n}, n\in\mathbb{T})$ be a family of $S$-valued random
variables defined on a filtered probability space $(\Omega,
\mathcal{F}, (\mathcal{F}_{r}, r\in\mathbb{N}), \mathbb{P})$. Let
$\nu$ be a probability on $(S, \mathcal{S})$ and $P$ be a
$\mathbb{T}$-transition probability. We say that $(X_{n}, n\in
\mathbb{T})$ is a $(\mathcal{F}_{r})$-bifurcating Markov chain with
initial distribution $\nu$ and $\mathbb{T}$-transition probability
$P$ if:
\begin{itemize}
\item$X_{n}$ is $\mathcal{F}_{r_{n}}$-measurable for all $n\in
\mathbb{T}$;
\item$\mathcal{L}(X_{1})=\nu$;
\item for all $r\in\mathbb{N}$ and for all family $(f_{n}, n\in
\mathbb{G}_{r})\subseteq\mathcal{B}_{b}(S^{3})$
\[
\mathbb{E} \biggl[\prod_{n\in
\mathbb{G}_{r}}f_{n}(X_{n},X_{2n},X_{2n+1})
\Big/ \mathcal{F}_{r} \biggr]=\prod_{n\in
\mathbb{G}_{r}}Pf_{n}(X_{n}).
\]
\end{itemize}
\end{definition}

In the following, when unspecified, the filtration implicitly used
will be ${\mathcal F}_r=\sigma(X_i, i\in{\mathbb T}_r)$. We denote
by $(Y_{r},r\in\mathbb{N})$ the Markov chain on $S$ with
$Y_{0}=X_{1}$ and transition probability $Q$. The chain $(Y_{r},
r\in\mathbb{N})$ corresponds to a random lineage taken in the
population.

We denote by $\mathfrak{G}$ the set of all permutations of
$\mathbb{N}^{\ast}$ that leaves each $\mathbb{G}_{r}$ invariant. We
draw a permutation $\Pi$ uniformly on $\mathfrak{G}$, independently
of $X=(X_{n},n\in\mathbb{T})$. Drawing $\Pi$ ``uniformly'' on
$\mathfrak{G}$ means drawing the restriction of $\Pi$ on $\G_{r}$
uniformly among the $(2^{r})!$ permutations of $\G_{r}$. In
particular,
$ (\Pi(2^{r}),\Pi(2^{r}+1),\ldots,\Pi(2^{r+1}-1) )$ can be
viewed as a random drawing of all the elements of $\G_{r}$ without
replacement. Notice that $\Pi$ allows one to define a random order on
$\mathbb{T}$ which preserves the genealogical order. For example,
$ (\Pi(i), 1\leq i\leq n )$ denotes the set of the ``first''
$n$ individuals of $\T$. $\Pi$ was introduced by Guyon in order to
sample over the ``first'' $n$ individuals. As mentioned in
\cite{Guyon}, this choice of $\Pi$ allows one to preserve the same
asymptotic behavior for the empirical means resulting from the
sampling over (say) the $r$th generation, the first $(r+1)$
generations or the ``first'' $n$ individuals. In general, the choice
of another permutation does not preserve the asymptotic behavior of
these empirical means. We refer to \cite{Guyon}, Section 2.2, for
more detail.

Throughout the paper, we will denote by:
\begin{itemize}
\item$f\otimes g$ the mapping $(x,y)\mapsto f(x)g(y)$.
\item$Q^p$ the $p$th iterated of $Q$ recursively defined by the
formulas $Q^0(x,\cdot)=\delta_x$ and $Q^{p+1}(x,B)=\int
_{S}Q(s,dy)Q^p(y,B)$ for all $B\in\mathcal{S}$; $Q^p$ is a transition
probability in $(S,\mathcal{S})$.
\item$\nu Q$ the distribution on $(S,\mathcal{S})$ defined by $\nu Q
(B)=\int_{S}\nu(dx)Q(x,B)$; $\nu Q^p$ is the law of $Y_p$.
\item$(Qf)(x)=\int_{S}f(y)Q(x,dy)$ when it is defined.
\item$(\nu f)$ or $(\nu,f)$ the integral $\int_{S}f \,d\nu$ when it
is defined.
\end{itemize}
For all $i\in\mathbb{T}$, we set $\Delta_{i}=(X_{i},X_{2i},X_{2i+1})$.
We introduce the following empirical quantities:
%
\begin{equation}
\label{empiricalmean2} \cases{\displaystyle  \overline{M}_{\mathbb{G}_{r}}(f)=\frac
{1}{|\mathbb{G}_{r}|}
\sum_{i\in
\mathbb{G}_{r}}f(\tilde{\Delta}_{i}),
\cr
\displaystyle
\overline{M}_{\mathbb{T}_{r}}(f)=\frac{1}{|\mathbb{T}_{r}|}\sum
_{i\in
\mathbb{T}_{r}}f(\tilde{\Delta}_{i}),
\cr
\displaystyle
\overline{M}{}^{\Pi}_{n}(f)=\frac{1}{n}\sum
_{i=1}^{n}f(\tilde{\Delta}_{\Pi(i)}),}
\end{equation}
where $f(\tilde{\Delta}_{i})=f(\Delta
_{i})=f(X_{i},X_{2i},X_{2i+1}) $ if $f\in
\mathcal{B}(S^{3})$ and $f(\tilde{\Delta}_{i})=f(X_{i}) $ if
$f\in\mathcal{B}(S)$.

Guyon in \cite{Guyon} studied limit theorems of the empirical means
(\ref{empiricalmean2}), namely the law of large numbers ($L^2$ and
almost sure versions) and the central limit theorems for
(\ref{empiricalmean2}) when $f\in\mathcal{B}(S^{3})$, but centered
by the conditional expectation rather than by the limit mean. An
extension of the BMC has been proposed in \cite{DelMar}, in which
the authors studied a model of BMC with missing data. To take into
account the possibility for a cell to die, the authors of
\cite{DelMar} use Galton--Watson tree instead of a regular tree. And
they give a weak law of large numbers, an invariance principle and
the central limit result for the average over one generation or up
to one generation. As previously mentioned, this setting will be
considered in incoming works. One can also mention the work of De
Saporta et al. \cite{DesaporGegMar} dealing with bifurcating
autoregressive processes with missing data in the estimation
procedure of the parameters of the asymmetric BAR process. They use
a two type Galton--Watson process to model the genealogy and give
convergence and asymptotic normality of their estimators. It is
important to remark that the nonasymptotic study of deviation
inequalities has not been considered at all in these works, despite
their practical interest.


\subsection{Objectives}

Our objectives in this paper are:
\begin{itemize}
\item to give some limit theorems for BMC that
complete those done in \cite{Guyon} (LLN, LIL$,\ldots$);

\item to give probability inequalities and deviation inequalities for
the empirical means (\ref{empiricalmean2}), that is, for $f \in\BB
(S)$ and all $x>0$
\[
\P\bigl(\overline{M}_{\mathbb{T}_r}(f)-(\mu,f)\ge x \bigr)\le
e^{-C(x,r)},
\]
where $C(x,r)$ will crucially depend on our set of assumptions on
$f$ and on the ergodic property of $Q$ but valid for (nearly) all
$r$;

\item to study
moderate deviation principle (MDP) for BMC, that is, for some range of
speed $\sqrt{r}\ll b_r\ll r$ (depending on assumptions) and for
$f\in\BB_b(S^3)$ with $Pf=0$
\[
\frac{b^2_{|{\mathbb{T}_r}|}}{|\mathbb{T}_r|} \log\P\biggl(\frac
1{b_{|\mathbb{T}_r|}}M_{\mathbb{T}_r}(f)
\ge x \biggr)\sim- \frac{x^2}{2\sigma^2};
\]

\item to obtain the MDP and deviation inequalities for the estimator of
bifurcating autoregressive
process, which are important for a rigorous statistical study.
\end{itemize}
All these results will be obtained under hypothesis of geometric
ergodicity or
uniform geometric ergodicity, meaning that $Q^r$ converges (uniformly)
exponentially fast to a limiting measure.

The limit theorems, proved in this paper, include strong law of large
numbers for the empirical average $\overline{M}{}^{\Pi}_{n}(f)$ with
$f\in\mathcal{B}(S)$ (this case is not studied in \cite{Guyon}),
the law of the iterated logarithm and the almost sure functional
central limit theorem. A strong law of large numbers will be obtained via
control of 4th order moments. We thus generalize the computation of
2nd order moments made by Guyon in \cite{Guyon}. It will be noted
that the technique we will use can be applied to compute the other
higher-order moments, but at the price of huge and tedious
computations.

Deviation inequalities will be obtained in the setting of unbounded
functions, by using the classical Markov inequality and under
geometric ergodicity assumption. The results are, however, at this
point quite restrictive.

Exponential deviation inequalities will be shown for bounded functions
and under a uniform geometric ergodicity assumption. Their proof
intensively uses the Azuma--Bennett--Hoeffding inequality
\cite{Azuma,Bennett,Hoeffding}, which requires bounded random
variables. Extension to unbounded functions and weaker ergodicity
assumptions will be done in a further work, using transportation
inequalities in the spirit of~\cite{DGW04}.

The MDP will be mainly deduced from these inequalities and general
results on moderate deviations of martingales; see \cite{Djellout},
recalled in the Appendix B. Their speed will depend on whether uniform
geometric ergodicity or only geometric ergodicity is satisfied.

Before presenting the plan of our paper, let us recall the
definition of a moderate deviation principle (MDP): let $(b_n)_{n\ge
0}$ be a positive sequence such that
\[
\frac{b_n}{n}\underset{n\rightarrow\infty} {\longrightarrow} 0
\quad\mbox{and}\quad
\frac{b_n^2}{n}\underset{n\rightarrow\infty} {\longrightarrow}\infty.
\]
We say that a sequence of centered random variables $(M_{n})_{n}$ with
topological state space $(S,{\mathcal S})$ satisfies a MDP
with speed $b_n^2/n$ and rate function $I\dvtx
S\rightarrow\mathbb{R}_{+}^{*}$ if for each $A\in{\mathcal S}$,
\begin{eqnarray*}
-\inf_{x\in A^{o}}I(x)&\leq&\liminf_{n\rightarrow\infty}
\frac{n}{b^2_{n}}\log\mathbb{P} \biggl(\frac{n}{b_n}M_n\in A
\biggr)\leq\limsup_{n\rightarrow\infty}\frac{n}{b^2_{n}}\log\mathbb
{P}
\biggl(\frac{n}{b_n}M_n\in A \biggr)\\
&\leq&-\inf
_{x\in\overline{A}}I(x);
\end{eqnarray*}
here $A^{o}$ and $\overline{A}$ denote the interior and closure of $A$,
respectively.

The MDP can thus be seen as an intermediate behavior between the
central limit theorem ($b_n=b\sqrt n$) and large
deviation ($b_n=b n$). Usually, the MDP exhibits a simpler rate
function inherited from
the approximated Gaussian process, and holds for a larger class of
dependent random variables
than the large deviation principle.

Our paper is organized as follows. Section \ref{momentscontrol}
states the moments control inequalities and their consequences. We
shall state in this section a first result on the MDP for BMC in a
general framework, but with a very restricted range of speed.
Section~\ref{expoprobaineq} deals with the exponential inequalities
and their consequences. In this section, we shall generalize the MDP
done in Section \ref{momentscontrol}, allowing for a larger range of
speed, but under more stringent assumptions. In Section
\ref{Application}, we will focus particularly on the first order
bifurcating autoregressive processes. The proofs of some
inequalities are technical so postponed in Appendix \ref{appendixA}.
Appendix \ref{appendixB} is devoted to definitions and limit
theorems for martingales used intensively in the paper, and are included
here for completeness.

\section{Moments control and consequences}\label{momentscontrol}

Let $F$ be a vector subspace of $\mathcal{B}(S)$ such that:
\begin{longlist}[(iii)]
\item[(i)] $F$ contains the constants;
\item[(ii)] $F^{2}\subset F$;
\item[(iii)] $F\otimes F \subset L^{1}(P(x,\cdot))$ for all $x\in S$,
and $P(F\otimes F)\subset F$;
\item[(iv)] there exists a probability $\mu$ on $(S,\mathcal{S})$
such that $F\subset L^{1}(\mu)$ and
\[
\lim_{r\rightarrow
\infty}\mathbb{E}_{x} \bigl[f(Y_{r}) \bigr]=(\mu,f)
\]
for all $x\in S$ and $f\in F$;
\item[(v)] for all $f\in F$, there exists $g\in F$ such that for
all $r\in\mathbb{N}$, $|Q^{r}f|\leq g$;
\item[(vi)] $F\subset L^{1}(\nu)$,
\end{longlist}
where we have used the notation $F^2=\{f^2 /f\in F\}$, $F\otimes
F=\{f\otimes g/f,g\in F\}$ and $PE=\{Pf/f\in E\}$ whenever an
operator $P$ acts on a set $E$.

The following hypothesis is about the geometric ergodicity of $Q$:

\begin{longlist}[(H1)]
\item[(H1)]
Assume that for all $f\in F$ such that $(\mu,f)=0$,
there exists $g\in F$ such that for all $r\in\mathbb{N}$ and for
all $x\in S$, $|Q^{r}f(x)|\leq\alpha^{r}g(x)$ for some
$\alpha\in(0,1)$; that is, the Markov chain $(Y_{r},r\in
\mathbb{N})$ is geometrically ergodic.
\end{longlist}

Recall that under this hypothesis, Guyon \cite{Guyon} has shown the
weak law of large numbers for the three empirical average
$\overline{M}_{\mathbb{G}_{r}}(f)$,
$\overline{M}_{\mathbb{T}_{r}}(f)$ and $\overline{M}{}^{\Pi}_{n}(f)$
(see \cite{Guyon}, Theorem 11 when $f\in F$ and Theorem 12 when
$f\in\mathcal{B}(S^3)$) and the strong law of large numbers only
for $\overline{M}_{\mathbb{G}_{r}}(f)$,
$\overline{M}_{\mathbb{T}_{r}}(f)$; see \cite{Guyon}, Theorem 14
and Corollary 15 when $f\in F$ and Theorem 18 when $f\in
\mathcal{B}(S^3)$.

When $f\in\mathcal{B}(S^3)$ and under the additional hypothesis
$Pf^{2}$ and $Pf^{4}$ exist and belong to $F$, he proved the central
limit theorem for $\overline{M}_{\mathbb{T}_{r}}(f)$ and
$\overline{M}{}^{\Pi}_{n}(f)$; see~\cite{Guyon}, Theorem 19 and
Corollary 21. Recall that the central limit theorem for the three
empirical means (\ref{empiricalmean2}) when $f\in\mathcal{B}(S)$ is
still an open question; see \cite{DelMar} for more precision.

In this section, we complete these results by showing the strong law
of large numbers for $\overline{M}{}^{\Pi}_{n}(f)$, when $f\in F$. We
prove\vspace*{1pt} also the law of the iterated logarithm (LIL) and almost sure
functional central limit theorem (ASFCLT) for
$\overline{M}{}^{\Pi}_{n}(f)$ when $f\in\mathcal{B}(S^3)$.


\subsection{Control of the 4th order moments}

In order to establish limit theorems below, let us state the
following:

\begin{theorem}\label{4ordermomentcontrol}
Let $F$ satisfy \textup{(i)--(vi)}. Let $f\in F$ such that $(\mu, f)=0$. We assume
hypothesis \textup{(H1)}. Then for all $r\in\N$,
%
\begin{equation}
\label{4ordercontrol} \mathbb{E} \bigl[ \bigl(\overline{M}_{\mathbb
{G}_{r}}(f)
\bigr)^{4} \bigr]\leq\cases{c \bigl(\frac{1}{4}
\bigr)^{r}, &\quad  if $\alpha^{2}<\frac{1}{2}$,
\vspace*{2pt}\cr
cr^{2} \bigl(\frac{1}{4} \bigr)^{r}, &\quad  if $
\alpha^{2}=\frac{1}{2}$,
\vspace*{2pt}\cr
\displaystyle
c\alpha^{4r}, &\quad  if $
\alpha^{2}>\frac{1}{2}$,}
\end{equation}
where the positive constant $c$ depends on $\alpha$ and $f$ (and may
differ line by line).
\end{theorem}

\begin{pf}
First note that $f(X_{i})\in L^{4}$ for all $i\in\mathbb{G}_{r}$.
Indeed, let $(z_{1},\ldots,z_{r})\in\{0,1\}^{r}$ the unique path in
the binary tree from the root 1 to $i$. Then,
\[
\mathbb{E} \bigl[f^{4}(X_{i}) \bigr]=\nu
P_{z_{1}}\cdots P_{z_{r}}f^{4},
\]
and from hypotheses (ii), (iii) and (vi) we conclude that $\nu
P_{z_{1}}\cdots P_{z_{r}}f^{4}<\infty$.

Now, the proof divides into two parts.\vspace*{9pt}

\textit{Part} 1. \textit{Computation of
$\mathbb{E} [ (\overline{M}_{\mathbb{G}_{r}}(f)
)^{4} ]$}.
Independently of $X$, let us draw four independent indices $I_{r}$,
$J_{r}$, $K_{r}$ and $L_{r}$ uniformly from $\mathbb{G}_{r}$. Then
\[
\mathbb{E} \bigl[ \bigl(\overline{M}_{\mathbb{G}_{r}}(f) \bigr)^{4}
\bigr] = \mathbb{E} \bigl[f(X_{I_{r}})f(X_{J_{r}})f(X_{K_{r}})f(X_{L_{r}})
\bigr].
\]
For all $p\in\{0,\ldots, r \}$, let us define the following
events:
\begin{itemize}
\item$E_{0}^{p}$: The ancestors of $I_{r}$, $J_{r}$, $K_{r}$ and
$L_{r}$ are different in $\mathbb{G}_{p}$.\vspace*{1.2pt}

\item$E_{1}^{p}$: Exactly two of $I_{r}$, $J_{r}$, $K_{r}$ and
$L_{r}$ have the same ancestor in $\mathbb{G}_{p}$.\vspace*{1.2pt}

\item$E_{2}^{p}$: $I_{r}$, $J_{r}$, $K_{r}$ and $L_{r}$ have the same ancestor
two by two in $\mathbb{G}_{p}$.\vspace*{1.2pt}

\item$E_{3}^{p}$: Exactly three of $I_{r}$, $J_{r}$, $K_{r}$ and
$L_{r}$ have the same ancestor in $\mathbb{G}_{p}$.\vspace*{1.2pt}

\item$E_{4}^{p}$: $I_{r}$, $J_{r}$, $K_{r}$ and
$L_{r}$ have the same ancestor in $\mathbb{G}_{p}$.
\end{itemize}
We also consider the following events whose for each fixed $p\leq
r$, probability depend only on $p$.
\begin{itemize}
\item$E_{0}^{\prime p}$: Draw uniformly four independent indices from
$\mathbb{G}_{p}$ which are different.\vspace*{1.2pt}

\item$E_{1}^{\prime p}$: Draw uniformly four independent indices from
$\mathbb{G}_{p}$ such that two are the
same, and the others are different.\vspace*{1.2pt}

\item$E_{2}^{\prime p}$: Draw uniformly four independent indices from
$\mathbb{G}_{p}$ which are the same,
two by two.\vspace*{1.2pt}

\item$E_{3}^{\prime p}$: Draw uniformly four independent indices from
$\mathbb{G}_{p}$ such that exactly three
are the same.\vspace*{1.2pt}

\item$E_{4}^{\prime p}$: Draw uniformly four independent indices from
$\mathbb{G}_{p}$ which are all the same.
\end{itemize}
In the sequel we do the convention that $E_{0}^{r+1}$ is a certain
event. Then after successive conditioning by events $E_{i}^{p}$ for
$p\in\{0,\ldots,r \}$ and $i\in\{0,\ldots,4 \}$, we
have
%
\begin{eqnarray}\label{eqin}
&&
\mathbb{E} \bigl[f(X_{I_{r}})f(X_{J_{r}})f(X_{K_{r}})f(X_{L_{r}})
\bigr] \nonumber\\
&&\qquad=
\mathbb{E} \bigl[f(X_{I_{r}})f(X_{J_{r}})f(X_{K_{r}})f(X_{L_{r}})
/ E_{0}^{2} \bigr]\times\mathbb{P} \bigl(E_{0}^{2}
\bigr)\nonumber
\\
&&\qquad\quad{}+ \sum_{p=2}^{r} \mathbb{E}
\bigl[f(X_{I_{r}})f(X_{J_{r}})f(X_{K_{r}})f(X_{L_{r}})
/ E_{0}^{p+1},E_{1}^{p} \bigr]\times
\mathbb{P} \bigl(E_{1}^{p} \cap E_{0}^{p+1}
\bigr)\nonumber\\[-8pt]\\[-8pt]
&&\qquad\quad{}+ \sum_{p=2}^{r} \mathbb{E}
\bigl[f(X_{I_{r}})f(X_{J_{r}})f(X_{K_{r}})f(X_{L_{r}})
/ E_{0}^{p+1},E_{2}^{p} \bigr]\times
\mathbb{P} \bigl(E_{2}^{p} \cap E_{0}^{p+1}
\bigr)
\nonumber\\
&&\qquad\quad{}+ \mathbb{E} \bigl[f(X_{I_{r}})f(X_{J_{r}})f(X_{K_{r}})f(X_{L_{r}})
/ E_{3}^{r} \bigr]\times\mathbb{P} \bigl(E_{3}^{r}
\bigr)\nonumber
\\
&&\qquad\quad{}+ \mathbb{E} \bigl[f(X_{I_{r}})f(X_{J_{r}})f(X_{K_{r}})f(X_{L_{r}})
/ E_{4}^{r} \bigr]\times\mathbb{P} \bigl(E_{4}^{r}
\bigr).\nonumber
\end{eqnarray}
Let us notice that
\begin{itemize}
\item for all $i\in\{1,2,3,4\}$, $E_{i}^{r}$ and
$E_{i}^{\prime r}$ have the same probability;

\item the realization of ``$E_{1}^{p}\cap E_{0}^{p+1}$'' can be seen as
``draw uniformly four independent indices from
$\mathbb{G}_{p}$ such that two are the same and others are different,
and the two indices which are the same take different paths at
$\mathbb{G}_{p+1}$.'' Thus ``$E_{1}^{p}\cap E_{0}^{p+1}$'' has the same
probability that ``$E_{1}^{\prime p}\cap A_{p,p+1}$,'' where ``$A_{p,p+1}$''
is the event, ``the indices which are the same in $\G_{p}$ take
different paths at $\G_{p+1}$'';

\item similarly, the realization of ``$E_{2}^{p}\cap
E_{0}^{p+1}$'' may be interpreted as, ``draw uniformly four independent
indices from $\mathbb{G}_{p}$ which are the same two by two, and all
the indices take different paths at $\mathbb{G}_{p+1}$.'' Thus
``$E_{2}^{p}\cap E_{0}^{p+1}$'' has the same probability that
``$E_{2}^{\prime p}\cap A_{p,p+1}$,'' where ``$A_{p,p+1}$'' is the event, ``the
indices which are the same in $\G_{p}$ take different paths at
$\G_{p+1}$'';

\item for all $p\in\{0,\ldots,r\}$, we have
\begin{eqnarray*}
\P\bigl(E_{1}^{\prime p}\bigr) &=& \frac{6(2^{p}-1)(2^{p}-2)}{2^{3p}},\qquad
\P\bigl(E_{2}^{\prime p}\bigr) = \frac{3(2^{p}-1)}{2^{3p}},\\
\P\bigl(E_{3}^{\prime p}\bigr) &=& \frac{4(2^{p}-1)}{2^{3p}},\qquad
\P\bigl(E_{4}^{\prime p}\bigr) = \frac{1}{2^{3p}}.
\end{eqnarray*}
\end{itemize}
We may then deduce that
\[
\mathbb{P} \bigl(E_{0}^{2} \bigr)=\frac{3}{32},\qquad
\mathbb{P} \bigl(E_{3}^{r} \bigr) = \frac{4(2^{r}-1)}{2^{3r}},\qquad
\mathbb{P} \bigl(E_{4}^{r} \bigr) = \frac{1}{2^{3r}}
\]
and for $p\in\{2,\ldots,r-1 \}$,
\[
\P\bigl(E_{1}^{p} \cap E_{0}^{p+1}
\bigr) = \mathbb{P} \bigl(E_{1}^{\prime p} \bigr)\mathbb{P}
\bigl(A_{p,p+1} / E_{1}^{\prime p} \bigr) =
\frac{3(2^{p}-1)(2^{p}-2)}{2^{3p}}
\]
and
\[
\P\bigl(E_{2}^{p} \cap E_{0}^{p+1}
\bigr) = \mathbb{P} \bigl(E_{2}^{\prime p} \bigr)\mathbb{P}
\bigl(A_{p,p+1} / E_{2}^{\prime p} \bigr) =
\frac{3}{4}\frac{2^{p}-1}{2^{3p}}.
\]
We are now going to compute each term which appears in (\ref{eqin}).
We have the following convention: $P(Q^{-1}f\otimes Q^{-1}f)=f^{2}$. In
the sequel, we will use intensively, with a slight modification, the
calculations made by Guyon \cite{Guyon} in order to compute
conditional expectations related to the event, ``draw uniformly two
independent indices from $\G_{p}$,'' for $p\in\{0,\ldots,r\}$.

(a) We have that
\[
\mathbb{E} \bigl[f(X_{I_{r}})f(X_{J_{r}})f(X_{K_{r}})f(X_{L_{r}})/
E_{4}^{r} \bigr] = \nu Q^{r}f^{4}.
\]

(b) Conditionally on $E_{3}^{r}$, we may assume that the
indices $I_{r}$, $K_{r}$ and $L_{r}$ are the same. We then have,
using the calculations made by Guyon \cite{Guyon},
\begin{eqnarray*}
&&
\mathbb{E} \bigl[f(X_{I_{r}})f(X_{J_{r}})f(X_{K_{r}})f(X_{L_{r}})
/ E_{3}^{r} \bigr] \\
&&\qquad= \mathbb{E} \bigl[f^{3}(X_{I_{r}})f(X_{J_{r}})
/ E_{3}^{r} \bigr]
\\
&&\qquad= \frac{2^{r}}{2^{r}-1} \Biggl\{\sum_{p=0}^{r-1}2^{-p-2}
\nu Q^{p}P \bigl(Q^{r-p-1}f^{3}\otimes
Q^{r-p-1}f \\
&&\qquad\quad\hspace*{116pt}{}+ Q^{r-p-1}f\otimes Q^{r-p-1}f^{3}
\bigr) \Biggr\}.
\end{eqnarray*}

(c) Let $p\in\{2,\ldots,r \}$. Conditionally on
$E_{2}^{p}$ and $E_{0}^{p+1}$ we may assume that $I_{r}$ and $J_{r}$
have the same ancestor at $\mathbb{G}_{p}$, and $K_{r}$ and $L_{r}$
have the same ancestor at~$\mathbb{G}_{p}$. For simplification, we
will use the following notation:
%
\begin{equation}
\label{notQ}Q_{\otimes}^kf:=Q^kf\otimes
Q^kf,
\end{equation}
and we thus have
\begin{eqnarray*}
&& \mathbb{E} \bigl[f(X_{I_{r}})f(X_{J_{r}})f(X_{K_{r}})f(X_{L_{r}})
/ E_{0}^{p+1},E_{2}^{p} \bigr]
\\
&&\qquad= \mathbb{E} \bigl[\mathbb{E} \bigl[\mathbb{E} \bigl
[f(X_{I_{r}})f(X_{J_{r}})f(X_{K_{r}})f(X_{L_{r}})
/ \mathcal{F}_{p+1} \bigr] / \mathcal{F}_{p} \bigr] /
E_{0}^{p+1},E_{2}^{p} \bigr]
\\
&&\qquad= \mathbb{E} \bigl[P \bigl(Q_{\otimes}^{r-p-1}f \bigr)
(X_{I_{r}\wedge_{p} J_{r}})P \bigl(Q_{\otimes}^{r-p-1}f \bigr)
(X_{K_{r}\wedge_{p} L_{r}}) /E_{0}^{p+1}, E_{2}^{p}
\bigr]
\\
&&\qquad = \frac{2^{p}}{2^{p}-1} \sum_{l=0}^{p-1}2^{-l-1}
\nu Q^{l}P \bigl( \bigl(Q^{p-l-1}P \bigl(Q_{\otimes}^{r-p-1}f
\bigr) \bigr)\\
&&\qquad\quad\hspace*{108pt}{}\otimes\bigl(Q^{p-l-1}P \bigl(Q_{\otimes
}^{r-p-1}f
\bigr) \bigr) \bigr),
\end{eqnarray*}
where $I_{r}\wedge_{p} J_{r}$ (resp., $K_{r}\wedge_{p} L_{r}$) denotes
the common ancestor of $I_{r}$ and $J_{r}$ which is in
$\mathbb{G}_{p}$ (resp., the common ancestor of $K_{r}$ and $L_{r}$
which is in $\mathbb{G}_{p}$).

(d) Let $p\in\{2,\ldots,r \}$. Now conditionally on
$E_{1}^{p}$ and $E_{0}^{p+1}$ we may assume that it is $K_{r}$ and
$L_{r}$ which have the same ancestor in $\mathbb{G}_{p}$. We denote
by $p(I_{r})$ and $p(J_{r})$, respectively, the ancestor of $I_{r}$
and $J_{r}$ which are in $\mathbb{G}_{p}$. As before, the common
ancestor of $K_{r}$ and $L_{r}$, which are in $\mathbb{G}_{p}$, is
denoted by $K_{r}\wedge_{p} L_{r}$. At this step, we may repeat the
successive conditioning that we have done in the beginning but this
time for indices $p(I_{r})$, $p(J_{r})$ and $K_{r}\wedge_{p} L_{r}$.
This leads us to
\begin{eqnarray*}
\hspace*{-4pt}&&
{\fontsize{10.7pt}{12pt}\selectfont{\mbox{$\displaystyle
\mathbb{E} \bigl[f(X_{I_{r}})f(X_{J_{r}})f(X_{K_{r}})f(X_{L_{r}})
/E_{0}^{p+1},E_{1}^{p} \bigr]$}}}
\\
\hspace*{-4pt}&&
{\fontsize{10.7pt}{12pt}\selectfont{\mbox{$\displaystyle \qquad= \mathbb{E} \bigl[Q^{r-p}f(X_{p(I_{r})})Q^{r-p}f(X_{p(J_{r})})P
\bigl(Q_{\otimes}^{r-p-1}f \bigr) (X_{K_{r}\wedge_{p}
L_{r}})
/E_{0}^{p+1},E_{1}^{p} \bigr]$}}}
\\
\hspace*{-4pt}&&
{\fontsize{10.7pt}{12pt}\selectfont{\mbox{$\displaystyle \qquad= \frac{2^{2p}}{(2^{p}-1)(2^{p}-2)}\sum_{l=1}^{p-1}
\frac{1}{2^{l+1}}\frac{1}{2}$}}}\\
\hspace*{-4pt}&&
{\fontsize{10.7pt}{12pt}\selectfont{\mbox{$\displaystyle \qquad\quad{}\times \sum_{m=0}^{l-1}2^{-m-1}
\bigl\{\nu Q^{m}P \bigl( \bigl(Q^{l-m-1}P
\bigl(Q_{\otimes}^{r-l-1}f \bigr) \bigr)
\otimes Q^{p-m-1}P
\bigl(Q_{\otimes}^{r-p-1}f \bigr) \bigr)$}}}\\
\hspace*{-4pt}&&
{\fontsize{10.7pt}{12pt}\selectfont{\mbox{$\displaystyle \qquad\quad\hspace*{0pt}{}+\nu Q^{m}P \bigl( \bigl(Q^{p-m-1}P \bigl(Q_{\otimes}^{r-p-1}f
\bigr) \bigr)
\otimes\bigl(Q^{l-m-1} P \bigl(Q_{\otimes}^{r-l-1}f
\bigr) \bigr) \bigr)$}}}\\
\hspace*{-4pt}&&
{\fontsize{10.7pt}{12pt}\selectfont{\mbox{$\displaystyle \qquad\quad\hspace*{0pt}{}+\nu Q^{m}P \bigl( \bigl(Q^{l-m-1}P \bigl(Q^{r-l-1}f
\otimes Q^{p-l-1}P \bigl(Q_{\otimes}^{r-p-1}f \bigr) \bigr)
\bigr)
\otimes\bigl(Q^{r-m-1}f \bigr) \bigr)$}}}\\
\hspace*{-4pt}&&
{\fontsize{10.7pt}{12pt}\selectfont{\mbox{$\displaystyle \qquad\quad\hspace*{0pt}{}+\nu Q^{m}P \bigl(Q^{r-m-1}f\otimes\bigl(Q^{l-m-1}P
\bigl(Q^{r-l-1}f\otimes Q^{p-l-1}P \bigl(Q_{\otimes}^{r-p-1}f
\bigr) \bigr) \bigr) \bigr)$}}}\\
\hspace*{-4pt}&&
{\fontsize{10.7pt}{12pt}\selectfont{\mbox{$\displaystyle \qquad\quad\hspace*{0pt}{}+\nu Q^{m}P \bigl( \bigl(Q^{l-m-1}P \bigl(Q^{p-l-1}P
\bigl(Q_{\otimes
}^{r-p-1}f \bigr)\otimes Q^{r-l-1}f \bigr)
\bigr)\otimes\bigl(Q^{r-m-1}f \bigr) \bigr)$}}}\\
\hspace*{-4pt}&&
{\fontsize{10.7pt}{12pt}\selectfont{\mbox{$\displaystyle \qquad\quad\hspace*{0pt}{}+\nu Q^{m}P \bigl( \bigl(Q^{r-m-1}f \bigr)\otimes
\bigl(Q^{l-m-1}P \bigl(Q^{p-l-1}P \bigl(Q_{\otimes}^{r-p-1}f
\bigr)\otimes Q^{r-l-1}f \bigr) \bigr) \bigr) \bigr\}.$}}} %
\end{eqnarray*}

(e) Finally,
\begin{eqnarray*}
&& \mathbb{E} \bigl[f(X_{I_{r}})f(X_{J_{r}})f(X_{K_{r}})f(X_{L_{r}})
/ E_{0}^{2} \bigr]
\\
&&\qquad= \mathbb{E} \bigl[\mathbb{E} \bigl[\mathbb{E} \bigl
[f(X_{I_{r}})f(X_{J_{r}})f(X_{K_{r}})f(X_{L_{r}})
/ \mathcal{F}_{2} \bigr] /\mathcal{F}_{1} \bigr]
/E_{0}^{2} \bigr]
\\
&&\qquad= \mathbb{E} \bigl[P \bigl(Q_{\otimes
}^{r-2}f \bigr)
(X_{2}) P \bigl(Q_{\otimes}^{r-2}f \bigr)
(X_{3}) /E_{0}^{2} \bigr]
\\
&&\qquad= \nu P \bigl(P \bigl(Q_{\otimes}^{r-2}f \bigr)\otimes P
\bigl(Q_{\otimes}^{r-2}f \bigr) \bigr).
\end{eqnarray*}
Gathering together all of these terms, each multiplied by their respective
probability, we obtain an explicit expression for
$\mathbb{E} [ (\overline{M}_{\mathbb{G}_{r}}(f)
)^{4} ]$.\vspace*{9pt}

\textit{Part} 2. \textit{Rate}.
We are now going to give some rates for the different terms that
appear in the expression of
$\mathbb{E} [ (\overline{M}_{\mathbb{G}_{r}}(f)
)^{4} ]$.

Throughout this part, we will use intensively the following to bound
quantities which appear in the expression of
$\mathbb{E} [ (\overline{M}_{\mathbb{G}_{r}}(f)
)^{4} ]$:
\begin{itemize}
\item Let $f\in F$ such that $(\mu, f) = 0$. Then from (i)--(vi) and
hypothesis (H1), there exists a positive constant $c$ such
that $\forall l,m,n \in\N$,
\[
\nu Q^{l}P \bigl(Q^{m}f \otimes Q^{n}f \bigr)
\leq\alpha^{m+n}\nu Q^{l}P (g\otimes g ) \leq c
\alpha^{m+n},
\]
where $g$ is given in hypothesis (H1).
\end{itemize}
In the sequel, $c$ denotes a positive constant which depends on $f$,
and $c_{1}$ denotes a positive constant which depends on $\alpha$.
The constants $c$ and $c_{1}$ may vary from one line to another and
from one expression to another.

(a)
For the first term appearing in (\ref{eqin}), we have
\[
\mathbb{E} \bigl[f(X_{I_{r}})f(X_{J_{r}})f(X_{K_{r}})f(X_{L_{r}})
/ E_{0}^{2} \bigr]\times\mathbb{P} \bigl(E_{0}^{2}
\bigr)\leq c_{1}c\alpha^{4r}.
\]

(b)
For the fifth term appearing in (\ref{eqin}), we have
\[
\mathbb{E} \bigl[f(X_{I_{r}})f(X_{J_{r}})f(X_{K_{r}})f(X_{L_{r}})
/ E_{4}^{r} \bigr]\times\mathbb{P} \bigl(E_{4}^{r}
\bigr)\leq c \bigl(\tfrac{1}{2} \bigr)^{3r},
\]
where, from (ii), (v) and (vi), $c$ is such that $\nu Q^{r}f^{4}<c$.

(c)
For the fourth term appearing in (\ref{eqin}), we have
\[
\mathbb{E} \bigl[f(X_{I_{r}})f(X_{J_{r}})f(X_{K_{r}})f(X_{L_{r}})
/ E_{3}^{r} \bigr]\times\mathbb{P} \bigl(E_{3}^{r}
\bigr)\leq cc_{1}\alpha^{r} \biggl(\frac{1}{4}
\biggr)^{r}\sum_{p=0}^{r-1}
\biggl(\frac{1}{2\alpha} \biggr)^{p},
\]
where, from (ii), (iii), (v) and (vi), $c$ is such that for all
$p,q\in\mathbb{N}$
\[
\max\bigl(\nu Q^{p}P \bigl(Q^{q}f^{3}\otimes g
\bigr), \nu Q^{p}P \bigl(g\otimes Q^{q}f^{3}
\bigr) \bigr)<c,
\]
and from hypothesis (H1), $g$ is such that for all
$p\in\{1,\ldots,r-1 \}$
%
\begin{equation}
\label{g} Q^{r-p-1}f\leq\alpha^{r-p-1}g.
\end{equation}
Now depending on the value of $\alpha$, we obtain that
\begin{eqnarray*}
&&
\mathbb{E} \bigl[f(X_{I_{r}})f(X_{J_{r}})f(X_{K_{r}})f(X_{L_{r}})
/ E_{3}^{r} \bigr]\times\mathbb{P} \bigl(E_{3}^{r}
\bigr)\\
&&\qquad\leq\cases{\displaystyle  c_{1}c \biggl( \biggl(\frac{\alpha}{4}
\biggr)^{r}+ \biggl(\frac
{1}{2^{3}} \biggr)^{r} \biggr),
&\quad  if $\alpha\neq\dfrac{1}{2}$,
\vspace*{2pt}\cr
\displaystyle
c_{1}cr \biggl(
\frac{1}{2^{3}} \biggr)^{r}, &\quad  if $\alpha=\dfrac{1}{2}$.}
\end{eqnarray*}

(d)
Let us denote the third term appearing in (\ref{eqin})
by
\[
A_r:=\sum_{p=2}^{r}
\mathbb{E} \bigl[f(X_{I_{r}})f(X_{J_{r}})f(X_{K_{r}})f(X_{L_{r}})
/E_{0}^{p+1},E_{2}^{p} \bigr]\times
\mathbb{P} \bigl(E_{2}^{p}\cap E_{0}^{p+1}
\bigr).
\]
So we have
\[
A_r \leq c_{1}c \Biggl( \biggl(\frac{1}{4}
\biggr)^{r} + \alpha^{4r}\sum_{p=2}^{r-1}
\biggl(\frac{1}{4\alpha^{4}} \biggr)^{p} \Biggr),
\]
where, from (ii), (iii), (v) and (vi), $c$ is such that for all $p\in\{
2,\ldots,r-1\}$, $q\in\{0,\ldots,r-1\}$,
$l\in\{0,\ldots,p-1\}$
\[
\max\bigl(\nu Q^{q}P \bigl(Q_{\otimes}^{r-q-1}f^{2}
\bigr),\nu Q^{l}P \bigl(Q_{\otimes}^{p-l-1}P(g\otimes g)
\bigr) \bigr)<c,
\]
and $g$ is defined as before (\ref{g}) and the notation $Q_{\otimes}$
is given in (\ref{notQ}).\vadjust{\goodbreak}

Now depending on the value of $\alpha$, we obtain that:
\begin{itemize}
\item if
$\alpha^{2}\neq\frac{1}{2}$, then $ A_r\leq
c_{1}c ( (\frac{1}{4} )^{r} +
\alpha^{4r} );$
\item if $\alpha^{2}=\frac{1}{2}$, then $ A_r \leq
c_{1}c (r-1 ) (\frac{1}{4} )^{r}$.
\end{itemize}

(e)
For the second term appearing in (\ref{eqin}), we have
when $p=r$:
\begin{itemize}
\item if $\alpha=\frac{1}{2}$, then
\[
\mathbb{E} \bigl[f(X_{I_{r}})f(X_{J_{r}})f(X_{K_{r}})f(X_{L_{r}})
/E_{1}^{r} \bigr]\times\mathbb{P} \bigl(E_{1}^{r}
\bigr)\leq c_{1}c \bigl(\tfrac{1}{4} \bigr)^{r};
\]

\item if $\alpha\neq\frac{1}{2}$:
\begin{itemize}
\item if $\alpha^{2}=\frac{1}{2}$, then
\[
\mathbb{E} \bigl[f(X_{I_{r}})f(X_{J_{r}})f(X_{K_{r}})f(X_{L_{r}})
/E_{1}^{r} \bigr]\times\mathbb{P} \bigl(E_{1}^{r}
\bigr)\leq c_{1}(r-1) \bigl(\tfrac{1}{4} \bigr)^{r};
\]

\item if $\alpha^{2}\neq\frac{1}{2}$, then
\begin{eqnarray*}
&&
\mathbb{E} \bigl[f(X_{I_{r}})f(X_{J_{r}})f(X_{K_{r}})f(X_{L_{r}})
/E_{1}^{r} \bigr]\times\mathbb{P} \bigl(E_{1}^{r}
\bigr)\\
&&\qquad\leq c_{1}c \biggl( \biggl(\frac{\alpha^{2}}{2}
\biggr)^{r} + \biggl(\frac{1}{4} \biggr)^{r} \biggr),
\end{eqnarray*}
\end{itemize}
\end{itemize}
where, from (ii), (iii), (v) and (vi), $c$ is such that for all $l\in\{
2,\ldots,r-1\}$, $q\in\{0,\ldots,l-1\}$
\begin{eqnarray*}
&&\max\bigl(\nu Q^{q}P \bigl(Q^{l-q-1}P (g\otimes g )\otimes
Q^{r-q-1}f^{2} \bigr),\\
&&\hspace*{7.3pt}\qquad \nu Q^{q}P
\bigl(Q^{l-q-1}P \bigl(g\otimes Q^{r-l-1}f^{2} \bigr)
\otimes g \bigr) \bigr)<c
\end{eqnarray*}
and $g$ is defined as before (\ref{g}).

(f)
For the second terms appearing in (\ref{eqin}), and
for the remaining term in the sum ($p\not= r$), let us denote by
\[
B_r:=\sum_{p=2}^{r-1}
\mathbb{E} \bigl[f(X_{I_{r}})f(X_{J_{r}})f(X_{K_{r}})f(X_{L_{r}})
/E_{0}^{p+1},E_{1}^{p} \bigr]\times
\mathbb{P} \bigl(E_{1}^{p}\cap E_{0}^{p+1}
\bigr).
\]
So we have:
\begin{itemize}
\item if $\alpha= \frac{1}{2}$, then $ B_r\leq
c_{1}c (\frac{1}{4} )^{r}$;\vspace*{1pt}

\item if $\alpha\neq\frac{1}{2}$:\vspace*{1pt}
\begin{itemize}
\item if $\alpha^{2}=\frac{1}{2}$, then $ B_r\leq
c_{1}cr^{2} (\frac{1}{4} )^{r}$;\vspace*{1pt}

\item if $\alpha^{2}\neq\frac{1}{2}$, then $ B_r\leq
c_{1}c (\alpha^{4r} +
(\frac{\alpha^{2}}{2} )^{r} +
(\frac{1}{4} )^{r} )$,
\end{itemize}
\end{itemize}
where $c$ is defined in the same way as before.

Now the results of the Theorem \ref{4ordermomentcontrol} follow from
(a)--(f) of
part 2.
\end{pf}

It leads us to an extension of Theorem \ref{4ordermomentcontrol} to
the two
empirical averages $\overline{M}_{\mathbb{T}_{r}}(f)$ and
$\overline{M}{}^{\Pi}_{n}(f)$.
%
\begin{corollary}\label{4ordermomentcontrol2}
Let $F$ satisfy \textup{(i)--(vi)}. Let $f\in F$ such that $(\mu, f)=0$. We
assume that hypothesis \textup{(H1)} is fulfilled. Then for all $r\in
\N$ and $n\in\N$,
%
\begin{equation}
\label{4ordercontrol2} \mathbb{E} \bigl[ \bigl(\overline{M}_{\mathbb
{T}_{r}}(f)
\bigr)^{4} \bigr]\leq\cases{c \bigl(\frac{1}{4}
\bigr)^{r+1}, &\quad  if $\alpha^{2}<\frac{1}{2}$,
\vspace*{2pt}\cr
cr^{2} \bigl(\frac{1}{4} \bigr)^{r+1}, &\quad  if $
\alpha^{2}=\frac{1}{2}$,
\vspace*{2pt}\cr
c\alpha^{4(r+1)}, &\quad  if $
\alpha^{2}>\frac{1}{2}$,}
\end{equation}
and
%
\begin{equation}
\label{4ordercontrol3} \mathbb{E} \bigl[ \bigl(\overline{M}{}^{\Pi}_{n}(f)
\bigr)^{4} \bigr]\leq\cases{c \bigl(\frac{1}{4}
\bigr)^{r_{n}+1}, &\quad  if $\alpha^{2}<\frac{1}{2}$,
\vspace*{2pt}\cr
cr_{n}^{2} \bigl(\frac{1}{4} \bigr)^{r_{n}+1},
&\quad  if $\alpha^{2}=\frac{1}{2}$,
\vspace*{2pt}\cr
c\alpha^{4(r_{n}+1)}, &\quad  if
$\alpha^{2}>\frac{1}{2}$,}
\end{equation}
where the positive constant $c$ depends on $\alpha$ and $f$ and may
differ line by line.
\end{corollary}

\begin{pf}
The proof follows the same steps as in the proof of parts 2
and 3 of Theorem \ref{someprobaineq1}, and uses the
results of the proof of Theorem \ref{thmslln} to get the control of
the 4th order moment in incomplete generation. See Sections
\ref{sllnonincompletetree} and \ref{proofsomeprobaineq1} for more
detail.
\end{pf}

\begin{remark}\label{4ordermomentcontrol3}
If $f\in\mathcal{B}(S^{3})$ is such that $Pf^{2}$ and
$Pf^{4}$ exist and belong to $F$, with $Pf=0$, then we have for all
$r\in\N$ and for some positive constant~$c$,
%
\begin{equation}
\label{4ordercontrol4} \mathbb{E} \bigl[ \bigl(\overline{M}_{\mathbb
{G}_{r}}(f)
\bigr)^{4} \bigr]\leq\frac{c}{|\mathbb{G}_{r}|^{2}}.
\end{equation}
Indeed, let $M_{\mathbb{G}_{r}}(f) = \sum_{i\in
\mathbb{G}_{r}}f(\Delta_{i})$. We have
\begin{eqnarray*}
\mathbb{E} \bigl[ \bigl(M_{\mathbb{G}_{r}}(f) \bigr)^{4} \bigr] &=&
\mathbb{E} \bigl[M_{\mathbb{G}_{r}}\bigl(f^{4}\bigr) \bigr]+6\mathbb{E}
\biggl[\sum_{i\neq
j\in\mathbb{G}_{r}}f^{2}(
\Delta_{i})f^{2}(\Delta_{j}) \biggr]\\
&&{}+4
\mathbb{E} \biggl[\sum_{i\neq
j\in\mathbb{G}_{r}}f^{3}(
\Delta_{i})f(\Delta_{j}) \biggr]
\\
&&{}+12\mathbb{E} \biggl[\sum_{i\neq j\neq
k\in\mathbb{G}_{r}}f^{2}(
\Delta_{i})f(\Delta_{j})f(\Delta_{k})
\biggr]\\
&&{}+24\mathbb{E} \biggl[\sum_{i\neq
j\neq k\neq
l\in\mathbb{G}_{r}}f(
\Delta_{i})f(\Delta_{j})f(\Delta_{k})f(\Delta
_{l}) \biggr]
\\
&=&\mathbb{E} \biggl[\sum_{i\in\mathbb
{G}_{r}}Pf^{4}(X_{i})
\biggr]+6\mathbb{E} \biggl[\sum_{i\neq
j\in\mathbb{G}_{r}}Pf^{2}(X_{i})Pf^{2}(X_{j})
\biggr],
\end{eqnarray*}
where the last equality was obtained after conditioning by
$\mathcal{F}_{r}$ and using the fact that $Pf=0$. Now, dividing by
$|\mathbb{G}_{r}|^{4}$ leads us to
\begin{eqnarray*}
\mathbb{E} \bigl[ \bigl(\overline{M}_{\mathbb{G}_{r}}(f) \bigr)^{4}
\bigr]&=&\frac{6}{|\mathbb{G}_{r}|^{2}}\mathbb{E} \biggl[\frac{1} {
|\mathbb{G}_{r}|^{2}}\sum
_{i\neq j\in
\mathbb{G}_{r}}Pf^{2}(X_{i})Pf^{2}(X_{j})
\biggr] \\
&&{}+ \frac{1}{|\mathbb{G}_{r}|^{3}}\mathbb{E} \biggl[\frac
{1}{|\mathbb
{G}_{r}|}\sum
_{i\in
\mathbb{G}_{r}}Pf^{4}(X_{i}) \biggr]
\\
&\leq& \frac{6}{|\mathbb{G}_{r}|^{2}}\mathbb{E} \bigl[ \bigl(\overline
{M}_{\mathbb{G}_{r}}
\bigl(Pf^{2}\bigr) \bigr)^{2} \bigr]\\
&&{} + \frac{1}{|\mathbb{G}_{r}|^{3}}
\mathbb{E} \bigl[\overline{M}_{\mathbb{G}_{r}}\bigl(Pf^{4}\bigr) \bigr],
\end{eqnarray*}
and (\ref{4ordercontrol4}) then follows from the control of
\[
\bigl(\mathbb{E} \bigl[ \bigl(\overline{M}_{\mathbb{G}_{r}}
\bigl(Pf^{2}\bigr) \bigr)^{2} \bigr] \bigr)_{r}\quad \mbox{and}\quad
\bigl(\mathbb{E} \bigl[\overline{M}_{\mathbb{G}_{r}}\bigl(Pf^{4}\bigr)
\bigr] \bigr)_{r};
\]
see \cite{Guyon}.
\end{remark}

\begin{remark}\label{remcontrolT}
From Remark \ref{4ordermomentcontrol3}, we deduce that if $f\in
\mathcal{B}(S^{3})$ is such that $Pf^{2}$ and $Pf^{4}$ exist and
belong to $F$, with $Pf=0$, then we have for all $r\in\N$ and for
some positive constant $c$,
%
\begin{equation}
\label{4ordercontrolT} \mathbb{E} \bigl[ \bigl(
\overline{M}_{\mathbb{T}_{r}}(f) \bigr)^{4} \bigr]\leq c
\bigl(\tfrac{1}{4} \bigr)^{r+1}.
\end{equation}
Indeed, from the equality
\[
\overline{M}_{\mathbb{T}_{r}}(f) = \sum_{q=0}^{r}
\frac{|\G_{q}|}{|\T_{r}|}\overline{M}_{\mathbb{G}_{q}}(f),
\]
we deduce that
\[
\E\bigl[ \bigl(\overline{M}_{\mathbb{T}_{r}}(f) \bigr)^{4} \bigr]
\leq\Biggl(\sum_{q=0}^{r}
\frac{|\G_{q}|}{|\T_{r}|} \bigl\llVert\overline{M}_{\mathbb
{G}_{q}}(f)\bigr\rrVert
_{4} \Biggr)^{4},
\]
where $\|\cdot\|_{4}$ stands for the $L^{4}$-norm. We then infer
from (\ref{4ordercontrol4}) that
\[
\E\bigl[ \bigl(\overline{M}_{\mathbb{T}_{r}}(f) \bigr)^{4} \bigr]
\leq c \Biggl(\sum_{q=0}^{r}
\frac{ (\sqrt{2} )^{q}}{2^{r+1}} \Biggr)^{4}
\]
for some positive constant $c$. (\ref{4ordercontrolT}) then follows
from the last inequality.
\end{remark}

\subsection{Strong law of large numbers on incomplete subtree}
\label{sllnonincompletetree}

We now turn to prove the strong law of large numbers for
$\overline{M}{}^{\Pi}_{n}(f)$, completing the work of Guyon
\cite{Guyon}, where the LLN was proved only for the two averages
$\overline{M}_{{\T}_r}(f)$ and
$\overline{M}_{\mathbb{G}_r}(f)$.\vadjust{\goodbreak}
%
\begin{theorem}\label{thmslln}
Let $F$ satisfy \textup{(i)--(vi)}. Let $f\in F$ such that $(\mu,f)=0$. We
assume that hypothesis \textup{(H1)} is fulfilled with $\alpha
\in(0,\frac{\sqrt[4]{8}}{2} )$. Then
$\overline{M}{}^{\Pi}_{n}(f)$ almost surely converges to 0 as $n$
goes to $\infty$.
\end{theorem}

\begin{pf}
From the decomposition
\[
\overline{M}{}^{\Pi}_{n}(f)=\sum_{q=0}^{r_{n}-1}
\frac
{2^{q}}{n}\overline{M}_{\mathbb{G}_{q}}(f) + \frac{1}{n}\sum
_{i=2^{r_{n}}}^{n}f(X_{\Pi(i)}),
\]
it is enough to check that
\[
\sum_{n=1}^{\infty}\mathbb{E} \Biggl[ \Biggl(
\frac{1}{n}\sum_{i=2^{r_{n}}}^{n}f
(X_{\Pi(i)} ) \Biggr)^{4} \Biggr]<\infty.
\]
Indeed, since $\overline{M}_{\mathbb{G}_{q}}(f)$ almost surely
converges to 0 (Corollary 15 in \cite{Guyon}),
we deduce that the first term on the right-hand side
of the previous decomposition almost surely converges to 0 (Lemma 13 in
\cite{Guyon}).
We have
%
\begin{eqnarray}
\label{cc1}
&&
\mathbb{E} \Biggl[ \Biggl(\frac{1}{n}\sum
_{i=2^{r_{n}}}^{n}f (X_{\Pi(i)} )
\Biggr)^{4} \Biggr]\nonumber\\
&&\qquad=\frac{1}{n^{4}}\mathbb{E} \Biggl[\sum
_{i=2^{r_{n}}}^{n}f^{4} (X_{\Pi(i)} )
\Biggr]
+ \frac{6}{n^{4}}\mathbb{E} \Biggl[\sum_{i,j=2^{r_{n}};i\neq
j}^{n}f^{2}
(X_{\Pi(i)} )f^{2} (X_{\Pi(j)} ) \Biggr]
\\
&&\qquad\quad{}+ \frac{4}{n^{4}}\mathbb{E} \Biggl[\sum_{i,j=2^{r_{n}};i\neq
j}^{n}f^{3}
(X_{\Pi(i)} )f (X_{\Pi(j)} ) \Biggr]
\nonumber\\
&&\qquad\quad{}+ \frac{12}{n^{4}}\mathbb{E} \Biggl[\sum_{i,j,k=2^{r_{n}};i\neq
j\neq k}^{n}f^{2}(X_{\Pi(i)})f(X_{\Pi(j)})f(X_{\Pi(k)})
\Biggr]
\nonumber
\\
&&\qquad\quad{}+ \frac{24}{n^{4}}\mathbb{E} \Biggl[\sum_{i,j,k,l=2^{r_{n}};
i\neq j\neq k\neq
l}^{n}f
(X_{\Pi(i)} )f (X_{\Pi(j)} )f (X_{\Pi(k)} )f
(X_{\Pi(l)} ) \Biggr].
\nonumber
\end{eqnarray}
We will control each term appearing in decomposition (\ref{cc1}). For the
first term on the right-hand side of (\ref{cc1}), using (ii), (v) and
(vi) we have for
some positive constant $c$,
\[
\mathbb{E} \Biggl[\sum_{i=2^{r_{n}}}^{n}f^{4}
(X_{\Pi
(i)} ) \Biggr] = \bigl(n-2^{r_{n}}+1\bigr)\nu
Q^{r_{n}}f^{4} \leq c\bigl(n-2^{r_{n}}+1\bigr),
\]
which implies that
%
\begin{equation}
\label{t1}\frac{1}{n^{4}}\mathbb{E} \Biggl[\sum_{i=2^{r_{n}}}^{n}f^{4}
(X_{\Pi(i)} ) \Biggr] = O \biggl(\frac{1}{n^{3}} \biggr).
\end{equation}
Recall the following: for $i,j,k$ and $l\in
\{2^{r_{n}},\ldots,n\}$:
\begin{itemize}
\item If $i\neq j$, then $r_{n}\geq1$. Independently on $(X,\Pi)$,
draw two independent
indices $I_{r_{n}}$ and $J_{r_{n}}$ uniformly from
$\mathbb{G}_{r_{n}}$. Then the law of $(\Pi(i),\Pi(j))$ is the
conditional law of $(I_{r_{n}},J_{r_{n}})$ given $\{I_{r_{n}}\neq
J_{r_{n}}\}$.
\item If $i\neq j \neq k$, then $r_{n}\geq2$. Independently on $(X,\Pi
)$, draw three independent
indices $I_{r_{n}}, J_{r_{n}}$ and $K_{r_{n}}$ uniformly from
$\mathbb{G}_{r_{n}}$. Then the law of $(\Pi(i),\Pi(j), \Pi(k))$ is
the conditional law of $(I_{r_{n}},J_{r_{n}},K_{r_{n}})$ given
$\{I_{r_{n}}\neq J_{r_{n}}\neq K_{r_{n}}\}$.
\item If $i\neq j \neq k \neq l$, then $r_{n}\geq2$. Independently on
$(X,\Pi)$, draw four independent
indices $I_{r_{n}}, J_{r_{n}}, K_{r_{n}}$ and $L_{r_{n}}$ uniformly
from $\mathbb{G}_{r_{n}}$. Then the law of $(\Pi(i),\Pi(j), \Pi(k)),
\Pi(l))$ is the conditional law of $(I_{r_{n}},J_{r_{n}},K_{r_{n}},
L_{r_{n}})$ given $\{I_{r_{n}}\neq J_{r_{n}}\neq K_{r_{n}}\neq
J_{r_{n}}\}$.
\end{itemize}
Now we have to control the second and third terms of (\ref{cc1}). We
have to check that
%
\begin{equation}
\label{t2}\frac{1}{n^{4}}\mathbb{E} \Biggl[\sum_{i,j=2^{r_{n}};i\neq
j}^{n}f^{2}
(X_{\Pi(i)} )f^{2} (X_{\Pi(j)} ) \Biggr]=O \biggl(
\frac{1}{n^{2}} \biggr)
\end{equation}
and
%
\begin{equation}
\label{t3}\frac{1}{n^{4}}\mathbb{E} \Biggl[\sum_{i,j=2^{r_{n}};i\neq
j}^{n}f^{3}
(X_{\Pi(i)} )f (X_{\Pi(j)} ) \Biggr]=o \biggl(\frac{1}{n^{2}}
\biggr).
\end{equation}
Indeed, from the previous reminder and (i)--(vi), we have for some
positive constant $c$,
\begin{eqnarray*}
&&
\mathbb{E} \Biggl[\sum_{i,j=2^{r_{n}};i\neq
j}^{n}f^{2}
(X_{\Pi(i)} )f^{2} (X_{\Pi(j)} ) \Biggr] \\
&&\qquad=
\frac{ (n-2^{r_{n}})(n-2^{r_{n}}+1)}{(1-2^{-r_{n}})} \\
&&\qquad\quad{}\times\sum_{p=0}^{r_{n}-1}
2^{-p-1}\nu Q^{p}P \bigl(Q_{\otimes}^{r_{n}-p-1}f^{2}
\bigr)
\\
&&\qquad\leq c\bigl(n-2^{r_{n}}\bigr) \bigl(n-2^{r_{n}}+1\bigr),
\end{eqnarray*}
which implies (\ref{t2}). In the same way and using in addition
hypothesis (H1), we obtain that
\begin{eqnarray*}
&& \mathbb{E} \Biggl[\sum_{i,j=2^{r_{n}};i\neq
j}^{n}f^{3}
(X_{\Pi(i)} )f (X_{\Pi(j)} ) \Biggr]
\\
&&\qquad=\frac{(n-2^{r_{n}})(n-2^{r_{n}}+1)}{(1-2^{-r_{n}})}\\
&&\qquad\quad{}\times \sum_{p =
0}^{r_{n} - 1}
2^{-p-2} \nu Q^{p}P\bigl(Q^{r_{n}-p-1}f^{3}
\otimes Q^{r_{n}-p-1}f
\\
&&\hspace*{93pt}\qquad\quad{} + Q^{r_{n}-p-1}f \otimes Q^{r_{n}-p-1}f^{3}\bigr)
\\
&&\qquad\leq\cases{\displaystyle  c 2^{-r_{n}}\bigl(n-2^{r_{n}}\bigr)
\bigl(n-2^{r_{n}}+1\bigr), &\quad  if $\alpha< \frac{1}{2}$,
\vspace*{2pt}\cr
\displaystyle
c
r_{n}2^{-r_{n}}\bigl(n-2^{r_{n}}\bigr)
\bigl(n-2^{r_{n}}+1\bigr), &\quad  if $\alpha= \frac{1}{2}$,
\vspace*{2pt}\cr
\displaystyle
c
\alpha^{r_{n}}\bigl(n-2^{r_{n}}\bigr) \bigl(n-2^{r_{n}}+1
\bigr), &\quad  if $\alpha> \frac{1}{2}$,}
\end{eqnarray*}
which implies (\ref{t3}).

Let us deal with the remaining term of (\ref{cc1}):
\begin{eqnarray*}
&&
\frac{1}{n^{4}}\mathbb{E} \Biggl[\sum_{i,j,k=2^{r_{n}};i\neq
j\neq k}^{n}f^{2}(X_{\Pi(i)})f(X_{\Pi(j)})f(X_{\Pi(k)})
\Biggr]\\
&&\qquad= \frac{(n-2^{r_{n}}-1)(n-2^{r_{n}})(n-2^{r_{n}}+1)}{\mathbb
{P}(I_{r_{n}}\neq
J_{r_{n}}\neq K_{r_{n}})\times n^{4}}
\\
&&\qquad\quad{}\times\mathbb{E} \bigl[f^{2} (X_{I_{r_{n}}} )f
(X_{J_{r_{n}}} )f (X_{K_{r_{n}}} ) \mathbf{1}_{\{I_{r_{n}}\neq
J_{r_{n}}\neq K_{r_{n}}\}} \bigr].
\end{eqnarray*}
Then, we get an explicit expression for the last expectation similar
to that obtained in part (d) of the calculus of
$\mathbb{E} [(\overline{M}_{\mathbb{G}_{r}}(f))^{4} ]$ with a
slight modification of the functions. Calculating the rate of this
expression, we obtain
\begin{eqnarray*}
& &\sum_{n=4}^{\infty}\frac{1}{n^{4}}
\mathbb{E} \Biggl[\sum_{i,j,k=2^{r_{n}};i\neq
j\neq k}^{n}f^{2}(X_{\Pi(i)})f(X_{\Pi(j)})f(X_{\Pi(k)})
\Biggr]
\\
&&\qquad\leq c\sum_{n=1}^{\infty}\frac{1}{n}
\alpha^{2r_{n}} +c\sum_{n=1}^{\infty} \sum
_{p=2}^{r_{n}-1}\sum
_{l=0}^{p-1}\frac{1}{n}\frac
{1}{2^{p}}
\frac{1}{2^{l+1}}\alpha^{2r_{n}-2p} \\
&&\qquad\quad{}+ c\sum_{n=1}^{\infty}
\sum_{p=2}^{r_{n}-1}\sum
_{l=0}^{p-1}\frac{1}{n}\frac{1}{2^{p}}
\frac{1}{2^{l+1}}\alpha^{2r_{n}-p-l}
\end{eqnarray*}
for some positive $c$. Now it is not
hard to see that the right-hand side is finite.

Finally, to check that the series of general term
\[
\frac{1}{n^{4}}\mathbb{E} \Biggl[\sum_{i,j,k,l=2^{r_{n}};
i\neq j\neq k\neq
l}^{n}f
(X_{\Pi(i)} )f (X_{\Pi(j)} )f (X_{\Pi(k)} )f
(X_{\Pi(l)} ) \Biggr]
\]
is finite, it is enough, according to the calculation of rates we
have done in part 2 of the proof of Theorem
\ref{4ordercontrol}, to check that
$\sum_{n=1}^{\infty}\alpha^{4r_{n}}<\infty$, which is the
case if $\alpha\in(0,\frac{\sqrt[4]{8}}{2} )$, and this completes
the proof of Theorem \ref{thmslln}.
\end{pf}

\begin{remark}
Note that this theorem can be improved, but the price to pay is
enormous computations related to the calculation of higher moments.
If $f$ is bounded, this result is true for every $\alpha\in(0,1)$, as
we will see in Section \ref{expoprobaineq}.
\end{remark}

\subsection{Law of the iterated logarithm (LIL)}\label{lil}

Using the LIL for martingales (see Theorem \ref{thmStoutLIL}
of Stout in Appendix \ref{appendixB}), we are going to prove a
LIL for the BMC. This will be done when $f$ depends on the
mother-daughters triangle $(\Delta_i)$. We use the notation
$M_{n}^{\Pi}(f)=\sum_{i=1}^{n}f(\Delta_{\Pi(i)})$ and
$M_{\mathbb{T}_{r}}(f) = \sum_{i\in\mathbb{T}_{r}
}f(\Delta_{i})$.
%
\begin{theorem}\label{thmLILforBMC}
Let $F$ satisfy \textup{(i)--(vi)}. Let
$f\in\mathcal{B} (\mathcal{S}^{3} )$ such that $Pf=0$,
$Pf^{2}$ and $Pf^{4}$ exist and belong to $F$. We assume that
hypothesis \textup{(H1)} is fulfilled. Then
\[
\limsup_{n\rightarrow\infty}\frac{M_{n}^{\Pi}(f)}{\sqrt{2\langle
M^{\Pi}(f)\rangle_{n}\log\log\langle M^{\Pi}(f)\rangle_{n}}}=1 \qquad\mbox{a.s.}
\]
And in particular,
\[
\limsup_{r\rightarrow\infty}\frac{M_{\mathbb{T}_{r}}(f)}{\sqrt
{2|\mathbb{T}_{r}|\log\log{|\mathbb{T}_{r}|}}} = \sqrt{\bigl(
\mu,Pf^{2}\bigr)} \qquad\mbox{a.s.}
\]
\end{theorem}

\begin{pf} We will check the hypothesis of Stout Theorem's \ref
{thmStoutLIL}.
Let $f\in\mathcal{B} (\mathcal{S}^{3} )$. We introduce the
filtration $(\mathcal{H}_{n})_{n\ge0}$ defined by
$\mathcal{H}_{0}=\sigma(X_{1})$ and $\mathcal{H}_{n} =
\sigma(\Delta_{\Pi(i)}, \Pi(i+1), 1\leq i\leq n )$. Let
$ (M^{\Pi}_{n}(f) )_{n\ge0}$ defined by $M^{\Pi}_{0}(f) = 0$
and $M^{\Pi}_{n}(f)=\sum_{i=1}^{n}f(\Delta_{\Pi(i)})$. Then
since $Pf=0$, $ (M^{\Pi}_{n}(f) )$ is a
$\mathcal{H}_{n}$-martingale with
$\mathbb{E} [M^{\Pi}_{1}(f) ]=0$. The bracket of the above
martingale is given by
\[
\bigl\langle M^{\Pi}(f)\bigr\rangle_{n} = \sum
_{i=0}^{n} Pf^{2}(X_{\Pi(i)}) =
M_{n}^{\Pi}\bigl(Pf^{2}\bigr).
\]
We have the following decomposition:
\[
\frac{\langle M^{\Pi}(f)\rangle_{n}}{n} = \overline{M}{}^{\Pi}_n
\bigl(Pf^2\bigr) = \sum_{q=0}^{r_{n}-1}
\frac{2^{q}}{n}\overline{M}_{\mathbb
{G}_{q}}\bigl(Pf^{2}\bigr) +
\frac{1}{n}\sum_{i=2^{r_{n}}}^{n}Pf^{2}(X_{\Pi(i)}).
\]
Since
\[
\forall q\leq r_{n}-1\qquad \frac{2^{q}}{2^{r_{n}+1}} \leq\frac{2^{q}}{n}
\leq\frac{2^{q}}{2^{r_{n}}} \quad\mbox{and}\quad \frac{1}{n}\sum
_{i=2^{r_{n}}}^{n}Pf^{2}(X_{\Pi(i)}) \leq
\overline{M}_{\G_{r_{n}}}\bigl(Pf^{2}\bigr),
\]
we deduce that
\[
\sum_{q=0}^{r_{n}-1}
\frac{2^{q}}{2^{r_{n}+1}} \overline{M}_{\G_{q}}\bigl(Pf^{2}\bigr) \leq
\overline{M}{}^{\Pi}_{n}\bigl(Pf^{2}\bigr) \leq\sum
_{q=0}^{r_{n}} \frac{2^{q}}{2^{r_{n}}}
\overline{M}_{\G_{q}}\bigl(Pf^{2}\bigr).
\]
From the strong law of large numbers of
$\overline{M}_{\G_{q}}(Pf^{2})$ (see \cite{Guyon}, Corollary 15) and
from Lemma 5.2 of \cite{DesaporGegMar}, we infer that
\[
\sum_{q=0}^{r_{n}-1}\frac{2^{q}}{2^{r_{n}+1}}
\overline{M}_{\G_{q}}\bigl(Pf^{2}\bigr) \stackrel{\mathrm{a.s.}}
{\longrightarrow} \frac{(\mu, Pf^{2})}{2} \quad\mbox{and}\quad \sum_{q=0}^{r_{n}}
\frac{2^{q}}{2^{r_{n}}} \overline{M}_{\G_{q}}\bigl(Pf^{2}\bigr)
\stackrel{\mathrm{a.s.}}{\longrightarrow} 2\bigl(\mu, Pf^{2}\bigr).
\]
Using these results, we thus deduce that $\langle
M^{\Pi}(f)\rangle_{n}=O(n)$ and
$n=\break O (\langle M^{\Pi}(f)\rangle_{n} )$ a.s.
This implies in particular that $\langle
M^{\Pi}(f)\rangle_{n}\underset{n\rightarrow \infty}{\longrightarrow}
\infty$ a.s.

Now let $K_{n}=\frac{\sqrt{2}}{\sqrt{\log\log(n)}}$ in Theorem
\ref{thmStoutLIL}, and we have
\begin{eqnarray*}
R&:=&\sum_{n=1}^{\infty}\frac{2\log\log\langle
M^{\Pi}(f)\rangle_{n}}{K_{n}^{2}\langle
M^{\Pi}(f)\rangle_{n}}\\
&&\hspace*{14.6pt}{}\times
\mathbb{E} \bigl[ f^{2}(\Delta_{\Pi(n)})\mathbf{1}_{ \{f^{2}(\Delta_{\Pi
(n)})>{K_{n}^{2}\langle
M^{\Pi}(f)\rangle_{n}}
/({2\log\log\langle M^{\Pi}(f)\rangle_{n}}) \}}/
\mathcal{H}_{n-1} \bigr]
\\
&\leq&\sum_{n=1}^{\infty}\frac{4(\log\log\langle
M^{\Pi}(f)\rangle_{n})^{2}}{K_{n}^{4}(\langle
M^{\Pi}(f)\rangle_{n})^{2}}Pf^{4}(X_{\Pi(n)})\qquad
\mbox{a.s.},
\end{eqnarray*}
since $\langle M^{\Pi}(f)\rangle_{n}=O(n)$ a.s., so that for
$R<\infty$ a.s., it is enough to check that
%
\begin{equation}
\label{eq2} \sum_{n=1}^{\infty}
\frac{Pf^{4}(X_{\Pi(n)})}{n^{\delta
}}<\infty\qquad\mbox{a.s. with any $1<\delta<2$.}
\end{equation}
Now, according to (v) and (vi), there exists a positive constant $c$
such that for all $n\geq1$, $\mathbb{E} [Pf^{4} (X_{\Pi(n)} ) ]=\nu
Q^{r_{n}}Pf^{4}\leq c$, and (\ref{eq2}) follows. Applying
Theorem~\ref{thmStoutLIL}, we have
\[
\limsup_{n\rightarrow\infty}\frac{M_{n}^{\Pi}(f)}{\sqrt{2\langle
M^{\Pi}(f)\rangle_{n}\log\log\langle M^{\Pi}(f)\rangle_{n}}}=1 \qquad\mbox{a.s.}
\]
Now, for $n=|\mathbb{T}_{r}|$, we have the following:
\begin{eqnarray*}
&&\frac{M_{\mathbb{T}_{r}}(f)}{\sqrt{2\langle M^{\Pi}(f)\rangle
_{|\mathbb{T}_{r}|}\log\log\langle M^{\Pi}(f)\rangle_{|\mathbb
{T}_{r}|}}} \\
&&\qquad=\sqrt{\frac{|\mathbb{T}_{r}|{\langle
M^{\Pi}(f)\rangle_{|\mathbb{T}_{r}|}}/{|\mathbb{T}_{r}|}}{2\log\log
\langle
M^{\Pi}(f)\rangle_{|\mathbb{T}_{r}|}}}\times\frac{M_{\mathbb
{T}_{r}}(f)}{|\mathbb{T}_{r}|{\langle
M^{\Pi}(f)\rangle_{|\mathbb{T}_{r}|}}/{|\mathbb{T}_{r}|}}
\end{eqnarray*}
and since $\frac{\langle
M^{\Pi}(f)\rangle_{|\mathbb{T}_{r}|}}{|\mathbb{T}_{r}|}=\overline
{M}_{\mathbb{T}_{r}}(Pf^2)\underset{r
\rightarrow\infty}{\longrightarrow} (\mu,Pf^{2})$ a.s. (see Theorem
18 in \cite{Guyon}), we get
\[
\limsup_{r\rightarrow
\infty}\frac{M_{\mathbb{T}_{r}}(f)}{\sqrt{2|\mathbb{T}_{r}|
\log\log|\mathbb{T}_{r}|}} = \sqrt{\bigl(
\mu,Pf^{2}\bigr)} \qquad\mbox{a.s.},
\]
which completes the proof.
\end{pf}

\begin{remark}
Let us note that using Theorem \ref{thmslln}, we can prove that if
hypothesis (H1) is fulfilled with $\alpha\in
(0,\frac{\sqrt[4]{8}}{2} )$, then
\[
\limsup_{n\rightarrow\infty}\frac{M_{n}^{\Pi}(f)}{\sqrt{2n\log
\log{n}}}= \sqrt{\bigl(
\mu,Pf^{2}\bigr)} \qquad\mbox{a.s.},
\]
and via the computation of $2k$th order moments of
$\overline{M}_{\mathbb{G}_{r}}(g)$, with $k>2$ and $g\in
\mathcal{B}(S )$, it is possible to prove the latter for all
$\alpha\in(0,1)$. But, as already emphasized, this comes at the price
of enormous computations.
\end{remark}

\subsection{Almost-sure functional central limit
theorem (ASFCLT)}\label{almost-surefunctionalclt}

We are now going to prove an ASFCLT theorem for the BMC $(X_{n},n\in
\mathbb{T})$. Here again, this will be done when $f$ depends on the
mother-daughters triangle by using the ASFCLT for discrete time
martingale. We refer to Chaabane, Theorem \ref{thmchaabanetlcpsf},
Appendix \ref{appendixB}, for the definition of an
ASFCLT.
%
\begin{theorem}\label{thmasfclt}
Let $F$ satisfy \textup{(i)--(vi)}. Let $f\in\mathcal{B}(\mathcal{S}^{3})$
such that $Pf=0$, $Pf^{2}$ and $Pf^{4}$ exist and belong to $F$. We
assume that hypothesis \textup{(H1)} is fulfilled with $\alpha\in
(0,\frac{\sqrt[4]{8}}{2} )$. Then $M_{n}^{\Pi}(f)$ verifies
an ASFCLT, when $n$ goes to~$\infty$.
\end{theorem}
\begin{pf} We use Theorem \ref{thmchaabanetlcpsf}.
Let $(\mathcal{H}_{n})_{n\in\N}$ be the filtration defined as in
Section \ref{lil}. Then $(M_{n}^{\Pi}(f))$ is a $\mathcal{H}_{n}$
martingale. We have to check the hypotheses of Theorem \ref{thmchaabanetlcpsf}. For all $n\geq1$, let $V_{n}=s\sqrt{n}$ where
$s^{2}=(\mu,Pf^{2})$. Then according to Theorem \ref{thmslln},
\[
\frac{\langle
M^{\Pi}(f)\rangle_{n}}{V_{n}^{2}}=V_{n}^{-2}M_n^{\Pi}
\bigl(Pf^2\bigr)\underset{n\rightarrow\infty} {\longrightarrow} 1\qquad
\mbox{a.s.}
\]
Let $\varepsilon>0$. We have
\begin{eqnarray*}
&&
\sum_{n\geq
1}\frac{1}{V_{n}^{2}}\mathbb{E}
\bigl[f^{2}(\Delta_{\Pi(n)})\mathbf{1}_{ \{
\llvert f (\Delta_{\Pi(n)} )\rrvert>\varepsilon
V_{n} \}} /
\mathcal{H}_{n-1} \bigr]\\
&&\qquad\leq\frac{1}{\varepsilon^{2}s^{4}}\sum
_{n\geq
1}\frac{Pf^{4} (X_{\Pi(n)} )}{n^{2}} \qquad\mbox{a.s.}
\end{eqnarray*}
According to (v) and (vi), there exists a positive constant
$c$ such that for all $n\geq1$,
$\mathbb{E} [Pf^{4} (X_{\Pi(n)} ) ]=\nu
Q^{r_{n}}Pf^{4}\leq c$, and therefore, $\forall\varepsilon>0$
\[
\sum_{n\geq
1}\frac{1}{V_{n}^{2}}\mathbb{E}
\bigl[f^{2}(\Delta_{\Pi(n)})\mathbf{1}_{ \{
\llvert f (\Delta_{\Pi(n)} )\rrvert>\varepsilon
V_{n} \}} /
\mathcal{H}_{n-1} \bigr]<\infty\qquad\mbox{a.s.}
\]
Finally, we have
\[
\sum_{n\geq
1}\frac{1}{V_{n}^{4}}\mathbb{E}
\bigl[f^{4} (\Delta_{\Pi
(n)} )\mathbf{1}_{ \{
\llvert f (\Delta_{\Pi(n)} )\rrvert\leq
V_{n} \}} /
\mathcal{H}_{n-1} \bigr]\leq\frac{1}{s^{4}}\sum
_{n\geq
1}\frac{Pf^{4} (X_{\Pi(n)} )}{n^{2}}\qquad \mbox{a.s.},
\]
which as before is a.s. finite, and the proof is then complete.\vadjust{\goodbreak}
\end{pf}

\begin{remark}
As before, let us note that this result can be extended to the general
case $\alpha\in(0,1)$, but at the price of enormous computation
related to the computation of $2k$-order moments, $k>2$, for
$\overline{M}_{\mathbb{G}_{r}}(g)$, $g\in\mathcal{B}(S)$.
\end{remark}

\subsection{Deviation inequalities for BMC}\label{someprobaineq4bmc}

We are now going to give some deviation inequalities under
(i)--(vi) and (H1) for the empirical means
(\ref{empiricalmean2}) when $f\in\mathcal{B}(S)$ with $(\mu,f)=0$
and when $f\in\mathcal{B}(S^3)$ with $(\mu, Pf)=0$. This will help
us in the sequel to obtain a MDP
result in a general framework, that is, for functional of BMC with
unbounded test functions. Let us recall that the main disadvantage of
this ``weak'' set of assumptions is that the
range of speed for the MDP is very restricted. However, we still work
under geometric ergodicity
assumption and general test function, which will not be the
case when we would want to extend the MDP; see Section
\ref{expoprobaineq}. Note that we postpone to Appendix \ref{appendixA} nearly all
the proofs of this section, these proofs being quite long and technical.

\begin{theorem}\label{someprobaineq1} Let $F$ satisfy conditions
\textup{(i)--(vi)}. We
assume that \textup{(H1)} is fulfilled. Let $f\in F$ such that
$(\mu, f)=0$. Then we have for all $\delta>0$ and all $r\in\N$ and
all $n\in\N$,
%
\begin{eqnarray}
\label{probaineq1} \mathbb{P} \bigl( \bigl|\overline{M}_{\mathbb{G}_{r}}(f)
\bigr|>\delta
\bigr)&\leq&\cases{\displaystyle \frac{c}{\delta^{2}} \biggl(\frac{1}{2}
\biggr)^{r}, &\quad  if $\displaystyle \alpha^{2}<\frac{1}{2}$;
\vspace*{2pt}\cr
\displaystyle \frac{c}{\delta^{2}}r \biggl(\frac{1}{2} \biggr)^{r}, &\quad  if $\displaystyle
\alpha^{2}=\frac{1}{2}$;
\vspace*{2pt}\cr
\displaystyle \frac{c}{\delta^{2}}
\alpha^{2r}, &\quad  if $\displaystyle \alpha^{2}>\frac{1}{2}$;}
\\
\label{probaineq2} \mathbb{P} \bigl( \bigl|\overline{M}{}^{\Pi}_{n}(f)
\bigr|>\delta\bigr)&\leq&\cases{\displaystyle \frac{c}{\delta^{2}} \biggl(\frac{1}{2}
\biggr)^{r_n+1}, &\quad  if $\displaystyle \alpha^{2}<\frac{1}{2}$;
\vspace*{2pt}\cr
\displaystyle \frac{c}{\delta^{2}} r_{n} \biggl(\frac{1}{2}
\biggr)^{r_{n}+1}, &\quad  if $\displaystyle \alpha^{2}=\frac{1}{2}$;
\vspace*{2pt}\cr
\displaystyle \frac{c}{\delta^{2}}\alpha^{2(r_{n}+1)}, &\quad  if $\displaystyle \alpha^{2}>
\frac{1}{2}$;}
\end{eqnarray}
and
%
\begin{equation}
\label{probaineq3} \mathbb{P} \bigl( \bigl|\overline{M}_{\mathbb{T}_{r}}(f)
\bigr|>\delta
\bigr)\leq\cases{\displaystyle \frac{c}{\delta^{2}} \biggl(\frac{1}{2}
\biggr)^{r+1}, &\quad  if $\displaystyle \alpha^{2}<\frac{1}{2}$;
\vspace*{2pt}\cr
\displaystyle \frac{c}{\delta^{2}}r \biggl(\frac{1}{2} \biggr)^{r+1}, &\quad  if $\displaystyle
\alpha^{2}=\frac{1}{2}$;
\vspace*{2pt}\cr
\displaystyle \frac{c}{\delta^{2}}
\alpha^{2(r+1)}, &\quad  if $\displaystyle \alpha^{2}>\frac{1}{2}$;}
\end{equation}
where the positive constant $c$ depends on $f$ and $\alpha$ and may
differ term by term.
\end{theorem}

\begin{pf} See Section \ref{proofsomeprobaineq1} in Appendix \ref{appendixA}.
\end{pf}

We shall also need an extension of Theorem \ref{someprobaineq1} to
the case when $f$ does not only depend on an individual $X_i$, but
on the mother-daughters triangle~$(\Delta_i)$.
%
\begin{theorem}\label{someprobaineq2} Let $F$ satisfy conditions
\textup{(i)--(vi)}. We assume that \textup{(H1)}
is fulfilled. Let $f\in\mathcal{B} (S^{3} )$ such that $Pf$
and $Pf^{2}$ exist and belong to $F$ and $(\mu, Pf)=0$. Then we have
the same conclusion as in Theorem \ref{someprobaineq1} for the
three empirical averages given in (\ref{empiricalmean2}):
$\overline{M}_{\mathbb{G}_{r}}(f)$,
$\overline{M}_{\mathbb{T}_{r}}(f)$ and $\overline{M}{}^{\Pi}_{n}(f)$.
\end{theorem}

\begin{pf} See Section \ref{proofsomeprobaineq2} in Appendix \ref{appendixA}.
\end{pf}

We thus have the following first result on the superexponential
convergence in probability, whose definition we present now:
%
\begin{definition}\label{exponentialdefini}
Let $(E,d)$ a metric space. Let $(Z_{n})$ be a sequence of random
variables valued in $E$, $Z$ be a random variable valued in $E$ and
$(v_{n})$ be a rate. We say that $Z_n$ converges
$v_n$-superexponentially fast in probability to $Z$ if for all
$\delta>0$,
\[
\limsup_{n\rightarrow\infty}\frac{1}{v_n}\log\P\bigl(d(Z_n,Z)>
\delta\bigr)=-\infty.
\]
This ``exponential convergence'' with speed $v_n$ will be shortened as
\[
Z_n \stackrel{\mathrm{superexp}} {\underset{v_n} {
\longrightarrow}} Z.
\]
\end{definition}
We may now set:

\begin{prop}\label{expoconv1}
Let $F$ satisfy conditions \textup{(i)--(vi)}. Let $f\in
\mathcal{B} (S^{3} )$ such that $Pf$ and $Pf^{2}$ exist and
belong to $F$ and $(\mu, Pf)=0$. We assume that \textup{(H1)} is
fulfilled. Let $(b_{n})$ be a sequence of increasing positive real
numbers such that
%
\begin{equation}
\label{vb}\quad\frac{b_{n}}{\sqrt{n}} \longrightarrow+\infty,\qquad \frac
{b_{n}}{\sqrt{n\log
n}}
\longrightarrow0,\qquad \frac{n}{b_{n}}\mbox{ is nondecreasing}.
\end{equation}
Then
\[
\overline{M}{}^{\Pi}_{n}(f)\stackrel{\mathrm{superexp}}{\underset{{b_n^2}/{n}}
{\longrightarrow}} 0.
\]
\end{prop}
\begin{pf} The proof is a direct consequence of Theorem \ref{someprobaineq2}.
\end{pf}


\subsection{Moderate deviations for BMC}

Now, using the MDP for martingale (see, e.g., \cite{Djellout,TheseWorms}), we are going to prove
a MDP for BMC. We will use Proposition~\ref{propmdp}, in Appendix
\ref{appendixB}.\vadjust{\goodbreak}
%
\begin{theorem}\label{thmmdp}
Let $F$ satisfy conditions \textup{(i)--(vi)}. We assume that \textup{(H1)} is
satisfied. Let $f\in\mathcal{B} (S^{3} )$ such that $Pf^{2}$
and $Pf^{4}$ exist and belong to $F$. Assume that $Pf=0$. Let
$(b_{n})$ be a sequence of increasing positive real numbers
satisfying (\ref{vb}). If
%
\begin{equation}
\label{cond-led}\qquad \limsup_{n\rightarrow
\infty}\frac{n}{b_{n}^{2}}\log\bigl(n\underset{1\leq k\leq c^{-1}(b_{n+1})}
{\operatorname{ess}\operatorname{sup}}\mathbb{P}
\bigl( \bigl|f (\Delta_{\Pi(k)} ) \bigr|>b_{n} /{\mathcal{H}_{k-1}}
\bigr) \bigr) = -\infty,
\end{equation}
where $c^{-1}(b_{n+1}):=\inf\{k\in\N\dvtx  \frac{k}{b_{k}}\geq
b_{n+1} \}$, then $ (M^{\Pi}_{n}(f)/b_n )$ satisfies a MDP
in $\mathbb{R}$ with the speed $b_{n}^{2}/n$ and the rate function
$ I(x)=\frac{x^{2}}{2(\mu,Pf^{2})}$.
\end{theorem}
\begin{pf} First, note that under the hypothesis, $M_{n}^{\Pi}(f)$ is
a $\mathcal{H}_{n}$-martingale, with $\mathcal{H}_{0}=\sigma(X_{1})$
and $\mathcal{H}_{n}=\sigma(\Delta_{\Pi(i)}, \Pi(i+1), 1\leq
i\leq n )$. From Proposition \ref{propmdp} in Appendix \ref
{appendixB}, we only have to check conditions (C1) and
(C3).

On one hand, (\ref{probaineq2}) applied to $Pf^4-(\mu,Pf^4)$ implies
that for all $\delta>0$,
\[
\limsup_{n\rightarrow\infty}\frac{n}{b_{n}^{2}} \log\mathbb{P} \Biggl(
\frac{1}{n}\sum_{i=1}^{n}Pf^{4}
(X_{\Pi(i)})>\bigl(\mu,Pf^{4}\bigr) + \delta\Biggr) = -\infty,
\]
and this implies the exponential Lindeberg condition (see, e.g.,
\cite{TheseWorms}), that is, condition (C3).

On the other hand, we have $\langle
M^{\Pi}(f)\rangle_{n}=M_{n}^{\Pi}(Pf^{2})$ and (\ref{probaineq2})
applied to $Pf^2-(\mu,Pf^2)$ implies that
\[
\overline{M}{}^{\Pi}_{n} \bigl(Pf^{2}-\bigl(
\mu,Pf^{2}\bigr) \bigr) \underset{{b_n^2}/{n}} {\stackrel{\mathrm
{superexp}} {\longrightarrow}} 0,
\]
that is, condition (C1).
\end{pf}
%
\begin{remark}
One of the main difficulties in the application of this Theorem lies
in the verification of (\ref{cond-led}). Note, however, that in the
range of speed considered it is sufficient to have some uniform
control in $X_i$ of some moment of $f(X_i,X_{2i},X_{2i+1})$
conditionally on $X_i$, which leads to condition of the type
$P|f|^k$ bounded for some $k\ge2$. It is, of course, the case if $f$
is bounded.
\end{remark}
%
\begin{remark}\label{remarkmdpcumulant}
In the special case of model (\ref{bar11}), we have (see Section~\ref{Application}), for $f$ such that $Pf=0$ and for all $k$,
\[
\mathbb{E} \biggl[\exp\biggl(\lambda\frac{b_{n}}{n} f (\Delta_{\Pi(k)}
) \biggr) \Big/\mathcal{H}_{k-1} \biggr] =\exp\biggl(\frac{b_{n}^{2}}{n}
\biggl(\frac{\lambda^{2}
Pf^{2}}{2n} \biggr) (X_{\Pi(k)} ) \biggr).
\]
This condition implies that a MDP is satisfied for
$(M^{\Pi}_{n}(f)/b_n)$. Indeed, if this relation is satisfied, we
then have that for $\lambda\in\mathbb{R}$ the quantity
\[
G_{n}(\lambda)=\frac{\lambda^{2}}{2n}\sum_{k=1}^{n}Pf^{2}(X_{\Pi(k)})=
\frac{\lambda^{2}}{2}\overline{M}{}^{\Pi
}_n\bigl(Pf^2
\bigr)
\]
is an upper and lower cumulant (see, e.g., \cite{TheseWorms}), and
we may apply G\"artner--Ellis-type methodology.
In addition, due to (\ref{probaineq2}) applied to $Pf^2-(\mu,Pf^2)$,
we have for $\lambda\in\mathbb{R}$,
\[
G_{n}(\lambda)\underset{{b_n^2}/{n}} {\stackrel{\mathrm
{superexp}} {\longrightarrow}} \frac{\lambda
^{2}(\mu,Pf^{2})}{2},
\]
which implies that $ (M^{\Pi}_{n}(f)/b_n )$ satisfies a MDP in
$\mathbb{R}$ with the speed $b_{n}^{2}/n$ and the rate function
$ I(x)=\frac{x^{2}}{2(\mu,Pf^{2})}$.
\end{remark}
%

\section{Exponential deviation inequalities for BMC and
consequences}\label{expoprobaineq}

We give here stronger deviation inequalities than the one obtained
in Section \ref{momentscontrol}, namely exponential deviation
inequalities. Of course, it requires more stringent assumptions.

\subsection{Exponential deviation inequalities}

Let us consider the following hypothesis.

\begin{longlist}[(H2)]
\item[(H2)]
There exists a probability $\mu$ on $(S,\mathcal{S})$
such that, for all $f\in\mathcal{B}_{b}(S)$ with $(\mu,f)=0$,
there exists a positive constant $c$ such that
\[
\bigl|Q^{r}f(x)\bigr|\leq c\alpha^{r} \qquad\mbox{for some } \alpha\in
(0,1) \mbox{ and for all } x\in S.
\]
\end{longlist}
One can easily check that, under hypothesis (H2),
$\mathcal{B}_{b}(S)$ fulfills hypothesis (i)--(vi) of the previous
section.

Under this assumption, we will prove exponential deviation
inequalities for $\overline{M}_{\mathbb{G}_{r}}(f)$,
$\overline{M}_{\mathbb{T}_{r}}(f)$ and $\overline{M}{}^{\Pi}_{n}(f)$
when $f\in\mathcal{B}_{b}(S)$ with $(\mu, f)=0$ [resp., $f\in
\mathcal{B}_{b} (S^{3} )$ with $(\mu, Pf)=0$].
%
\begin{theorem}\label{expoprobaineq1} Assume that \textup{(H2)} is
satisfied. Let $f\in\mathcal{B}_{b}(S)$ such that $(\mu, f)=0$.
Then we have for all $\delta
>0$,
%
\begin{eqnarray}
\label{expoineq1}
&&\mathbb{P} \bigl(\overline{M}_{\mathbb
{G}_{r}}(f)>\delta\bigr)\nonumber\\[-8pt]\\[-8pt]
&&\qquad\leq\cases{\displaystyle  \exp\bigl(c''\delta\bigr)\exp
\bigl(-c'\delta^{2}|\mathbb{G}_{r}| \bigr),\cr
\qquad\forall r\in\N, &\quad if $\displaystyle \alpha\leq\frac{1}{2}$,
\cr
\displaystyle
\exp
\bigl(-c'\delta^{2}|\mathbb{G}_{r}| \bigr),\cr
\qquad\forall r \in\N\mbox{ such that }r> r_0, &\quad if
$\displaystyle \frac{1}{2}<\alpha<\frac{\sqrt{2}}{2}$,
\cr
\displaystyle
\exp\biggl(-c'
\delta^{2}\frac{|\mathbb{G}_{r}|}{r} \biggr),\cr
\qquad\forall r\in\N
\mbox{ such that }r>r_0, &\quad if $\displaystyle \alpha^{2}=\frac{1}{2}$,
\cr
\displaystyle
\exp\biggl(-c'\delta^{2}
\frac{1}{\alpha^{2r}} \biggr),\cr
\qquad\forall r \in\N
\mbox{ such that }r>r_0, &\quad if $\displaystyle
\alpha^{2}>\frac{1}{2}$,}\nonumber
\\
\label{expoineq2}
&&\mathbb{P} \bigl(\overline{M}_{\mathbb
{T}_{r}}(f)>\delta\bigr)\nonumber\\[-8pt]\\[-8pt]
&&\qquad
\leq\cases{\displaystyle  \exp\bigl(c''\delta\bigr)\exp
\bigl(-c'\delta^{2}|\mathbb{T}_{r}| \bigr),\cr
\qquad\forall r\in\N, &\quad if
$\displaystyle \alpha<\frac{1}{2}$,
\cr
\displaystyle
\exp\bigl(2c'
\delta(r+1) \bigr)\exp\bigl(-c'\delta^{2}|\mathbb
{T}_{r}| \bigr),\cr
\qquad\forall r\in\N, &\quad if  $\displaystyle \alpha=\frac{1}{2}$,
\cr
\displaystyle
\exp\bigl(-c'\delta^{2}|\mathbb{T}_{r}|
\bigr),\cr
\qquad\forall r \in\N\mbox{ such that }r>r_0-1, &\quad if  $\displaystyle
\frac{1}{2}<\alpha<\frac{\sqrt{2}}{2}$,
\cr
\displaystyle
\exp\biggl(-c'
\delta^{2}\frac{|\mathbb{T}_{r}|}{r+1} \biggr),\cr
\qquad\forall r\in\N\mbox{ such that }
r>r_0-1, &\quad if $\displaystyle \alpha=\frac{\sqrt{2}}{2}$,
\cr
\displaystyle
\exp
\biggl(-c'\delta^{2}\frac{1}{\alpha^{2(r+1)}} \biggr),\cr
\qquad\forall r\in\N^{*}\mbox{ such that }r>r_0-3, &\quad if $\displaystyle \alpha>
\frac{\sqrt{2}}{2}$,}\nonumber
\end{eqnarray}
and
%
\begin{eqnarray}
\label{expoineq3}
&&
\mathbb{P} \bigl(\overline{M}{}^{\Pi}_{n}(f)>
\delta\bigr)\nonumber\\[-8pt]\\[-8pt]
&&\qquad\leq\cases{\displaystyle  \exp\bigl(c''\delta\bigr)
\exp\bigl(-c'\delta^{2}n \bigr),\cr
\qquad\forall n\in\N, &\quad if
$\displaystyle \alpha<\frac{1}{2}$,
\cr
\displaystyle
\exp\bigl(2c'\delta(r_{n}+1)
\bigr)\exp\bigl(-c'\delta^{2}n \bigr),\cr
\qquad\forall n\in
\N, &\quad if $\displaystyle \alpha=\frac{1}{2}$,
\cr
\displaystyle
\exp\bigl(-c'
\delta^{2}n \bigr),\cr
\qquad\forall n\in\N\mbox{ such that }r_{n}>r_0,
&\quad if  $\displaystyle \frac{1}{2}<\alpha<\frac{\sqrt{2}}{2}$,
\cr
\displaystyle
\exp\biggl(-c'
\delta^{2}\frac{n}{r_{n}+1} \biggr),\cr
\qquad\forall n\in\N\mbox{ such that }
r_{n}>r_0, &\quad if $\displaystyle \alpha=\frac{\sqrt{2}}{2}$,
\cr
\displaystyle
\exp
\biggl(-c'\delta^{2}\frac{1}{\alpha^{2(r_{n}+1)}} \biggr),\cr
\qquad\forall n\in\N^{*}\mbox{ such that }r_{n} >r_0-2, &\quad if
$\displaystyle \alpha>\frac{\sqrt{2}}{2}$,}\nonumber
\end{eqnarray}
where $r_0:= \log(\frac{\delta}{c_{0}} )/\log(\alpha)$,
and $c_{0}$, $c'$ and $c''$ are positive constants which depend on
$\alpha$ and $f$, and differ line by line; see the proofs for the
dependence.
\end{theorem}
\begin{pf} The details of the proof are in Section \ref
{proofexpoprobaineq1} in Appendix \ref{appendixA}. It relies
mainly on successive conditioning, using carefully the uniform
geometric ergodicity assumption to get rid of the conditioning.
\end{pf}

The condition about $\alpha$ less than $1/2$ or greater is of course
linked to the binary structure of the tree. The extension to $m$-ary
tree will follow from the same ideas.

%
\begin{theorem}\label{expoprobaineq2} Assume that
\textup{(H2)} is satisfied. Let $f\in\mathcal{B}_{b} (S^{3} )$
such that $(\mu, Pf)=0$. Then we have the same conclusions, for
the three empirical averages $\overline{M}_{\mathbb{G}_{r}}(f)$,
$\overline{M}{}^{\Pi}_{n}(f)$ and $\overline{M}_{\mathbb{T}_{r}}(f)$,
as in the Theorem \ref{expoprobaineq1}.
\end{theorem}

\begin{pf} See Section \ref{proofexpoprobaineq2} in Appendix \ref{appendixA}.
\end{pf}


Now, using the Borel--Cantelli Theorem and (\ref{expoineq3}), we state
easily the following:

\begin{corollary}\label{corslln2} Assume that
\textup{(H2)} is satisfied. Let $f\in\mathcal{B}_{b}(S)$ such that
$(\mu,f)=0$ [resp., $f\in\mathcal{B}_{b}(S^{3})$ and $(\mu,Pf)=0$].
Then $\overline{M}{}^{\Pi}_{n}(f)$ almost surely converges to $0$ as
$n$ goes to $\infty$.
\end{corollary}

\begin{remark}
Of course uniform ergodicity and bounded test functions are surely
a very strong set of assumptions, but it is not so difficult to
verify if the Markov chain's daughters lie in a compact set. We are
convinced that it is possible to consider the geometric ergodic case
and bounded test functions, but for the price of tedious calculations
that we will pursue in an other work. We will also investigate the
use of transportation inequalities, leading to deviation inequality
for Lipschitz test functions under some Wasserstein contraction
property for the kernel $P$, in the spirit of the Theorems 2.5 or
2.11 in \cite{DGW04}.
\end{remark}
%
\subsection{Moderate deviation principle for BMC}

We introduce the following assumption on the speed of the MDP.
%
\begin{hyp}
\label{h1} Let $(b_{n})$ be an increasing sequence of positive real
numbers such that
\[
\frac{b_{n}}{\sqrt{n}}\longrightarrow+\infty
\]
and:
\begin{itemize}
\item if $\alpha^{2}<\frac{1}{2}$, the sequence $(b_{n})$ is such
that $ b_{n}/n\longrightarrow0$;\vspace*{1pt}

\item if $\alpha^{2}=\frac{1}{2}$, the sequence $(b_{n})$ is such
that $ (b_{n}\log
n)/n\longrightarrow0$;\vspace*{1pt}

\item if $\alpha^{2}>\frac{1}{2}$, the sequence $(b_{n})$ is such that
$ (b_{n}\alpha^{r_{n}+1})/\sqrt{n}\longrightarrow0$.\vadjust{\goodbreak}
\end{itemize}
\end{hyp}
Using the MDP for martingale with bounded
jumps (see, e.g., \cite{Dembo,Djellout}), we can now state the
following:
%
\begin{theorem}\label{thmmdp2} Assume that
\textup{(H2)} is satisfied. Let $f\in\mathcal{B}_{b}(S^{3})$ such that
$Pf=0$. Let $(b_n)$ be a sequence of real numbers satisfying the
Assumption \ref{h1}; then $ (M_{n}^{\Pi}(f)/b_n )$ satisfies a
MDP in $S$ with the speed $b_{n}^{2}/n$ and rate function
$ I(x)=\frac{x^{2}}{2(\mu,Pf^{2})}$.
\end{theorem}

\begin{pf}
The proof easily follows from the previous exponential probability
inequalities and the MDP for martingale
with bounded jumps; see, for example, \cite{Dembo,Djellout,TheseWorms}.
\end{pf}

\begin{remark}
Taking particularly $n=|\mathbb{T}_{r}|$ and $(b_n)$ as a sequence
of real numbers satisfying Assumption \ref{h1}, we get that for all
$f\in\mathcal{B}_{b}(S^{3})$,
$ (M_{\mathbb{T}_{r}}(f)/\break b_{|\mathbb{T}_{r}|} )$ satisfies a
MDP in $\mathbb{R}$ with the speed
$b_{|\mathbb{T}_{r}|}^{2}/|\mathbb{T}_{r}|$ and the rate function
$I(x)=\frac{x^{2}}{2(\mu,Pf^{2})}$.
\end{remark}


\section{Application: First order Bifurcating autoregressive
processes}\label{Application}

In this section, we seek to apply the results of the previous
sections to the following bifurcating autoregressive process with
memory 1
defined by
%
\begin{equation}
\label{bar1} \mathcal{L}(X_{1})=\nu
\quad\mbox{and}\quad \forall n\geq1\qquad \cases{\displaystyle
X_{2n}=\alpha_{0}X_{n} + \beta_{0}
+ \varepsilon_{2n},
\cr
\displaystyle
X_{2n+1} = \alpha_{1}X_{n}
+ \beta_{1} + \varepsilon_{2n+1},}
\end{equation}
where $\alpha_{0}, \alpha_{1}\in(-1,1)$; $\beta_{0}, \beta_{1}\in
\mathbb{R}$, $ ((\varepsilon_{2n}, \varepsilon_{2n+1}), n\geq
1 )$ forms a sequence of i.i.d. bivariate random variables and $\nu
$ a
probability measure on $\mathbb{R}$.

Several extensions of the model have been proposed and various
estimators are studied in the literature for the unknown parameters;
see, for instance, \cite{BZ4,BH99,BH00,BHY09,BZ105,BZ205}. See
\cite{BerSapGeg} for a relevant references.

Throughout this section, we assume that the distribution $\nu$ has
finite moments of all orders.

In the sequel, we will study (\ref{bar1}) in two settings:
\begin{itemize}
\item the Gaussian setting which corresponds to the case where $
((\varepsilon_{2n}, \varepsilon_{2n+1}), n\geq1 )$
forms a sequence of i.i.d. bivariate random variables with law
$\mathcal{N}_{2}(0,\Gamma)$ with
%
\begin{equation}
\label{matrixcov} \Gamma=\sigma^{2}\pmatrix{ 1 & \rho
\cr
\rho& 1},\qquad \sigma^{2}>0,\qquad \rho\in(-1,1);
\end{equation}

\item the bounded setting which corresponds to the case where
$X_{1}$ and $ ((\varepsilon_{2n},\break  \varepsilon_{2n+1})$, $n\geq
1 )$, which forms a sequence of centered i.i.d. bivariate random
variables, take their values in a compact set. Let us note that in this case,
$(X_{n},n\in\mathbb{T})$ takes its values in a compact set.
\end{itemize}
Our main goal is to give deviation inequalities and MDP for the
estimator of
the 4-dimensional unknown parameter
$\theta=(\alpha_{0},\beta_{0},\alpha_{1},\beta_{1})$ and for the
statistical
test defined in \cite{Guyon}.

To estimate the 4-parameter
$\theta=(\alpha_{0},\beta_{0},\alpha_{1},\beta_{1})$, as well as
$\sigma^{2}$ and $\rho$, assume we observe a complete subtree
$\mathbb{T}_{r+1}$. The least square estimator
$\hat{\theta}^{r}= (\hat{\alpha}_{0}^{r}, \hat{\beta}_{0}^{r},
\hat{\alpha}_{1}^{r}, \hat{\beta}_{1}^{r} )$ of $\theta$ is given
by (see \cite{Guyon}), for $\eta\in\{0,1\}$,
%
\begin{equation}
\label{estmaparam} \cases{\displaystyle  \hat{\alpha}_{\eta}^{r}=
\frac{|\mathbb{T}_{r}|^{-1}\sum_{i\in
\mathbb{T}_{r}}X_{i}X_{2i+\eta}- (|\mathbb{T}_{r}|^{-1}\sum_{i\in\mathbb
{T}_{r}}X_{i}
) (|\mathbb{T}_{r}|^{-1}\sum_{i\in\mathbb
{T}_{r}}X_{2i+\eta} )} {
|\mathbb{T}_{r}|^{-1}\sum_{i\in\mathbb{T}_{r}}X_{i}^{2}-
(|\mathbb{T}_{r}|^{-1}\sum_{i\in\mathbb
{T}_{r}}X_{i} )^{2}},
\cr
\displaystyle
\hat{\beta}_{\eta}^{r} = |
\mathbb{T}_{r}|^{-1}\sum_{i\in\mathbb{T}_{r}}X_{2i+\eta}-
\hat{\alpha}_{\eta}^{r}|\mathbb{T}_{r}|^{-1}
\sum_{i\in
\mathbb{T}_{r}}X_{i}. }\hspace*{-35pt}
\end{equation}
Notice that in the Gaussian case, this least square estimator
corresponds to the maximum likelihood estimator.

We also need to introduce the estimators of the conditional variance
$\sigma^{2}$ and the
conditional sister--sister correlation $\rho$. These estimators are
naturally given by
%
\begin{equation}
\label{estimcov} \cases{\displaystyle  \hat{\sigma}_{r}^{2}=
\frac{1}{2\mathbb{T}_{r}}\sum_{i\in
\mathbb{T}_{r}}\bigl(\hat{
\varepsilon}_{2i}^{2}+\hat{\varepsilon}_{2i+1}^{2}
\bigr),
\cr
\displaystyle
\hat{\rho}_{r}=\frac{1}{\hat{\sigma}_{r}^{2}}\sum
_{i\in
\mathbb{T}_{r}}\hat{\varepsilon}_{2i} \hat{
\varepsilon}_{2i+1},}
\end{equation}
where the residues are defined by
$\hat{\varepsilon}_{2i+\eta}=X_{2i+\eta}-\hat{\alpha}^{r}_{\eta
}X_{i}-\hat{\beta}_{\eta}^{r}$,
with $\eta\in\{0,1\}$.

Let us denote by $\mathcal{C}_{\mathrm{pol}}(\mathbb{R})$ [resp.,
$\mathcal{C}_{\mathrm{pol}}(\mathbb{R}^{3})$] the set of all continuous
functions $f\dvtx \R\rightarrow\R$ (resp., $f\dvtx \R^{3}\rightarrow\R$) such
that $|f|$ is bounded above by a polynomial. From \cite{Guyon}, we
know that $\mathcal{C}_{\mathrm{pol}}(\mathbb{R})$ fulfills hypotheses
(i)--(vi).

We will take $F=\mathcal{C}_{\mathrm{pol}}^{1}(\mathbb{R})$ the set of all
$\mathcal{C}^{1}$ functions $f\dvtx \R\rightarrow\R$ such that $|f|+|f'|$
is bounded above by a polynomial. Then, one can check that $F$
fulfills hypotheses (i)--(vi). Moreover, for all $f\in F$, hypothesis
(H1) holds with $\alpha=\max(|\alpha_{0}|,|\alpha_{1}|)$.
Let $\mu$ be the unique stationary distribution of the induced
Markov chain $(Y_{r}, r\in\mathbb{N})$; see \cite{Guyon} for more
details.

Let us denote by $\mathcal{C}_{\mathrm{pol}}^{1}(\R^{3})$ the set of all
$\mathcal{C}^{1}$ functions $f\dvtx \R^{3}\rightarrow\R$ such that
$|f|+|f'|$ is bounded\vspace*{1pt} above by a polynomial. We shall denote by
$\mathbf{x}$ (resp., $\mathbf{x}^{2}$, $\mathbf{xy}$, $\mathbf{y},
\ldots$) the element of $\mathcal{C}_{\mathrm{pol}}^{1}(\mathbb{R}^{3})$
defined by $(x,y,z)\mapsto x$ (resp., $x^{2}$, $xy$, $y,\ldots$).

We define two continuous functions $\mu_1\dvtx  \Theta\rightarrow\R$ and
$\mu_2\dvtx
\Theta\times\R_+^*\rightarrow\R$ by writing
%
\begin{equation}
\label{mu} (\mu, \mathbf{x})=\mu_{1}(\theta) \quad\mbox{and}\quad \bigl(\mu,
\mathbf{x}^{2}\bigr) = \mu_{2}\bigl(\theta,
\sigma^{2}\bigr),
\end{equation}
where
$\theta=(\alpha_{0},\beta_{0},\alpha_{1},\beta_{1})\in\Theta
=(-1,1)\times
\R\times(-1,1)\times\R$.

To segregate between $H_0=\{(\alpha_0,\beta_0)=(\alpha_1,\beta_1)\}
$ and its
alternative $H_1=\{(\alpha_0,\beta_0)\not=(\alpha_1,\beta_1)\}$,
we shall use the test statistic
\[
\chi_{r}^{(1)}=\frac{|\mathbb{T}_{r}|}{2\hat{\sigma
}_{r}^{2}(1-\hat\rho_{r})} \bigl\{\bigl(\hat{
\alpha}_{0}^{r}-\hat{\alpha}_{1}^{r}
\bigr)^{2} \bigl(\hat{\mu}_{2,r}-\hat{\mu}_{1,r}^{2}
\bigr)+ \bigl(\bigl(\hat{\alpha}_{0}^{r}-\hat{
\alpha}_{1}^{r}\bigr)\hat{\mu}_{1,r}+\hat{
\beta}_{0}^{r} - \hat{\beta}_{1}^{r}
\bigr)^{2} \bigr\},
\]
where we write $\hat{\mu}_{1,r}=\mu_{1}(\hat{\theta}^{r})$ and
$\hat{\mu}_{2,r}=\mu_{2}(\hat{\theta}_{r},\hat{\sigma}_{r})$.

As usual the Gaussian setting has specific properties that allow
easier calculations and more general assumptions.

\subsection{The Gaussian setting}

We introduce the following assumption on the speed of the MDP.
Let $(b_{n})$ be an increasing sequence of positive real numbers such that
%
\begin{equation}
\label{hh2}\frac{b_{n}}{\sqrt{n}}\longrightarrow+ \infty
\quad\mbox{and}\quad
\frac{b_{n}}{\sqrt{n\log n}}\rightarrow0.
\end{equation}

\begin{prop}\label{expoconvtheta}
Let $(b_{n})$ be a sequence of real numbers satisfying (\ref{hh2}). Then
\[
\hat{\theta}^{r}\underset{{b_{|\mathbb{T}_{r}|}^{2}}/{|\mathbb
{T}_{r}|}} {\stackrel{
\mathrm{superexp}} {\longrightarrow}} \theta.
\]
\end{prop}

\begin{pf}
We will treat the case of $\hat{\alpha}_{0}^{r}$ given in (\ref
{estmaparam}). The others,
$\hat{\beta}_{0}^{r},\hat{\alpha}_{1}^{r}$ and $\hat{\beta
}_{1}^{r}$, given in (\ref{estmaparam}), may be
treated in a similar way. Note that $ \hat{\alpha
}_{0}^{r}=\frac{C_{r}}{B_{r}}$, where
\[
C_{r}=\overline{M}_{\mathbb{T}_{r}}(\mathbf{xy})-\overline
{M}_{\mathbb{T}_{r}}(\mathbf{x}) \overline{M}_{\mathbb{T}_{r}}(\mathbf
{y})
\quad\mbox{and}\quad B_{r}=\overline{M}_{\mathbb{T}_{r}}\bigl(
\mathbf{x}^{2}\bigr)-\overline{M}_{\mathbb{T}_{r}}(
\mathbf{x})^{2}.
\]

Now, using Lemma \ref{exponentialmapping} and Proposition \ref
{expoconv1}, it follows that
\[
\hat{\alpha}_{0}^{r}\underset{{b_{|\mathbb
{T}_{r}|}^{2}}/{|\mathbb{T}_{r}|}}
{\stackrel{\mathrm{superexp}} {\longrightarrow}} \alpha_{0}.
\]
\upqed\end{pf}
%
We recall that in the BAR model (\ref{bar1}), we use
$\alpha=\max\{|\alpha_{0}|, |\alpha_{1}|\}$, and $b:=
\mu_{2}(\theta,\sigma^{2})-\mu_{1}(\theta)^{2}$, where $\mu_1$
and $\mu_2$ are given in (\ref{mu}), so we have the following
deviation inequality:
%
\begin{prop}\label{devineqtheta} For all $\delta>0$, for all $r\in
\N$ and for all $\gamma< \min(\frac{c_{1}b}{1+\delta},\allowbreak
\frac{c_{1}b}{1+\sqrt{\delta}},
\frac{c_{1}b}{1+\sqrt[4]{\delta}} )$, where $c_{1}$ is a
positive constant which depends on $\mu_{1}$, we have
%
\begin{equation}
\label{devineqthetha1} \mathbb{P} \bigl(\bigl\llVert\hat{\theta}^{r}-
\theta\bigr\rrVert>\delta\bigr)\leq\cases{\displaystyle  \frac{c}{\gamma^{4q}\delta
^{4-p}} \biggl(
\frac{1}{4} \biggr)^{r+1}, &\quad  if $\displaystyle \alpha^{2}<
\frac{1}{2}$,
\cr
\displaystyle
\frac{c}{\gamma^{4q}\delta^{4-p}}r^{2} \biggl(
\frac{1}{4} \biggr)^{r+1} &\quad  if $\displaystyle \alpha^{2}=
\frac{1}{2}$,
\cr
\displaystyle
\frac{c}{\gamma^{4q}\delta^{4-p}}\alpha^{4(r+1)}, &\quad  if $\displaystyle
\alpha^{2}>\frac{1}{2}$, }
\end{equation}
where the constant $c$ depends on $\alpha$, $\mu_{1}$,
$\mu_{2}$ and differs line by line, $p=p(\delta)\in\{0,2,4\}$ and
$q=q(\delta)\in\{0,1\}$.
\end{prop}

\begin{remark}
The values of $p$ and $q$ in Proposition \ref{devineqtheta} depend
on the order of $\delta$. For example, if $\delta$ is small enough,
we have $p=0$ and $q=0$.
\end{remark}

%
\begin{pf}
See Section \ref{proofdevineqtheta} in Appendix \ref{appendixA}.
\end{pf}

\begin{remark}
Proposition \ref{devineqtheta} can be improved by calculating
the $2k$th order moments, with $k>2$, as in the proof of Theorem
\ref{4ordermomentcontrol}. But, as we have said, this comes at the
price of enormous computation.
\end{remark}
%
%
\begin{prop}\label{expoconvsigmarho} Let
$(b_{n})$ be a sequence of real numbers satisfying (\ref{hh2}). Then
\[
\bigl(\hat{\sigma}_{r}^{2}, \hat{\rho}_{r}
\bigr)\underset{{b_{|\mathbb
{T}_{r}|}^{2}}/{|\mathbb{T}_{r}|}} {\stackrel{\mathrm{superexp}} {
\longrightarrow}} \bigl(\sigma^{2}, \rho\bigr).
\]
\end{prop}

\begin{pf} Let us first deal with $\sigma_{r}^{2}$ given in (\ref
{estimcov}). We have (see, e.g., \cite{Guyon})
\[
\hat{\sigma}_{r}^{2}=\tfrac{1}{2}
\overline{M}_{\mathbb
{T}_{r}}\bigl(f(\cdot,\theta)\bigr)+D_{r},
\]
where $ f(x,y,z,\theta)=(y-\alpha_{0}x-\beta
_{0})^{2}+(z-\alpha_{1}x-\beta_{1})^{2}$ and
\[
D_r=\frac{1}{2 |\mathbb{T}_{r}|}\sum_{i\in\mathbb{T}_{r}
}
\bigl(f\bigl(\Delta_i,\hat{\theta}^r\bigr)-f(
\Delta_i,{\theta})\bigr).
\]
By the
Taylor--Lagrange formula, we can find $g\in
\mathcal{C}_{\mathrm{pol}}(\mathbb{R}^{3})$ such that (see~\cite{Guyon})
\[
|D_{r}|\leq\tfrac{1}{2}\bigl\|\hat{\theta}^{r}-\theta\bigr\|
\bigl(1+\|\theta\|+\bigl\|\hat{\theta}^{r}-\theta\bigr\| \bigr)\overline{M}
_{\mathbb{T}_{r}}(g).
\]
Now, Propositions \ref{expoconv1} and
\ref{expoconvtheta} lead us to
\[
\hat{\sigma}_{r}^{2}\underset{{b_{|\mathbb
{T}_{r}|}^{2}}/{|\mathbb{T}_{r}|}}
{\stackrel{\mathrm{superexp}} {\longrightarrow}} \sigma^{2}.
\]
The proof for $\hat{\rho}_{r}$ given in (\ref{estimcov}) is similar.
\end{pf}
%
%
\begin{prop}\label{mdpforbar1} Let $(b_{n})$ be a sequence of real
numbers satisfying (\ref{hh2}). Then the sequence
$ (|\mathbb{T}_{r}|(\hat{\theta}^{r}-\theta)/b_{|\mathbb
{T}_{r}|} )$
satisfies the MDP on $\R^4$ with the speed
$b_{|\mathbb{T}_{r}|}^{2}/|\mathbb{T}_{r}|$ and the rate function $I$
given by
\[
I(x)=\tfrac{1}{2}x^{t}\bigl(\Sigma'
\bigr)^{-1}x,
\]
where
\[
\Sigma'=\sigma^{2}\pmatrix{ K & \rho K
\cr
\rho K & K}
\]
with
\[
K=\frac{1}{\mu_{2}(\theta,\sigma^{2})-\mu
_{1}(\theta)^{2}} \pmatrix{ 1 & -\mu_{1}(\theta)
\vspace*{2pt}\cr
-
\mu_{1}(\theta) & \mu_{2}\bigl(\theta,
\sigma^{2}\bigr)}.\vadjust{\goodbreak}
\]
\end{prop}

\begin{pf}
We first observe that
\[
\frac{|\mathbb{T}_{r}|}{b_{|\mathbb{T}_{r}|}} \bigl(\hat{\theta}^{r} -
\theta\bigr) =
M(A_{r},B_{r}). \frac{U^{r}(f)}{b_{|\mathbb{T}_{r}|}},
\]
where\vspace*{1pt} $f=(f_{1}, f_{2}, f_{3}, f_{4})^{t}=(\mathbf{xy}, \mathbf{y},
\mathbf{xz}, \mathbf{z})^{t}$, $U^{r}(f)= M_{\mathbb{T}_{r}}(f-Pf)$,
$A_{r}=\overline{M}_{\mathbb{T}_{r}}(\mathbf{x})$,
$B_{r}=\overline{M}_{\mathbb{T}_{r}}(\mathbf{x}^{2})-\overline
{M}_{\mathbb{T}_{r}}(\mathbf{x})^{2}$
and
\[
M(A_{r},B_{r})=\pmatrix{ \dfrac{1}{B_{r}} &
\dfrac{-A_{r}}{B_{r}} & 0 & 0
\cr
\displaystyle
\dfrac{-A_{r}}{B_{r}} & \dfrac{B_{r}+A_{r}^{2}}{B_{r}} & 0 & 0
\cr
\displaystyle
0 & 0 & \dfrac{1}{B_{r}} & \dfrac{-A_{r}}{B_{r}}
\cr
\displaystyle
0 & 0 &
\dfrac{-A_{r}}{B_{r}} & \dfrac{B_{r}+A_{r}^{2}}{B_{r}}}.
\]
For the sake of simplicity we wrote $Pf=(Pf_{1}, Pf_{2}, Pf_{3},
Pf_{4})^{t}$, where $P$ denotes the $\T$-transition probability
associated to BAR(1) process in the Gaussian case, which is given by
\begin{eqnarray*}
P(x,dy,dz)&=&\frac{1}{2\pi\sigma^{2}(1-\rho^{2})}\\
&&{}\times\exp\biggl(-\frac
{1}{2}\pmatrix{y-
\alpha_{0}x-\beta_{0}
\cr
\displaystyle
z - \alpha_{1}x-
\beta_{1}}^{t}\Gamma^{-1}\pmatrix{ y-
\alpha_{0}x-\beta_{0}
\cr
\displaystyle
z - \alpha_{1}x-
\beta_{1}} \biggr) \,dy\,dz,
\end{eqnarray*}
where $\Gamma$ is the covariance matrix defined in
(\ref{matrixcov}).

On one hand, from Proposition \ref{expoconv1},
\[
A_r\underset{{b_{|\mathbb{T}_{r}|}^{2}}/{|\mathbb
{T}_{r}|}} {\stackrel{\mathrm{superexp}} {
\longrightarrow}} a: = \mu_{1}(\theta) \quad\mbox{and}\quad B_{r}
\underset{{b_{|\mathbb{T}_{r}|}^{2}}/{|\mathbb
{T}_{r}|}} {\stackrel{\mathrm{superexp}} {\longrightarrow}}
b:= \mu_{2}\bigl(\theta, \sigma^{2}\bigr)-
\mu_{1}(\theta)^{2},
\]
so that by Lemma \ref{exponentialmapping}, we obtain
\[
M(A_{r},B_{r}) \underset{{b_{|\mathbb{T}_{r}|}^{2}}/{|\mathbb
{T}_{r}|}} {\stackrel{
\mathrm{superexp}} {\longrightarrow}} M(a,b):=\pmatrix{ K & 0
\cr
\displaystyle
0 & K }.
\]
On the other hand, let $\lambda= (\lambda_{1}, \lambda_{2},
\lambda_{3}, \lambda_{4})^{t}\in\mathbb{R}^{4}$. For all $x\in\R$,
we have that
\begin{eqnarray*}
&&
P\exp\bigl(\lambda^{t}(f - Pf) \bigr) (x) \\
&&\qquad= \int
_{\R^{2}} \exp\Biggl(\sum_{i=1}^4
\lambda_{i}(f_{i} - Pf_{i}) \Biggr)
(x,y,z)P(x,dy,dz)
\\
&&\qquad = \int_{\R^{2}} \exp\left(\lambda^t\pmatrix{
xy-x(\alpha_{0}x+\beta_{0})
\cr
\displaystyle
y- \alpha_{0}x-
\beta_{0}
\cr
\displaystyle
xz-x(\alpha_{1}x+\beta_{1})
\cr
\displaystyle
z-
\alpha_{1}x-\beta_{1}} \right)P(x,dy,dz)
\\
&&\qquad= \exp\biggl(-\pmatrix{ \alpha_{0}x+\beta_{0}
\cr
\displaystyle
\alpha_{1}x+\beta_{1}}^t \pmatrix{
\lambda_{1}x + \lambda_{2}
\cr
\displaystyle
\lambda_{3}x +
\lambda_{4}} \biggr)\\
&&\qquad\quad{}\times \int_{\R^{2}} \exp\biggl(\pmatrix{
\lambda_{1}x + \lambda_{2}
\cr
\displaystyle
\lambda_{3}x +
\lambda_{4}}^t \pmatrix{ y
\cr
\displaystyle
z} \biggr)P(x,dy,dz).
\end{eqnarray*}
We know that
\begin{eqnarray*}
&&
\int_{\R^{2}} \exp\biggl(\pmatrix{ \lambda_{1}x +
\lambda_{2}
\cr
\displaystyle
\lambda_{3}x + \lambda_{4}}^t
\pmatrix{ y
\cr
\displaystyle
z} \biggr)P(x,dy,dz)\\
&&\qquad=\exp\biggl(\pmatrix{ \alpha_{0}x+
\beta_{0}
\cr
\displaystyle
\alpha_{1}x+\beta_{1}
}^t \pmatrix{ \lambda_{1}x + \lambda_{2}
\cr
\displaystyle
\lambda_{3}x + \lambda_{4} } \biggr)
\\
&&\qquad\quad{}\times\exp\biggl(\frac{1}{2}\pmatrix{ \lambda_{1}x +
\lambda_{2}
\cr
\displaystyle
\lambda_{3}x + \lambda_{4}
}^t\Gamma\pmatrix{ \lambda_{1}x + \lambda_{2}
\cr
\displaystyle
\lambda_{3}x + \lambda_{4} } \biggr).
\end{eqnarray*}
Let $\Xi(x)$ denote the square matrix with entries $(Pf_{i}f_{j}
-Pf_{i}Pf_{j})(x)$, for $1\le i,j\le4$. So we obtain that
\begin{eqnarray*}
P\exp\bigl(\lambda^{t}(f - Pf) \bigr) (x)&=& \exp\biggl(
\frac{1}{2}\pmatrix{ \lambda_{1}x + \lambda_{2}
\cr
\displaystyle
\lambda_{3}x + \lambda_{4} }^t\Gamma\pmatrix{
\lambda_{1}x + \lambda_{2}
\cr
\displaystyle
\lambda_{3}x +
\lambda_{4} } \biggr)
\\
& =& \exp\Biggl(\frac{1}{2}\sum_{i,j = 1}^{4}
\lambda_{i} \lambda_{j} (Pf_{i}f_{j}
- Pf_{i}Pf_{j}) (x) \Biggr)
\\
&=& \exp\biggl(\frac{1}{2}\lambda^{t}\Xi(x) \lambda\biggr).
\end{eqnarray*}
Recall that the filtration $(\mathcal{H}_{n})_{n\ge0}$ is defined
by $\mathcal{H}_{0}=\sigma(X_{1})$ and $\mathcal{H}_{n} =
\sigma(\Delta_{\Pi(i)},\allowbreak \Pi(i+1), 1\leq i\leq n )$.
Therefore, from the previous calculations, we deduce that for all
$k\in\N$,
\begin{eqnarray*}
\E\bigl[\exp\bigl(\lambda^{t}(f-Pf) (\Delta_{\Pi(k)}) \bigr)
/ \HH_{k-1} \bigr] &=& P \bigl(\exp\bigl(\lambda^{t}(f-Pf)
\bigr) \bigr) (X_{\Pi(k)}) \\
&=& \exp\bigl(\tfrac{1}{2}
\lambda^{t}\Xi(X_{\Pi(k)})\lambda\bigr).
\end{eqnarray*}
Now, recall that $(M_{n}^{\Pi}(f-Pf))_{n\in\N}$ is a
$(\HH_{n})$-martingale and by straightforward calculations, its
increasing process is given by $ \langle
M^{\Pi}(f-Pf)\rangle_{n} = \sum_{k = 1}^{n}\Xi(X_{\Pi(k)})$.
From the foregoing, we infer that
\[
\biggl(\exp\biggl(\lambda^{t}M_{n}^{\Pi}(f-Pf) -
\frac{\lambda^{t}\langle
M^{\Pi}(f-Pf)\rangle_{n}\lambda}{2} \biggr) \biggr)_{n\in\N}
\]
is a $(\HH_{n})$-martingale. It\vspace*{1pt} then follows that for all
$\lambda\in\R^{4}$, $
G_{n}(\lambda)=\frac{1}{2n}\lambda^{t}\langle
M^{\Pi}(f-Pf)\rangle_{n}\lambda$ is an upper and lower cumulant.
Moreover, from Proposition \ref{expoconv1} and Lemma
\ref{exponentialmapping},
\[
G_{n}(\lambda)\underset{{b_{|\mathbb{T}_{r}|}^{2}}/{|\mathbb
{T}_{r}|}} {\stackrel{
\mathrm{superexp}} {\longrightarrow}} \tfrac{1}{2}\lambda^{t}
\Sigma\lambda\qquad\mbox{where } \Sigma=\sigma^{2}\pmatrix{ K^{-1} &
\rho K^{-1}
\cr
\displaystyle
\rho K^{-1} & K^{-1} }.
\]
We thus deduce that (see, e.g., \cite{TheseWorms})
$ (M_{n}^{\Pi}(f)/b_{n} )$ satisfies a MDP on $\R^4$ with
speed $b_{n}^{2}/n$ and the rate function
%
\begin{equation}
\label{tauxJJ}J(x)=\tfrac{1}{2}x^{t}\Sigma^{-1}x.
\end{equation}
Taking $n=|\mathbb{T}_{r}|$, it follows that
$ (U^{r}(f)/b_{|\mathbb{T}_{r}|} )$ satisfies a MDP with speed
$b_{|\mathbb{T}_{r}|}^{2}/|\mathbb{T}_{r}|$ and the rate\vspace*{1pt}
function $J$ given in (\ref{tauxJJ}). Finally, using the contraction
principle (see, e.g.,
\cite{DemZei}) as in \cite{Worms}, we get the result.
\end{pf}
%

Let us now consider the test statistic.

\begin{prop}\label{mdpforestimator1} Let $(b_n)$ a sequence of real
numbers satisfying (\ref{hh2}).
Then under the null hypothesis
$H_{0}=\{(\alpha_{0},\beta_{0})=(\alpha_{1},\beta_{1})\}$,
$\frac{|\mathbb{T}_{r}|^{1/2}}{b_{|\mathbb{T}_{r}|}}(\chi_{r}^{(1)})^{1/2}$
satisfies a MDP on $\R$ with speed
$b_{|\mathbb{T}_{r}|}^{2}/|\mathbb{T}_{r}|$ and the rate
function
\[
I'(y)= \cases{\displaystyle  \frac{y^{2}}{2}, &\quad  if $y\in
\mathbb{R}_{+}$,
\vspace*{2pt}\cr
\displaystyle
+\infty, &\quad  otherwise. }
\]
Under the alternative hypothesis $H_{1}$ of $H_{0}$, we have for all $A>0$,
\[
\limsup_{r\rightarrow\infty} \frac{|\mathbb{T}_{r}|}{b_{|\mathbb
{T}_{r}|}^{2}}\log\mathbb{P} \bigl(
\chi_{r}^{(1)}<A \bigr)=-\infty.
\]
\end{prop}

\begin{pf} We have
\[
H_{0}=\bigl\{g(\theta)=0\bigr\} \qquad\mbox{where } g(\theta)=(\alpha
_{0}-\alpha_{1}, \beta_{0}-
\beta_{1})^{t}.
\]
From Proposition \ref{mdpforbar1},
$ (|\mathbb{T}_{r}|(\hat{\theta}^{r}-\theta)/b_{|\mathbb
{T}_{r}|} )$
satisfies a MDP on $\R^4$ with speed
$b_{|\mathbb{T}_{r}|}^{2}/|\mathbb{T}_{r}|$ and the rate
function $ I(x)= \frac{1}{2}x^{t}(\Sigma')^{-1}x$. So
that, using the delta method for the MDP (see, e.g., \cite{GaoZhao},
Theorem 3.1) we conclude that\break
$ (|\mathbb{T}_{r}| (g(\hat{\theta}^{r})-g(\theta
))/b_{|\mathbb{T}_{r}|} )$
satisfies a MDP on $\R^2$ with speed
$b_{|\mathbb{T}_{r}|}^{2}/|\mathbb{T}_{r}|$ and the rate
function
\[
J(y) = \inf\bigl\{I(x); y = g'(\theta)x \bigr\}.
\]
Identification of this rate function by usual optimization argument
leads us to
%
\begin{equation}
\label{tauxJ} J(x) = \tfrac{1}{2}x^{t}\bigl(
\Sigma''\bigr)^{-1}x\qquad\mbox{where }
\Sigma'' = 2\sigma^{2}(1-\rho)K.
\end{equation}
Under the null hypothesis $H_{0}$, we have $g(\theta)=0$, so that
$ (|\mathbb{T}_{r}| g(\hat{\theta}^{r})/b_{|\mathbb
{T}_{r}|} )$
satisfies a MDP on $\R^2$ with speed
$b_{|\mathbb{T}_{r}|}^{2}/|\mathbb{T}_{r}|$ and rate function $J$
given in~(\ref{tauxJ}).\vspace*{2pt}

Now, since $K=K(\theta,\sigma)$ is a continuous function of
$(\theta, \sigma)$ (see \cite{Guyon}), so that, letting
$\hat{K}_{r}=K(\hat{\theta}^{r},\hat{\sigma}_{r})$, Lemma
\ref{exponentialmapping}, Propositions \ref{mdpforbar1} and
\ref{expoconvsigmarho} entail that
\[
\hat{\Sigma}''_{r} = 2\hat{
\sigma}_{r}^{2}(1-\hat{\rho}_{r})
\hat{K}_{r} \underset{{b_{|\mathbb{T}_{r}|}^{2}}/{|\mathbb
{T}_{r}|}} {\stackrel{
\mathrm{superexp}} {\longrightarrow}} \Sigma''.\vadjust{\goodbreak}
\]
It follows using the contraction principle
(see, e.g., \cite{Worms}) that
\[
\bigl(|\mathbb{T}_{r}|\mbox{$\hat{\Sigma}''_{r}$}{}^{-1/2}g\bigl(\hat{\theta
}^{r}\bigr)/b_{|\mathbb{T}_{r}|} \bigr)
\]
satisfies a MDP on $\R^2$ with speed
$b_{|\mathbb{T}_{r}|}^{2}/|\mathbb{T}_{r}|$ and the rate
function $ J'(y)=\frac{\|y\|^{2}}{2}$.

In particular,
\[
\biggl\llVert\frac{|\mathbb{T}_{r}|}{b_{|\mathbb{T}_{r}|}}
\mbox{$\hat{\Sigma}''_{r}$}{}^{-1/2}g
\bigl(\hat{\theta}^{r}\bigr)\biggr\rrVert= \frac{|\mathbb
{T}_{r}|^{1/2}}{b_{|\mathbb{T}_{r}|}}\sqrt
{ \chi_{r}^{(1)}}
\]
satisfies a MDP with speed
$b_{|\mathbb{T}_{r}|}^{2}/|\mathbb{T}_{r}|$ and the rate
function $I'$ given in the Proposition \ref{mdpforestimator1}.

Now, under the alternative hypothesis $H_{1}$,
\[
\frac{\chi_{r}^{(1)}}{|\mathbb{T}_{r}|} = g\bigl(\hat{\theta}^{r}\bigr)^{t}
\mbox{$\hat{\Sigma}''_{r}$}{}^{-1}g\bigl(
\hat{\theta}^{r}\bigr) \underset{{b_{|\mathbb
{T}_{r}|}^{2}}/{|\mathbb
{T}_{r}|}} {\stackrel{
\mathrm{superexp}} {\longrightarrow}} g(\theta)^{t}\bigl(
\Sigma''\bigr)^{-1}g(\theta)>0,
\]
so that\vspace*{2pt} $\chi_{r}^{(1)}$ converges
$\frac{b_{|\mathbb{T}_{r}|}^{2}}{|\mathbb{T}_{r}|}$-superexponentially
fast to $+\infty$. This concludes the proof of the Proposition
\ref{mdpforestimator1}.
\end{pf}
%
\subsection{Compact case: The uniformly ergodic setting}

We recall that the model under study in this section is the model
(\ref{bar1}) where we assume that the noise and initial state
$X_{1}$ take their values in a compact set. The results will be
given without proofs, since the proofs are similar to those done in
the previous section. The novelty here is that the range of speed is
improved in comparison to the previous section. However, we suppose
that the process takes its values in a compact set, which is not the
case in the previous section.

We take $F = \mathcal{C}_{b}^{1}(\mathbb{R})$ the set of all
$\mathcal{C}^{1}$ functions bounded on $\R$. Therefore, one can
easily check (as in \cite{Guyon}, proof of Proposition 28) that
hypothesis (H2) is satisfied with
$\alpha=\max(|\alpha_{0}|,|\alpha_{1}|)$. We use the same notation
as in the previous section.

Let us begin by the fact that the estimator of $\theta$ converges
super exponentially fast to the true parameter.
%
\begin{prop}\label{expoconvtheta2} Let $(b_n)$ a sequence of real numbers
satisfying the Assumption \ref{h1}. Then we have
\[
\hat{\theta}^{r} \underset{{b_{|\mathbb
{T}_{r}|}^{2}}/{|\mathbb{T}_{r}|}} {\stackrel{\mathrm
{superexp}} {\longrightarrow}} \theta.
\]
\end{prop}
We may now refine this result by proving deviation inequality.
%
\begin{prop}\label{devineqtheta2} For all $\delta>0$ and for all
$\gamma<
\min(\frac{c_{1}b}{1+\delta}, \frac{c_{1}b}{1+\sqrt{\delta}},
\frac{c_{1}b}{1+\sqrt[4]{\delta}} )$, where $c_{1}$ is a
positive constant which depends on\vadjust{\goodbreak} $\mu_{1}$, and for
$r_0:=\break
\frac{\log(\gamma^{q}\delta^{1-p/2}/c_{0} )}{\log
\alpha}$,
we have
%
\begin{eqnarray}
\label{devineqthetha2}\qquad
\mathbb{P} \bigl(\bigl\llVert\hat{\theta}^{r}-
\theta\bigr\rrVert>\delta\bigr)
&\leq&\cases{\displaystyle  c_{2}\exp
\bigl(c''\gamma^{q}\delta^{1-p/2}
\bigr)\exp\bigl(-c'\gamma^{2q}\delta^{2-p}|
\mathbb{T}_{r}| \bigr),\vspace*{2pt}\cr
\qquad\forall r\in\N, \hspace*{2.75pt}\qquad  \mbox{if $\displaystyle\alpha<
\frac{1}{2}$},
\vspace*{2pt}\cr
\displaystyle
c_{2}\exp\bigl(c'
\gamma^{q}\delta^{1-p/2}(r+1)-c'
\gamma^{2q}\delta^{2-p}|\mathbb{T}_{r}| \bigr),\vspace*{2pt}\cr
\qquad\forall r\in\N, \hspace*{2.75pt}\qquad  \mbox{if $\displaystyle\alpha=\frac{1}{2}$},
\vspace*{2pt}\cr
\displaystyle
c_{2}\exp
\bigl(-c'\gamma^{2q}\delta^{2-p}|
\mathbb{T}_{r}| \bigr),\vspace*{2pt}\cr
\qquad\forall r > r_0, \qquad
\mbox{if $\displaystyle\frac{1}{2}<\alpha<\frac{\sqrt{2}}{2}$},
\vspace*{2pt}\cr
\displaystyle
c_{2}\exp
\biggl(-c'\gamma^{q}\delta^{2-p}
\frac{|\mathbb
{T}_{r}|}{r+1} \biggr),\vspace*{2pt}\cr
\qquad\forall r >r_0,
\qquad  \mbox{if $\displaystyle\alpha=\frac{\sqrt{2}}{2}$},
\vspace*{2pt}\cr
\displaystyle
c_{2}\exp\biggl(-c'\gamma^{2q}\delta^{2-p}\frac{1}{\alpha
^{2(r+1)}} \biggr),\vspace*{2pt}\cr
\qquad\forall r >r_0, \qquad \mbox{if
$\displaystyle\alpha>\frac{\sqrt{2}}{2}$},}
\end{eqnarray}
where $c_{2}$ is a positive constant, $c'$ and $c''$ depend on
$\alpha$, and $c$ and may differ line by line, $c_{0}$ depends on
$\alpha$, $c$ and $\gamma$, and may differ line by line,
$p\in\{0,1,3/2\}$ and $q\in\{0,1\}$.
\end{prop}
We have now to consider super exponential convergence of the estimators
of the other parameters.

\begin{prop}\label{expoconvsigmarho2} Let $(b_n)$ a sequence of real
numbers satisfying Assumption \ref{h1}.
Then we have
\[
\bigl(\hat{\sigma}_{r}^{2}, \hat{\rho}_{r}
\bigr)\underset{{b_{|\mathbb{T}_{r}|}^{2}}/{|\mathbb
{T}_{r}|}} {\stackrel{\mathrm{superexp}} {
\longrightarrow}} \bigl(\sigma^{2}, \rho\bigr).
\]
\end{prop}
As previously we may now prove MDP for the estimator of $\theta$.
%
\begin{prop}\label{mdpforbar2} Let $(b_{n})$ a sequence of
real numbers satisfying the Assumption \ref{h1}. Then
$ (|\mathbb{T}_{r}|
(\hat{\theta}^{r}-\theta)/b_{|\mathbb{T}_{r}|} )$ satisfies
the MDP on $\R^4$ with the speed
$b_{|\mathbb{T}_{r}|}^{2}/|\mathbb{T}_{r}|$ and rate function
\[
I(x)=\tfrac{1}{2}x^{t}\bigl(\Sigma'
\bigr)^{-1}x,
\]
where
\[
\Sigma'=\sigma^{2}\pmatrix{ K & \rho K
\cr
\displaystyle
\rho K & K }
\]
with
\[
K=\frac{1}{\mu_{2}(\theta,\sigma
^{2})-\mu_{1}(\theta)^{2}} \pmatrix{ 1 & -\mu_{1}(\theta)
\cr
\displaystyle
-
\mu_{1}(\theta) & \mu_{2}\bigl(\theta,
\sigma^{2}\bigr)}.
\]
\end{prop}

\begin{remark}
Notice that the proof of Proposition \ref{mdpforbar2} does not need the
cumulant method as in the proof of Proposition \ref{mdpforbar1}.
Indeed, since we are in the bounded case, from MDP of martingale with
bounded jumps (see \cite{Dembo}), we need only to prove the
superexponential convergence of increasing process of the martingale.
This convergence is easily obtained from Theorem~\ref{expoprobaineq2}.
\end{remark}
Let us give us our last result by considering a MDP for the test statistic.
%
\begin{prop}\label{mdpforestimator2} Let $(b_n)$ a sequence of real
numbers satisfying
the Assumption \ref{h1}. Then under the null hypothesis $H_{0} =
\{(\alpha_{0},\beta_{0})=(\alpha_{1},\beta_{1})\}$,
$\frac{|\mathbb{T}_{r}|^{1/2}}{b_{|\mathbb{T}_{r}|}}(\chi_{r}^{(1)})^{1/2}$
satisfies a MDP on $\R$ with speed
$b_{|\mathbb{T}_{r}|}^{2}/|\mathbb{T}_{r}|$ and the rate
function
\[
I'(y)= \cases{\displaystyle  \frac{y^{2}}{2}, &\quad  if $y\in
\mathbb{R}_{+}$,
\vspace*{2pt}\cr
\displaystyle
+\infty, &\quad  otherwise. }
\]
Under the alternative hypothesis $H_{1}$ of $H_{0}$, we have for all $A>0$,
\[
\limsup_{r\rightarrow
\infty}\frac{|\mathbb{T}_{r}|}{b_{|\mathbb{T}_{r}|}^{2}}\log\mathbb{P}
\bigl(
\chi_{r}^{(1)}<A \bigr) = -\infty.
\]
\end{prop}

\begin{appendix}\label{app}
\section{Proof of the exponential inequalities}\label{appendixA}
This section is devoted to the proofs of Theorems \ref{someprobaineq1},
\ref{someprobaineq2}, \ref{expoprobaineq1}, \ref{expoprobaineq2}
and Proposition \ref{devineqtheta}.
\subsection{\texorpdfstring{Proof of Theorem \protect\ref{someprobaineq1}}
{Proof of Theorem 2.11}} \label{proofsomeprobaineq1}
Let $f\in F$ such that $(\mu, f) = 0$. We shall study
the three empirical averages $\overline{M}_{\mathbb{G}_{r}}(f)$,
$\overline{M}{}^{\Pi}_{n}(f)$ and $\overline{M}_{\mathbb{T}_{r}}(f)$
successively.\vspace*{1pt}

\textit{Part} 1. Let us first deal with $\overline{M}_{\mathbb
{G}_{r}}(f)$. By the  Markov inequality, we get, for all $\delta>0$,
\[
\mathbb{P} \bigl( \bigl|\overline{M}_{\mathbb{G}_{r}}(f) \bigr|>\delta\bigr)
=\mathbb{P}
\bigl( \bigl|\overline{M}_{\mathbb{G}_{r}}(f) \bigr|^{2}>\delta^{2}
\bigr) \leq\frac{1}{\delta^{2}}\mathbb{E} \bigl[\bigl(\overline
{M}_{\mathbb
{G}_{r}}(f)
\bigr)^{2} \bigr].
\]
By Guyon (see \cite{Guyon}), we have
\[
\mathbb{E} \bigl[\bigl(\overline{M}_{\mathbb{G}_{r}}(f)\bigr)^{2}
\bigr]=\sum_{p=0}^{r}2^{-p-\mathbf{1}_{p<r}}\nu
Q^{p}P \bigl(Q^{r-p-1}f\otimes Q^{r-p-1}f \bigr).
\]
Hypothesis (H1) implies that there exists $g\in F$ and
$\alpha\in(0,1)$ such that for all $p\in\{0,1,\ldots,r \}$,
\[
\nu Q^{p}P\bigl(Q^{r-p-1}f\otimes Q^{r-p-1}f\bigr) \leq
\alpha^{2(r-p-1)}\nu Q^{p}P(g\otimes g).
\]
Next, hypotheses (iii), (v) and (vi) imply
that there is a positive constant $c$ such that for all $p\in
\{0,1,\ldots,r \}$,
\[
\alpha^{2(r-p-1)}\nu Q^{p}P(g\otimes g)\leq c
\alpha^{2(r-p-1)}.
\]
This leads us to
%
\begin{eqnarray}
\label{moment2} \mathbb{E} \bigl[\bigl(\overline{M}_{\mathbb{G}_{r}}(f)
\bigr)^{2} \bigr]&\leq& c\sum_{p=0}^{r}2^{-p-\mathbf{1}_{p<r}}
\alpha^{2(r-p-1)} \nonumber\\[-8pt]\\[-8pt]
&=& \cases{\displaystyle  c \biggl(\frac{1}{2} \biggr)^{r}
+ c\frac{\alpha^{2r}- ({1}/{2} )^{r}}{2\alpha^{2}-1}, &\quad  if $\alpha
^{2}\neq\dfrac{1}{2}$,
\vspace*{2pt}\cr
\displaystyle
cr
\biggl(\frac{1}{2} \biggr)^{r}, &\quad  if $\alpha^{2}=
\dfrac{1}{2}$,}\nonumber
\end{eqnarray}
and therefore (\ref{probaineq1}) follows.\vspace*{9pt}

\textit{Part} 2. Let us now consider $\overline{M}{}^{\Pi}_{n}(f)$. By
the Markov inequality and the triangle inequality, we get, for all
$\delta>0$,
%
\begin{eqnarray}
\label{decomL2}
&&\mathbb{P} \bigl( \bigl|\overline{M}{}^{\Pi}_{n}(f)
\bigr|>\delta\bigr) \nonumber\\
&&\qquad= \mathbb{P} \bigl( \bigl|\overline{M}{}^{\Pi}_{n}(f)
\bigr|^{2}>\delta^{2} \bigr)\leq\frac{1}{\delta^{2}}\mathbb{E}
\bigl[ \bigl(\overline{M}{}^{\Pi
}_{n}(f) \bigr)^{2}
\bigr]
\\
&&\qquad\leq \frac{2}{\delta^{2}}\mathbb{E} \Biggl[ \Biggl(\sum
_{q=0}^{r_{n}-1}\frac{2^{q}}{n}\overline{M}_{\mathbb{G}_{q}}(f)
\Biggr)^{2} \Biggr] + \frac{2}{\delta^{2}}\mathbb{E} \Biggl[ \Biggl(
\frac{1}{n}\sum_{i=2^{r_{n}}}^{n}f(X_{\Pi(i)})
\Biggr)^{2} \Biggr].
\nonumber
\end{eqnarray}
In the last inequality (\ref{decomL2}), we have used the decomposition
\[
\overline{M}{}^{\Pi}_{n}(f)=\sum_{q=0}^{r_{n}-1}
\frac
{2^{q}}{n}\overline{M}_{\mathbb{G}_{q}}(f) + \frac{1}{n}\sum
_{i=2^{r_{n}}}^{n}f(X_{\Pi(i)}).
\]
In what follows, the constant $c$ may be slightly different from
that of part 1 and may differ term by term. For the first
term appearing in (\ref{decomL2}), we have
\[
\mathbb{E} \Biggl[ \Biggl(\sum_{q=0}^{r_{n}-1}
\frac
{2^{q}}{n}\overline{M}_{\mathbb{G}_{q}}(f) \Biggr)^{2} \Biggr] =
\Biggl\llVert\sum_{q=0}^{r_{n}-1}
\frac{2^{q}}{n}\overline{M}_{\mathbb{G}_{q}}(f)\Biggr\rrVert
_{2}^{2} \leq\Biggl(\sum_{q=0}^{r_{n}-1}
\frac{2^{q}}{n}\bigl\llVert\overline{M}_{\mathbb{G}_{q}}(f)\bigr
\rrVert
_{2} \Biggr)^{2}.
\]
Using (\ref{moment2}), we get that
\[
\sum_{q=0}^{r_{n}-1}\frac{2^{q}}{n}\bigl
\llVert\overline{M}_{\mathbb{G}_{q}}(f)\bigr\rrVert_{2}\leq\cases{\displaystyle
\frac{c}{n}\sum_{q=0}^{r_{n}-1}(\sqrt
2)^{q}\leq c\frac{\sqrt2^{{r_{n}}}}{n}, &\quad  if $\displaystyle\alpha^{2}<
\frac{1}{2}$,
\cr
\displaystyle
\frac{c}{n}\sum_{q=0}^{r_{n}}q^{1/2}
\sqrt2^{q}\leq c\frac{r_{n}^{1/2}\sqrt2^{r_{n}}}{n}, &\quad  if $\displaystyle\alpha^{2}=
\frac{1}{2}$,
\cr
\displaystyle
\frac{c}{n}\sum_{q=0}^{r_{n}-1}(2
\alpha)^{q}\leq c\alpha^{r_{n}}, &\quad  if $\displaystyle\alpha^{2}>
\frac{1}{2}$,}
\]
which implies that
%
\begin{equation}
\label{term1} \quad \mathbb{E} \Biggl[ \Biggl(\sum_{q=0}^{r_{n}-1}
\frac{2^{q}}{n}\overline{M}_{\mathbb
{G}_{q}}(f) \Biggr)^{2} \Biggr]
\leq\cases{\displaystyle  c\frac{2^{r_{n}}}{n^{2}}\leq c \biggl(\frac{1}{2}
\biggr)^{r_{n}+1}, &\quad  if $\displaystyle\alpha^{2}< \frac{1}{2}$,
\vspace*{2pt}\cr
\displaystyle
c
\frac{r_{n}}{2^{r_{n}+1}}, &\quad  if $\displaystyle\alpha^{2}=\frac{1}{2}$,
\vspace*{2pt}\cr
\displaystyle
c
\alpha^{2(r_{n}+1)}, &\quad  if $\displaystyle\alpha^{2}>\frac{1}{2}$.}
\end{equation}
Now, we have to control the second term in (\ref{decomL2}). As in
Guyon \cite{Guyon}, we have that
\begin{eqnarray*}
&&
\mathbb{E} \Biggl[ \Biggl(\frac{1}{n}\sum_{i=2^{r_{n}}}^{n}f(X_{\Pi(i)})
\Biggr)^{2} \Biggr] \\
&&\qquad\leq\frac{n-2^{r_{n}}+1}{n^{2}}\nu Q^{r_{n}}f^{2}
\\
&&\qquad\quad{} + \frac{(n-2^{r_{n}})(n-2^{r_{n}}+1)}{n^{2}(1-2^{-r_{n}})}\sum
_{p=0}^{r_{n}-1}2^{-p-1}
\nu Q^{p}P\bigl(Q^{r_{n}-p-1}f\otimes Q^{r_{n}-p-1}f\bigr)
\\
&&\qquad\leq\frac{c}{n} + c\sum_{p=0}^{r_{n}-1}2^{-p-1}
\alpha^{2r_{n}-2p-2}.
\end{eqnarray*}
Discussing following the value of $\alpha$, we obtain that
%
\begin{equation}
\label{term2} \mathbb{E} \Biggl[ \Biggl(\frac{1}{n}\sum
_{i=2^{r_{n}}}^{n}f(X_{\Pi(i)}) \Biggr)^{2}
\Biggr] \leq\cases{\displaystyle  c\frac{1}{2^{r_{n}+1}}, &\quad  if $\displaystyle\alpha^{2}<
\frac{1}{2}$,
\vspace*{2pt}\cr
\displaystyle
c\frac{r_{n}}{2^{r_{n}+1}}, &\quad  if $\displaystyle\alpha^{2}=
\frac{1}{2}$,
\vspace*{2pt}\cr
\displaystyle
c\alpha^{2(r_{n}+1)}, &\quad  if $\displaystyle\alpha^{2}>
\frac{1}{2}$.}
\end{equation}
Inequality (\ref{probaineq2}) then follows from (\ref{term1}) and
(\ref{term2}).\vspace*{9pt}

\textit{Part} 3. The case of $\overline{M}_{\mathbb{T}_{r}}(f)$ can
be deduced from the previous by taking \mbox{$n=|\mathbb{T}_{r}|$}.

\subsection{\texorpdfstring{Proof of Theorem \protect\ref{someprobaineq2}}
{Proof of Theorem 2.12}}\label{proofsomeprobaineq2}
Let $f\in\mathcal{B} (S^{3} )$ such that $Pf$ and $Pf^{2}$
exist and belong
to $F$ and $(\mu, Pf)=0$. We shall\vspace*{1pt} study the three empirical
averages $\overline{M}_{\mathbb{G}_{r}}(f)$,
$\overline{M}{}^{\Pi}_{n}(f)$ and $\overline{M}_{\mathbb{T}_{r}}(f)$
successively.

\textit{Part} 1. Let us first deal with
$\overline{M}_{\mathbb{G}_{r}}(f)$. By the Markov inequality, we get
for all $\delta>0$,
\begin{eqnarray*}
\mathbb{P} \bigl( \bigl|\overline{M}_{\mathbb{G}_{r}}(f) \bigr|>\delta\bigr)
&\leq& \frac{1}{\delta^{2}}\mathbb{E} \bigl[\bigl(\overline{M}_{\mathbb
{G}_{r}}(f)
\bigr)^{2} \bigr]
\\
&=& \frac{1}{\delta^{2}}\mathbb{E} \bigl[\bigl(\overline{M}_{\mathbb
{G}_{r}}(Pf)
\bigr)^{2} \bigr] + \frac{1}{\delta^{2}}\frac{1}{|\mathbb
{G}_{r}|}\mathbb{E}
\bigl[\overline{M}_{\mathbb{G}_{r}} \bigl(Pf^{2}-(Pf)^{2}
\bigr) \bigr]
\\
&\leq&\frac{1}{\delta^{2}}\mathbb{E} \bigl[\bigl(\overline{M}_{\mathbb
{G}_{r}}(Pf)
\bigr)^{2} \bigr] + \frac{c}{\delta^{2}} \biggl(\frac{1}{2}
\biggr)^{r}.
\end{eqnarray*}
The last inequality follows from the convergence of the sequence\break
$ (\mathbb{E} [\overline{M}_{\mathbb{G}_{r}} (Pf^{2}-(Pf)^{2}
) ] )_r$ (see \cite{Guyon}).

Now, using part 1 of the proof of Theorem
\ref{someprobaineq1} with $Pf$ instead of $f$ leads us to a similar
inequality (\ref{probaineq1}) in Theorem \ref{someprobaineq2} for
$f\in\mathcal{B} (S^{3} )$.

\textit{Part} 2. Let us now treat $\overline{M}{}^{\Pi}_{n}(f)$. Using
the two equalities
\begin{eqnarray*}
\overline{M}{}^{\Pi}_{n}(f)&=&\sum
_{q=0}^{r_{n}-1}\frac
{|\mathbb{G}_{q}|}{n} \overline{M}_{\mathbb{G}_{q}}(f)
+ \frac{1}{n}\sum_{i=2^{r_{n}}}^{n}f(
\Delta_{\Pi(i)}),
\\
\mathbb{E} \Biggl[ \Biggl(\frac{1}{n}\sum_{i=2^{r_{n}}}^{n}f(
\Delta_{\Pi(i)}) \Biggr)^{2} \Biggr] &=& \mathbb{E} \Biggl[
\Biggl(\frac{1}{n} \sum_{i=2^{r_{n}}}^{n}
Pf(X_{\Pi(i)}) \Biggr)^{2} \Biggr]\\
&&{} + \frac{1}{n}
\mathbb{E} \Biggl[ \frac{1}{n}\sum_{i=2^{r_{n}}}^{n}
\bigl(Pf^{2} - (Pf)^{2} \bigr) (X_{\Pi(i)} )
\Biggr],
\end{eqnarray*}
and part 2 of the proof of Theorem
\ref{someprobaineq1} with $Pf$ instead of $f$ leads us to a similar
inequality (\ref{probaineq2}) in Theorem \ref{someprobaineq2} for
$f\in\mathcal{B}$.\vspace*{9pt}

\textit{Part} 3. The case of $\overline{M}_{\mathbb{T}_{r}}(f)$ can
be deduced from the previous by taking \mbox{$n=|\mathbb{T}_{r}|$}.


\subsection{\texorpdfstring{Proof of Theorem \protect\ref{expoprobaineq1}}
{Proof of Theorem 3.1}} \label{proofexpoprobaineq1}
Let\vspace*{1pt} $f\in\mathcal{B}_b(S)$ such that $(\mu,
f)=0$. We shall study the three empirical averages
$\overline{M}_{\mathbb{G}_{r}}(f)$, $\overline{M}{}^{\Pi}_{n}(f)$ and
$\overline{M}_{\mathbb{T}_{r}}(f)$ successively.

\textit{Part} 1. Let us first deal with
$\overline{M}_{\mathbb{G}_{r}}(f)$. We have for all $\lambda>0$ and
for all $\delta>0$
%
\begin{equation}
\label{chernofineq} \mathbb{P}\bigl(\overline{M}_{\mathbb
{G}_{r}}(f)>\delta\bigr)
\leq\exp\bigl(-\lambda\delta|\mathbb{G}_{r}|\bigr)\mathbb{E} \biggl[\exp
\biggl(\lambda\sum_{i\in
\mathbb{G}_{r}}f(X_{i}) \biggr)
\biggr].
\end{equation}
By subtracting and adding terms, we get
\begin{eqnarray*}
&& \mathbb{E} \biggl[\exp\biggl(\lambda\sum_{i\in
\mathbb{G}_{r}}f(X_{i})
\biggr) \biggr]
\\
&&\qquad =\mathbb{E} \biggl[\mathbb{E} \biggl[ \prod_{i\in
\mathbb{G}_{r-1}}
\exp\bigl(\lambda\bigl(f(X_{2i})+f(X_{2i+1})-2Qf(X_{i})
\bigr)\bigr)\\
&&\qquad\quad\hspace*{94pt}{}\times\prod_{i\in
\mathbb{G}_{r-1}}\exp\bigl(2\lambda
Qf(X_{i})\bigr)/\mathcal{F}_{r-1} \biggr] \biggr].
\end{eqnarray*}
Now using the fact that conditionally to the $(r-1)$ first
generations the sequence $\{\Delta_{i}, i\in\mathbb{G}_{r-1}\}$ is
a sequence of independent random variables, we have that
\begin{eqnarray*}
&& \mathbb{E} \biggl[\mathbb{E} \biggl[ \prod_{i\in
\mathbb{G}_{r-1}}
\exp\bigl(\lambda\bigl(f(X_{2i})+f(X_{2i+1})-2Qf(X_{i})
\bigr)\bigr)\\
&&\qquad\quad\hspace*{60.8pt}{}\times\prod_{i\in
\mathbb{G}_{r-1}}\exp\bigl(2\lambda
Qf(X_{i})\bigr)/\mathcal{F}_{r-1} \biggr] \biggr]
\\
&&\qquad =\mathbb{E} \biggl[\prod_{i\in
\mathbb{G}_{r-1}}\exp\bigl(2
\lambda Qf(X_{i}) \bigr)\\
&&\qquad\quad\hspace*{10pt}{}\times\prod_{i\in
\mathbb{G}_{r-1}}
\mathbb{E} \bigl[\exp\bigl(\lambda\bigl(f(X_{2i})+f(X_{2i+1})-2Qf(X_{i})
\bigr) \bigr)/ \mathcal{F}_{r-1} \bigr] \biggr].
\end{eqnarray*}
Using the Azuma--Bennett--Hoeffding inequalities
\cite{Azuma,Bennett,Hoeffding} (see Lemma \ref{lemABH} for more
detail), we get according to (H2), for all $i\in
\mathbb{G}_{r-1}$,
\[
\mathbb{E} \bigl[\exp\bigl(\lambda\bigl(f(X_{2i})+f(X_{2i+1})-2Qf(X_{i})
\bigr) \bigr)/ \mathcal{F}_{r-1} \bigr] \leq\exp\bigl(2
\lambda^{2}c^{2}(1+\alpha)^{2} \bigr).
\]
This leads us to
\[
\mathbb{E} \biggl[\exp\biggl(\lambda\sum_{i\in
\mathbb{G}_{r}}f(X_{i})
\biggr) \biggr] \leq\exp\bigl(\lambda^{2}c^{2}(1+
\alpha)^{2}|\mathbb{G}_{r}| \bigr)\mathbb{E} \biggl[\prod
_{i\in
\mathbb{G}_{r-1}}\exp\bigl(2\lambda Qf(X_{i})
\bigr) \biggr].
\]
Doing the same thing for $\mathbb{E} [\prod_{i\in
\mathbb{G}_{r-1}}\exp(2\lambda Qf(X_{i}) ) ]$ with
$Qf$ replacing $f$, we get
\begin{eqnarray*}
&&
\mathbb{E} \biggl[\prod_{i\in
\mathbb{G}_{r-1}}\exp\bigl(2\lambda
Qf(X_{i}) \bigr) \biggr] \\
&&\qquad\leq\exp\bigl(2\lambda^{2}c^{2}
\bigl(\alpha+\alpha^{2}\bigr)^{2}|\mathbb
{G}_{r}| \bigr)\mathbb{E} \biggl[\prod_{i\in
\mathbb{G}_{r-2}}
\exp\bigl(2^{2}\lambda Q^{2}f(X_{i}) \bigr)
\biggr].
\end{eqnarray*}
Iterating this procedure, we get
\begin{eqnarray*}
\mathbb{E} \biggl[\exp\biggl(\lambda\sum_{i\in
\mathbb{G}_{r}}f(X_{i})
\biggr) \biggr]
&\leq&\mathbb{E} \bigl[\exp\bigl(2^{r}\lambda
Q^{r}f(X_{1}) \bigr) \bigr] \\
&&{}\times\prod
_{k=1}^{r} \exp\bigl(2^{k-1}
\lambda^{2}c^{2}\bigl(\alpha^{k-1}+\alpha
^{k}\bigr)^{2}|\mathbb{G}_{r}| \bigr).
\end{eqnarray*}
Once again, according to (H2), we have
\[
\mathbb{E} \biggl[\exp\biggl(\lambda\sum_{i\in
\mathbb{G}_{r}}f(X_{i})
\biggr) \biggr] \leq\exp\bigl(\lambda c\alpha^{r}|
\mathbb{G}_{r}| \bigr)\times\exp\Biggl(\lambda^{2}c^{2}(1+
\alpha)^2|\mathbb{G}_{r}|\sum_{k=1}^r
\bigl(2\alpha^2\bigr)^{k-1} \Biggr).
\]
Hence:
\begin{itemize}
\item if $\alpha^{2}\neq\frac{1}{2}$, then
\[
\mathbb{E} \biggl[\exp\biggl(\lambda\sum_{i\in
\mathbb{G}_{r}}f(X_{i})
\biggr) \biggr] \leq\exp\biggl(\lambda^{2}c^{2}(1 +
\alpha)^{2}\frac{1-(2\alpha
^{2})^{r}}{1-2\alpha^{2}}|\mathbb{G}_{r}| \biggr)\times
\exp\bigl(\lambda c\alpha^{r}|\mathbb{G}_{r}| \bigr);
\]

\item if $\alpha^{2}=\frac{1}{2}$, then
\[
\mathbb{E} \biggl[\exp\biggl(\lambda\sum_{i\in
\mathbb{G}_{r}}f(X_{i})
\biggr) \biggr] \leq\exp\bigl(\lambda^{2}c^{2}(1+
\alpha)^{2}r|\mathbb{G}_{r}| \bigr)\times\exp\biggl(\lambda
c\biggl(\frac{\sqrt{2}}{2}\biggr)^{r}|\mathbb{G}_{r}|
\biggr).
\]
\end{itemize}
We then consider three cases:\vspace*{1pt}

(a) If $\alpha^{2}<\frac{1}{2}$, then $\frac{1-(2\alpha
^{2})^{r}}{1-2\alpha^{2}}<\frac{1}{1-2\alpha^{2}}$ for all $r$.
Taking $\lambda=\frac{(1-2\alpha^2)\delta}{2c^{2}(1+\alpha)^{2}}$ in
(\ref{chernofineq}) leads us to
\[
\mathbb{P} \bigl(\overline{M}_{\mathbb{G}_{r}}(f)>\delta\bigr) \leq
\exp
\biggl(- \biggl(\frac{(1-2\alpha^2)\delta^{2}}{4c^{2}(1 +
\alpha)^{2}}-\alpha^{r}\frac{(1-2\alpha^2)\delta}{2c(1 +
\alpha)^{2}}
\biggr)|\mathbb{G}_{r}| \biggr).
\]
\begin{itemize}
\item If $\alpha\leq\frac{1}{2}$, then $(2\alpha)^{r}\leq1$ for
all $r\in\N$. We then have for all $r\in\N$,
\[
\mathbb{P} \bigl(\overline{M}_{\mathbb{G}_{r}}(f)>\delta\bigr) \leq
\exp\biggl(
\frac{(1-2\alpha^2)\delta}{2c(1+\alpha)^{2}} \biggr)\exp\biggl(- \frac
{(1-2\alpha^2)\delta^{2}|\G_{r}|}{4c^{2}(1+\alpha)^{2}} \biggr).
\]

\item If $\frac{1}{2}<\alpha<\frac{\sqrt{2}}{2}$, then for all
$r\in
\N$ such that $r > \log(\frac{\delta}{4c} )/\log\alpha
$, we
have $(\delta- 2c\alpha^{r}) > \frac{\delta}{2}$, and it then
follows that
\[
\mathbb{P} \bigl(\overline{M}_{\mathbb{G}_{r}}(f)>\delta\bigr) \leq
\exp
\biggl(- \frac{(1-2\alpha^2)\delta^{2}|\G_{r}|}{8c^{2}(1+\alpha)^{2}}
\biggr).
\]
\end{itemize}

(b) If $\alpha^{2}=\frac{1}{2}$, then for all $\lambda>0$,
\begin{eqnarray*}
\P\bigl(\overline{M}_{\G_{r}}(f) > \delta\bigr) &\leq&\exp\bigl(
\bigl(-\delta\lambda+ c^{2}(1+\alpha)^{2}r
\lambda^{2} \bigr)|\G_{r}| \bigr)\\
&&{}\times\exp\biggl(\lambda
c \biggl(\frac{\sqrt{2}}{2} \biggr)^{r}|\G_{r}| \biggr).
\end{eqnarray*}
Taking $\lambda= \frac{\delta}{2c^{2}(1+\alpha)^{2}r}$, we are led
to
\[
\P\bigl(\overline{M}_{\G_{r}}(f)>\delta\bigr) \leq\exp\biggl(-
\frac{\delta|\G_{r}|}{4c^{2}(1+\alpha)^{2}r} \biggl(\delta- 2c \biggl
(\frac{\sqrt{2}}{2}
\biggr)^{r} \biggr) \biggr).
\]
For all $r\in\N$ such that $r>
\log(\frac{\delta}{4c} )/\log(\frac{\sqrt{2}}{2} )$, we have
$(\delta-
2c (\frac{\sqrt{2}}{2} )^{r}) > \frac{\delta}{2}$ and for
such $r$, it follows that
\[
\P\bigl(\overline{M}_{\G_{r}}(f) > \delta\bigr) \leq\exp\biggl(-
\frac{\delta^{2} |\G_{r}|}{18c^{2}r} \biggr).
\]

(c) If $\alpha^{2}>\frac{1}{2}$, then for all $\lambda>0$,
\begin{eqnarray*}
\mathbb{P} \bigl(\overline{M}_{\mathbb{G}_{r}}(f)>\delta\bigr) &\leq&
\exp\bigl(-
\lambda\delta|\mathbb{G}_{r}| \bigr)\times\exp\biggl(\lambda^{2}c^{2}(1+
\alpha)^{2}\frac{(2\alpha^{2})^{r}-1} {
2\alpha^{2}-1}|\mathbb{G}_{r}| \biggr)\\
&&{}\times
\exp\bigl(\lambda c\alpha^{r}|\mathbb{G}_{r}| \bigr)
\\
&\leq& \exp\biggl(-|\mathbb{G}_{r}| \biggl(\lambda\delta-
\frac{\lambda^{2}c^{2}(1
+ \alpha)^{2}}{2\alpha^{2}-1}\bigl(2\alpha^{2}\bigr)^{r} \biggr)
\biggr)\\
&&{} \times\exp\bigl(\lambda c\alpha^{r}|\mathbb{G}_{r}|
\bigr).
\end{eqnarray*}
Taking
$\lambda=\frac{(2\alpha^{2}-1)\delta}{2c^{2}(1+\alpha)^{2}(2\alpha
^{2})^{r}}$
leads us to
\[
\mathbb{P} \bigl(\overline{M}_{\mathbb{G}_{r}}(f)>\delta\bigr) \leq
\exp
\biggl(-\frac{(2\alpha^{2}-1)\delta}{4c^{2}(1+\alpha
)^{2}\alpha^{2r}} \bigl(\delta- 2c\alpha^{r} \bigr) \biggr).
\]
Now for all $r\in\N$ such that $r>
\log(\frac{\delta}{4c} )/\log\alpha$, we have
\[
\P\bigl(\overline{M}_{\G_{r}}(f) \bigr) \leq\exp\biggl(-
\frac{(2\alpha^{2}-1)\delta^{2}}{8c^{2}(1+\alpha)^{2}\alpha
^{2r}} \biggr).
\]

\textit{Part} 2. Let us now deal with $\overline{M}_{\mathbb
{T}_{r}}(f)$. We have for all $\lambda>0$ and all $\delta>0$,
%
\begin{equation}
\label{chernofineq2} \mathbb{P} \bigl(\overline{M}_{\mathbb
{T}_{r}}(f)>\delta
\bigr) \leq\exp\bigl(-\lambda\delta|\mathbb{T}_{r}| \bigr)\mathbb{E} \biggl[
\exp\biggl(\lambda\sum_{i\in
\mathbb{T}_{r}}f(X_{i})
\biggr) \biggr].
\end{equation}
By subtracting and adding terms, we get
\begin{eqnarray*}
&& \mathbb{E} \biggl[\exp\biggl(\lambda\sum_{i\in
\mathbb{T}_{r}}f(X_{i})
\biggr) \biggr]
\\
&&\qquad =\mathbb{E}\biggl[\mathbb{E}\biggl[ \prod_{i\in
\mathbb{G}_{r-1}}
\exp\bigl(\lambda\bigl(f(X_{2i})+f(X_{2i+1})-2Qf(X_{i})
\bigr) \bigr)\\
&&\hspace*{22.6pt}\qquad\quad{}\times\prod_{i\in
\mathbb{G}_{r-1}}\exp\bigl(2\lambda
Qf(X_{i}) \bigr)
\times\prod_{i\in
\mathbb{T}_{r-1}}\exp\bigl(\lambda
f(X_{i}) \bigr) /\mathcal{F}_{r-1}\biggr]\biggr]
\\
&&\qquad = \mathbb{E}\biggl[\mathbb{E}\biggl[ \prod_{i\in\mathbb{G}_{r-1}}
\exp\bigl(\lambda\bigl(f(X_{2i})+f(X_{2i+1})-2Qf(X_{i})
\bigr) \bigr) \\
&&\hspace*{22.6pt}\qquad\quad{}\times\prod_{i\in
\mathbb{G}_{r-1}}\exp\bigl(
\lambda(f+2Qf) (X_{i}) \bigr)
\times\prod_{i\in
\mathbb{T}_{r-2}}\exp\bigl(\lambda
f(X_{i}) \bigr) /\mathcal{F}_{r-1}\biggr]\biggr].
\end{eqnarray*}
The fact that conditionally to the $(r-1)$ first generations the
sequence $\{\Delta_{i}, i\in\mathbb{G}_{r-1}\}$ is a sequence of
independent random variables and Azuma--Bennett--Hoeffding inequality
(see Lemma \ref{lemABH}) lead us according to (H2)~to
\begin{eqnarray*}
&&
\mathbb{E} \biggl[\exp\biggl(\lambda\sum_{i\in
\mathbb{T}_{r}}f(X_{i})
\biggr) \biggr] \\
&&\qquad\leq \exp\bigl(2\lambda^{2}c^{2}(1+
\alpha)^{2}|\mathbb{G}_{r-1}| \bigr)
\\
&&\qquad\quad{}\times\mathbb{E} \biggl[ \prod_{i\in
\mathbb{G}_{r-1}}\exp\bigl(
\lambda(f+2 Qf) (X_{i}) \bigr)\prod_{i\in\mathbb{T}_{r-2}}
\exp\bigl(\lambda f(X_{i}) \bigr) \biggr].
\end{eqnarray*}
Doing the same things for
\[
\mathbb{E} \biggl[ \prod_{i\in
\mathbb{G}_{r-1}}\exp\bigl(
\lambda(f+2 Qf) (X_{i}) \bigr)\prod_{i\in\mathbb{T}_{r-2}}
\exp\bigl(\lambda f(X_{i}) \bigr) \biggr]
\]
with $f+2Qf$ replacing $f$, we get
\begin{eqnarray*}
&&
\mathbb{E} \biggl[\exp\biggl(\lambda\sum_{i\in
\mathbb{T}_{r}}f(X_{i})
\biggr) \biggr] \\
&&\qquad\leq \exp\bigl(2\lambda^{2}c^{2}(1+
\alpha)^{2}|\mathbb{G}_{r-1}| \bigr)\times\exp\bigl(2
\lambda^{2}c^{2}\bigl(1+3\alpha+2\alpha^{2}
\bigr)^{2}|\mathbb{G}_{r-2}| \bigr)
\\
&&\qquad\quad{}\times\mathbb{E} \biggl[ \prod_{i\in
\mathbb{G}_{r-2}}\exp\bigl(
\lambda\bigl(f+2 Qf+2^{2}Q^{2}f\bigr) (X_{i})
\bigr)\prod_{i\in
\mathbb{T}_{r-3}}\exp\bigl(\lambda
f(X_{i}) \bigr) \biggr].
\end{eqnarray*}
Iterating this procedure leads us to
\begin{eqnarray*}
&&
\mathbb{E} \biggl[\exp\biggl(\lambda\sum_{i\in
\mathbb{T}_{r}}f(X_{i})
\biggr) \biggr] \\
&&\qquad\leq \exp\Biggl(2\lambda^{2}c^{2}(1+
\alpha)^{2}\sum_{q=1}^{r}
\Biggl(\sum_{k=0}^{q-1} (2\alpha
)^{k} \Biggr)^{2}|\mathbb{G}_{r-q}| \Biggr)
\\
&&\qquad\quad{}\times\mathbb{E} \bigl[\exp\bigl(\lambda\bigl(f+2Qf+2^{2}Q^{2}f+
\cdots+2^{r}Q^{r}f \bigr) (X_{1}) \bigr) \bigr].
\end{eqnarray*}
Using (H2) we get
\begin{eqnarray*}
&&
\mathbb{E} \biggl[\exp\biggl(\lambda\sum_{i\in
\mathbb{T}_{r}}f(X_{i})
\biggr) \biggr] \\
&&\qquad\leq\exp\Biggl(\lambda c\sum_{k=0}^r(2
\alpha)^k+2\lambda^{2}c^{2}(1+
\alpha)^{2}\sum_{q=1}^{r}
\Biggl(\sum_{k=0}^{q-1} (2\alpha
)^{k} \Biggr)^{2}|\mathbb{G}_{r-q}| \Biggr).
\end{eqnarray*}
Now for $\alpha\neq\frac{1}{2}$ and $\alpha^{2}\neq\frac{1}{2}$
we have
\begin{eqnarray*}
&&\mathbb{P} \bigl(\overline{M}_{\mathbb{T}_{r}}(f)>\delta\bigr)
\\
&&\qquad \leq\exp\bigl(-\lambda\delta|\mathbb{T}_{r}| \bigr) \exp\biggl(2
\lambda^{2}c^{2}(1+\alpha)^{2} \biggl(
\frac
{2^{r}-1}{(1-2\alpha)^{2}}-\frac{\alpha(1-\alpha^{r})2^{r+1}} {
(1-2\alpha)^{2}(1-\alpha)}\\
&&\qquad\quad\hspace*{194.5pt}{}+\frac{2\alpha^{2}(1-(2\alpha
^{2})^{r})2^{r}}{(1-2\alpha)^{2}(1-2\alpha^{2})} \biggr) \biggr)
\\
&&\qquad\quad{}\times\exp\biggl(\lambda c\frac{1-(2\alpha)^{r+1}}{1-2\alpha}
\biggr)
\\
&&\qquad\leq\exp\biggl(-|\mathbb{T}_{r}| \biggl(\lambda\delta-
\frac{\lambda
^{2}c^{2}(1+\alpha)^{2}}{(1-2\alpha)^{2}} \biggl(1+ \frac{4\alpha
^{2}(1-(2\alpha^{2})^{r})}{1-2\alpha^{2}} \biggr) \biggr) \biggr)\\
&&\qquad\quad{}\times\exp
\biggl(\lambda c\frac{1-(2\alpha)^{r+1}}{1-2\alpha} \biggr).
\end{eqnarray*}
Taking
$\lambda=\frac{\delta}{({2c^{2}(1+\alpha)^{2}}/{(1-2\alpha
)^{2}}) (1+{4\alpha^{2}
(1-(2\alpha^{2})^{r})}/({1-2\alpha^{2}}) )}$ leads us to
\begin{eqnarray*}
&&
\mathbb{P} \bigl(\overline{M}_{\mathbb{T}_{r}}(f)>\delta\bigr) \\
&&\qquad\leq
\exp
\biggl(-|\mathbb{T}_{r}|\frac{(1-2\alpha)^{2}\delta
^{2}}{4c^{2}(1+\alpha)^{2} (1+{4\alpha^{2}
(1-(2\alpha^{2})^{r})}/({1-2\alpha^{2}}) )} \biggr)
\\
&&\qquad\quad{}\times\exp\biggl(\frac{(1-2\alpha)^{2}\delta}{2c(1+\alpha)^{2}
(1+{4\alpha^{2}
(1-(2\alpha^{2})^{r})}/({1-2\alpha^{2}}) )}\frac{1-(2\alpha
)^{r+1}}{1-2\alpha} \biggr).
\end{eqnarray*}
\begin{itemize}
\item If $\alpha<\frac{1}{2}$, then $\frac{1-(2\alpha
^{2})^{r}}{1-2\alpha^{2}}<\frac{1}{1-2\alpha^{2}}$ for all $r\in\N$,
\begin{eqnarray*}
\mathbb{P} \bigl(\overline{M}_{\mathbb{T}_{r}}(f)>\delta\bigr) &\leq&
\exp\biggl(
\frac{1-2\alpha}{2c(1+\alpha)^{2}}\delta\biggr)\\
&&{}\times\exp\biggl
(-\frac{(1-2\alpha^2)(1-2\alpha)^{2}\delta
^{2}}{4c^{2}(1+\alpha)^{2}(1+2\alpha^{2})} |
\mathbb{T}_{r}| \biggr).
\end{eqnarray*}

\item If $\frac{1}{2}<\alpha<\frac{\sqrt{2}}{2}$,
then $\frac{1-(2\alpha^{2})^{r}}{1-2\alpha^{2}}<\frac{1}{1-2\alpha
^{2}}$ for all $r\in\N$,
\begin{eqnarray*}
&&
\P\bigl(\overline{M}_{\T_{r}}(f)>\delta\bigr) \\
&&\qquad\leq\exp\biggl(-
\frac{(1-2\alpha^2)(2\alpha-1)^{2}\delta|\T
_{r}|}{4c^{2}(1+\alpha)^{2}(1+2\alpha^{2})} \biggl(\delta- \frac
{2c(1-2\alpha^2)\alpha^{r+1}}{(2\alpha-1)(1+2\alpha
^{2})} \biggr) \biggr).
\end{eqnarray*}
Now for all $r\in\N$ such that $r+1>
\log(\frac{(2\alpha-1)(1+2\alpha^{2})\delta}{4c(1-2\alpha
^2)} )/\log\alpha$,
we have $
\delta- \frac{2c(1-2\alpha^2)\alpha^{r+1}}{(2\alpha-1)(1+2\alpha
^{2})} >
\frac{\delta}{2}$ so that for such $r$, we have
\[
\P\bigl(\overline{M}_{\T_{r}}(f)>\delta\bigr) \leq\exp\biggl(-
\frac{(1-2\alpha^2)(2\alpha-1)^{2}\delta^{2}|\T
_{r}|}{8c^{2}(1+\alpha)^{2}(1+2\alpha^{2})} \biggr).
\]
\item If $\alpha^{2}>\frac{1}{2}$, then for all $r\geq1$, we have
\begin{eqnarray*}
&&
\P\bigl(\overline{M}_{\T_{r}}(f)>\delta\bigr) \\
&&\qquad\leq\exp\biggl(-
\frac{(2\alpha-1)^{2}(2\alpha^{2}-1)\delta}{32c^{2}(1+\alpha
)^{2}\alpha^{2(r+1)}} \biggl(\delta- \frac{16\alpha^{2}c\alpha
^{r+1}}{(2\alpha^{2}-1)(2\alpha-1)} \biggr) \biggr).
\end{eqnarray*}
For all $r\in\N^{*}$ such that $r+3>
\log(\frac{(2\alpha^{2}-1)(2\alpha-1)\delta}{32c}
)/\log\alpha$,
we have $
\delta-
\frac{16\alpha^{2}c\alpha^{r+1}}{(2\alpha^{2}-1)(2\alpha-1)}>\frac
{\delta}{2}$
so that
\[
\mathbb{P} \bigl(\overline{M}_{\mathbb{T}_{r}}(f)>\delta\bigr) \leq
\exp
\biggl(-\frac{(1-2\alpha)^{2}(2\alpha^{2}-1)\delta^{2}}{64c^{2}(1
+ \alpha)^{2}} \biggl(\frac{1}{\alpha^{2}} \biggr)^{r+1}
\biggr).
\]
\end{itemize}
Now if $\alpha=\frac{1}{2}$, then $\sum_{q=1}^{r}\frac
{q^{2}}{2^{q}}< \sum_{q=1}^{\infty}\frac{q^{2}}{2^{q}}=6$.
Then for all $\lambda>0$,
\[
\mathbb{P} \bigl(\overline{M}_{\mathbb{T}_{r}}(f)>\delta\bigr) \leq
\exp\bigl(-
\bigl(\lambda\delta-27c^{2}\lambda^{2}\bigr)|\mathbb
{T}_{r}| \bigr)\times\exp\bigl(\lambda c(r+1) \bigr).
\]
Taking $\lambda=\frac{\delta}{54c^{2}}$ leads us to
\[
\mathbb{P} \bigl(\overline{M}_{\mathbb{T}_{r}}(f)>\delta\bigr) \leq
\exp
\biggl(-\frac{\delta^{2}}{108c^{2}}|\mathbb{T}_{r}| \biggr)\times\exp
\biggl(
\frac{\delta}{54c}(r+1) \biggr).
\]
Finally, if $\alpha^{2}=\frac{1}{2}$, in the same way as previously,
for all $r\in\N$ such that $ $ $r+1>\log
(\frac{(\sqrt{2}-1)\delta}{4c} )/\log(\frac
{\sqrt{2}}{2} )$,
we have
\[
\P\bigl(\overline{M}_{\T_{r}}(f)>\delta\bigr) \leq\exp\biggl(-
\frac{(\sqrt{2}-1)^{2}\delta^{2}}{4c^{2}(1+\sqrt{2})^{2}}\frac{|\T
_{r}|}{r+1} \biggr).
\]
%

\textit{Part} 3. Eventually, let us look at
$\overline{M}{}^{\Pi}_{n}(f)$. We have for all $\delta>0$
\[
\mathbb{P} \biggl(\frac{1}{n}M_{n}^{\Pi}(f)>\delta
\biggr)\leq\mathbb{P} \biggl(\frac{1}{n}\sum_{i\in\T_{r_{n}-1}}f
(X_{i} )>\frac{\delta}{2} \biggr) + \mathbb{P} \Biggl(
\frac{1}{n}\sum_{i=2^{r_{n}}}^{n}f
(X_{\Pi(i)} )>\frac{\delta}{2} \Biggr).
\]
On the one hand, (\ref{expoineq2}) leads us to
%
\begin{equation}
\label{h}\qquad \mathbb{P} \biggl(\frac{1}{n}\sum_{i\in\T_{r_{n}-1}}
f (X_{i} ) > \frac{\delta}{2} \biggr) \leq\cases{\displaystyle  \exp
\bigl(c''\delta\bigr)\exp\bigl(-c'
\delta^{2}n \bigr),\cr
\qquad\forall n\in\N, \hspace*{24.6pt}\qquad \mbox{if $\displaystyle\alpha<\frac{1}{2}$},
\cr
\displaystyle
\exp\bigl(2c'\delta(r_{n}+1) \bigr)\exp
\bigl(-c'\delta^{2}n \bigr),\cr
\qquad\forall n\in\N, \hspace*{24.6pt}\qquad \mbox{if  $
\displaystyle\alpha=\frac{1}{2}$},
\cr
\displaystyle
\exp\bigl(-c'\delta^{2}n
\bigr),\cr
\qquad\forall r_{n}>r_0, \hspace*{18.45pt}\qquad \mbox{if $\displaystyle\frac{1}{2}<
\alpha<\frac{\sqrt{2}}{2}$},
\cr
\displaystyle
\exp\biggl(-c'\delta^{2}
\frac{n}{r_{n}+1} \biggr),\cr
\qquad\forall r_{n}>r_0, \hspace*{18.45pt}\qquad \mbox{if $\displaystyle
\alpha=\frac{\sqrt{2}}{2}$},
\cr
\displaystyle
\exp\biggl(-c'\delta^{2}
\frac{1}{\alpha^{2(r_n+1)}} \biggr),\cr
\qquad\forall r_{n} >r_0-2, \qquad \mbox{if $\displaystyle
\alpha>\frac{\sqrt{2}}{2}$},}
\end{equation}
where $r_0:=\log(\frac{\delta}{c_{0}} )/\log\alpha$
and $c_{0}$, $c'$ and $c''$ are positive constants which depend on
$\alpha$, $\|f\|_{\infty}$ and $c$. $c_{0}$, $c'$and $c''$ differ
line by line. On the other hand, for all $\lambda>0$,
\[
\mathbb{P} \Biggl(\frac{1}{n}\sum_{i=2^{r_{n}}}^{n}f
(X_{\Pi(i)} )>\frac{\delta}{2} \Biggr)\leq\exp\biggl(-
\frac{\lambda\delta}{2}n \biggr)\mathbb{E} \Biggl[\exp\Biggl(\lambda
\sum
_{i=2^{r_{n}}}^{n}f (X_{\Pi(i)} ) \Biggr) \Biggr].
\]
Now let:
\begin{itemize}
\item$\mathcal{O}_{r_{n}}=\{\Pi(2^{r_{n}}),\Pi(2^{r_{n}}+1),\ldots,\Pi(n)\}$;\vspace*{1pt}
\item$\mathcal{O}_{r_{n}-1}^{1}$ the set\vspace*{1pt} of individuals of generation
$\mathbb{G}_{r_{n}-1}$
which are ancestors of one individual in
$\mathcal{O}_{r_{n}}$;\vspace*{1pt}
\item$\mathcal{O}_{r_{n}-1}^{2}$ the set\vspace*{1pt} of individuals of generation
$\mathbb{G}_{r_{n}-1}$ which are ancestors of two individuals in
$\mathcal{O}_{r_{n}}$;

\item$\mathcal{O}_{r_{n}}^{\prime }$ the set of individuals of $\mathcal
{O}_{r_{n}}$ whose parents belong to
$\mathcal{O}_{r_{n}-1}^{1}$;

\item$\mathcal{O}_{r_{n}-1}=\mathcal{O}_{r_{n}-1}^{1}\cup\mathcal
{O}_{r_{n}-1}^{2}$.
\end{itemize}
We introduce the filtration
$\tilde{\mathcal{F}_r}:=\sigma(\mathcal{F}_r,\Pi(i),1\le i\le\T)$.
Then we have
\begin{eqnarray*}
&&
\mathbb{E} \Biggl[\exp\Biggl(\lambda\sum_{i=2^{r_{n}}}^{n}f
(X_{\Pi(i)} ) \Biggr) \Biggr] \\
&&\qquad= \mathbb{E}\biggl[\exp\biggl(\lambda
\sum
_{i\in\mathcal
{O}_{r_{n}-1}^{2}} 2Qf (X_{i} )+\lambda\sum
_{i\in\mathcal
{O}_{r_{n}-1}^{1}}Qf (X_{i} ) \biggr)
\\
&&\qquad\quad{}\times\mathbb{E} \biggl[\exp\biggl(\lambda\sum
_{i\in\mathcal
{O}_{r_{n}}^{\prime}}f (X_{i} )- Qf (X_{[{i}/{2}]} ) \biggr)
\Big/\tilde{\mathcal{F}}_{r_{n}-1} \biggr]
\\
&&\qquad\quad{}\times\mathbb{E} \biggl[\exp\biggl(\lambda\sum
_{i\in\mathcal
{O}_{r_{n}-1}^{2}}f (X_{2i} )+ f (X_{2i+1} )- 2Qf
(X_{i} ) \biggr) \Big/\tilde{\mathcal{F}}_{r_{n}-1} \biggr]\biggr].
\end{eqnarray*}
Using the Azuma--Bennett--Hoeffding inequality, as in part 1, we
get
\[
\mathbb{E} \biggl[\exp\biggl(\lambda\sum_{i\in\mathcal
{O}_{r_{n}}^{\prime}}f
(X_{i} )- Qf (X_{[{i}/{2}]} ) \biggr)\Big/\tilde{\mathcal
{F}}_{r_{n}-1} \biggr]\leq\exp\biggl(\frac{\lambda^{2}c^{2}(1+\alpha
)^{2}}{2}\bigl|\mathcal
{O}_{r_{n}}^{\prime}\bigr| \biggr)
\]
and
\begin{eqnarray*}
&&\mathbb{E} \biggl[\exp\biggl(\lambda\sum_{i\in\mathcal
{O}_{r_{n}-1}^{2}}f
(X_{2i} )+ f (X_{2i+1} )- 2Qf (X_{i} ) \biggr)\Big/
\tilde{\mathcal{F}}_{r_{n}-1} \biggr] \\
&&\qquad\leq\exp\bigl(2\lambda^{2}c^{2}(1+
\alpha)^{2}\bigl|\mathcal{O}_{r_{n}-1}^{2}\bigr| \bigr).
\end{eqnarray*}
Now, we have
\begin{eqnarray*}
&& \exp\biggl(\frac{\lambda^{2}c^{2}(1+\alpha)^{2}}{2}\bigl| \mathcal
{O}_{r_{n}}^{\prime}\bigr|
\biggr) +\exp\bigl(2\lambda^{2}c^{2}(1+
\alpha)^{2}\bigl|\mathcal{O}_{r_{n}-1}^{2}\bigr| \bigr)
\\
&&\qquad= \exp\biggl(\lambda^{2}c^{2}(1+\alpha)^{2}
\biggl(2\bigl|\mathcal{O}_{r_{n}-1}^{2}\bigr|+\frac{|\mathcal{O}_{r_{n}}^{\prime}|} {
2} \biggr)
\biggr)
\\
&&\qquad\leq\exp\bigl(\lambda^{2}c^{2}(1+\alpha)^{2}n
\bigr).
\end{eqnarray*}
This leads us to
\begin{eqnarray*}
&& \mathbb{E} \Biggl[\exp\Biggl(\lambda\sum_{i=2^{r_{n}}}^{n}f
(X_{\Pi(i)} ) \Biggr) \Biggr]
\\
&&\qquad \leq\exp\bigl(\lambda^{2}c^{2}(1+
\alpha)^{2}n \bigr)\mathbb{E} \biggl[\exp\biggl(\lambda\sum
_{i\in\mathcal{O}_{r_{n}-1}^{2}} 2Qf (X_{i} )+\lambda\sum
_{i\in\mathcal
{O}_{r_{n}-1}^{1}}Qf (X_{i} ) \biggr) \biggr].
\end{eqnarray*}
Now let:
\begin{itemize}
\item$\mathcal{O}_{r_{n}-2}^{1,1}$ the set of individuals of
$\mathbb{G}_{r_{n}-2}$ which are ancestors of one individual in
$\mathcal{O}_{r_{n}-1}$ and one individual in
$\mathcal{O}_{r_{n}}$;\vspace*{1pt}

\item$\mathcal{O}_{r_{n}-2}^{1,2}$ the set of individuals of
$\mathbb{G}_{r_{n}-2}$ which are ancestors of one individual in
$\mathcal{O}_{r_{n}-1}$ and two individuals in
$\mathcal{O}_{r_{n}}$;\vspace*{1pt}

\item$\mathcal{O}_{r_{n}-2}^{2,2}$ the set of individuals of
$\mathbb{G}_{r_{n}-2}$ which are ancestors of two individuals in
$\mathcal{O}_{r_{n}-1}$ and two individuals in
$\mathcal{O}_{r_{n}}$;\vspace*{1pt}

\item$\mathcal{O}_{r_{n}-2}^{2,3}$ the set of individuals of
$\mathbb{G}_{r_{n}-2}$ which are ancestors of two individuals in
$\mathcal{O}_{r_{n}-1}$ and three individuals in
$\mathcal{O}_{r_{n}}$;\vspace*{1pt}

\item$\mathcal{O}_{r_{n}-2}^{2,4}$ the set of individuals of
$\mathbb{G}_{r_{n}-2}$ which are ancestors of two individuals in
$\mathcal{O}_{r_{n}-1}$ and four individuals in
$\mathcal{O}_{r_{n}}$;\vspace*{1pt}

\item$\mathcal{O}_{r_{n}-1}^{\prime}$ the set of individuals of
$\mathcal{O}_{r_{n}-1}$ whose parents belong to
$\mathcal{O}_{r_{n}-2}^{1,1}$;

\item$\mathcal{O}_{r_{n}-1}^{\prime \prime}$ the set of individuals of
$\mathcal{O}_{r_{n}-1}$ whose parents belong to
$\mathcal{O}_{r_{n}-2}^{1,2}$.
\end{itemize}
Then we have
\begin{eqnarray*}
&&
\mathbb{E} \biggl[\exp\biggl(\lambda\sum_{i\in\mathcal
{O}_{r_{n}-1}^{2}} 2Qf
(X_{i} )+\lambda\sum_{i\in\mathcal
{O}_{r_{n}-1}^{1}}Qf
(X_{i} ) \biggr) \biggr]\\
&&\qquad= \mathbb{E} [I_1\times
I_2\times I_3\times I_4\times
I_5 \times I_6 \times I_7 ],
\end{eqnarray*}
where
\begin{eqnarray*}
I_1 &=& \exp\biggl(\lambda\sum_{i\in\mathcal
{O}_{r_{n}-2}^{1,1}}Q^{2}f(X_{i})
+ \lambda\sum_{i\in\mathcal{O}_{r_{n}-2}^{1,2}}2Q^{2}f(X_{i})
+ \lambda\sum_{i\in\mathcal{O}_{r_{n}-2}^{2,2}}2Q^{2}f(X_{i})
\\
&&\hspace*{105.6pt}{} + \lambda\sum_{i\in\mathcal{O}_{r_{n}-2}^{2,3}}3Q^{2}f(X_{i})
+ \lambda\sum_{i\in\mathcal
{O}_{r_{n}-2}^{2,4}}4Q^{2}f(X_{i})
\biggr),
\\
I_2&=&\mathbb{E} \biggl[\exp\biggl(\lambda\sum
_{i\in\mathcal
{O}_{r_{n}-1}^{\prime}} Qf(X_{i})-Q^{2}f(X_{[{i}/{2}]})
\biggr) \Big/\tilde{\mathcal{F}}_{r_{n}-2} \biggr],
\\
I_3&=&\mathbb{E} \biggl[\exp\biggl(2\lambda\sum
_{i\in
\mathcal{O}_{r_{n}-1}^{\prime \prime}} Qf(X_{i})-Q^{2}f(X_{[{i}/{2}]})
\biggr) \Big/\tilde{\mathcal{F}}_{r_{n}-2} \biggr],
\\
I_4&=&\mathbb{E} \biggl[\exp\biggl(\lambda\sum
_{i\in\mathcal
{O}_{r_{n}-1}^{2,2}} Qf(X_{2i})+Qf(X_{2i+1})-2Q^{2}f(X_{i})
\biggr) \Big/\tilde{\mathcal{F}}_{r_{n}-2} \biggr],
\\
I_5&=&\mathbb{E} \biggl[\exp\biggl(\frac{\lambda}{2}\sum
_{i\in\mathcal{O}_{r_{n}-1}^{2,3}} 2Qf(X_{2i})+Qf(X_{2i+1})-3Q^{2}f(X_{i})
\biggr) \Big/\tilde{\mathcal{F}}_{r_{n}-2} \biggr],
\\
I_6&=&\mathbb{E} \biggl[\exp\biggl(\frac{\lambda}{2}\sum
_{i\in\mathcal{O}_{r_{n}-1}^{2,3}} Qf(X_{2i})+2Qf(X_{2i+1})-3Q^{2}f(X_{i})
\biggr) \Big/\tilde{\mathcal{F}}_{r_{n}-2} \biggr],
\\
I_7&=&\mathbb{E} \biggl[\exp\biggl(\lambda\sum
_{i\in\mathcal
{O}_{r_{n}-1}^{2,4}} 2Qf(X_{2i})+2Qf(X_{2i+1})-4Q^{2}f(X_{i})
\biggr) \Big/\tilde{\mathcal{F}}_{r_{n}-2} \biggr].
\end{eqnarray*}
Using the Azuma--Bennett--Hoeffding inequality, we get
\begin{eqnarray*}
&& I_2\times I_3\times I_4\times
I_5 \times I_6 \times I_7
\\
&&\qquad \leq\exp\biggl(\lambda^{2}c^{2}\bigl(\alpha+
\alpha^{2}\bigr)^{2} \biggl(\frac
{|\mathcal{O}_{r_{n}-1}^{\prime }|}{2}+ 2\bigl|
\mathcal{O}_{r_{n}-1}^{\prime \prime}\bigr|+2\bigl|\mathcal{O}_{r_{n}-1}^{2,2}\bigr|\\
&&\qquad\quad\hspace*{127.4pt}{}+
\frac
{9|\mathcal{O}_{r_{n}-1}^{2,3}|}{2} + 8\bigl|\mathcal{O}_{r_{n}-1}^{2,4}\bigr|
\biggr) \biggr)
\\
&&\qquad\leq\exp\bigl(2\lambda^{2}c^{2}\bigl(\alpha+
\alpha^{2}\bigr)^{2}n \bigr),
\end{eqnarray*}
hence
\[
\mathbb{E} \Biggl[\exp\Biggl(\lambda\sum_{i=2^{r_{n}}}^{n}f
(X_{\Pi(i)} ) \Biggr) \Biggr]\leq\exp\bigl(\lambda^{2}c^{2}(1+
\alpha)^{2}n \bigr) \exp\bigl(2\lambda^{2}c^{2}
\bigl(\alpha+\alpha^{2}\bigr)^{2}n \bigr) \mathbb{E}
[I_1 ].
\]
Now, iterating this procedure we get
\[
\mathbb{E} \Biggl[\exp\Biggl(\lambda\sum_{i=2^{r_{n}}}^{n}f
(X_{\Pi(i)} ) \Biggr) \Biggr]\leq\exp\Biggl(\lambda^{2}c^{2}(1+
\alpha)^{2}n\sum_{p=0}^{r_{n}}
\bigl(2\alpha^{2}\bigr)^{p} \Biggr)\exp\bigl(\lambda c
\alpha^{r_{n}}n \bigr).
\]
Then it follows as in part 1 that
%
\begin{eqnarray}
\label{hh}
&&\mathbb{P} \Biggl(\frac{1}{n}\sum_{i=2^{r_{n}}}^{n}f
(X_{\Pi(i)} )>\frac{\delta}{2} \Biggr)\nonumber\\[-8pt]\\[-8pt]
&&\qquad\leq\cases{\displaystyle  \exp
\bigl(c''\delta\bigr)\exp\bigl(-c'
\delta^{2}n \bigr),\cr
\qquad\forall n\in\N, &\quad if $\displaystyle\alpha\leq
\frac{1}{2}$,
\cr
\displaystyle
\exp\bigl(-c'\delta^{2}n
\bigr),\cr
\qquad\forall n \in\N\mbox{ such that }r_{n}>r_0, &\quad if $\displaystyle
\frac{1}{2}<\alpha<\frac{\sqrt{2}}{2}$,
\cr
\displaystyle
\exp\biggl(-c'
\delta^{2}\frac{n}{r_{n}} \biggr), \cr
\qquad\forall n\in\N\mbox{ such that }
r_{n}>r_0, &\quad if $\displaystyle\alpha^{2}=
\frac{1}{2}$,
\cr
\displaystyle
\exp\biggl(-c'\delta^{2}
\biggl(\frac{1}{\alpha} \biggr)^{2r_{n}} \biggr), \cr
\qquad\forall n \in\N
\mbox{ such that }r_{n}>r_0, &\quad if $\displaystyle\alpha^{2}>
\frac{1}{2}$,}\nonumber
\end{eqnarray}
where $r_0:=\log(\frac{\delta}{c_{0}}) /\log(\alpha)$ and the
positive constants $c_{0}$, $c'$ and $c''$ depend on
$\alpha$, $\delta$, $c$ and differ line to line. Finally (\ref{h})
and (\ref{hh}) lead us to (\ref{expoineq3}).

\subsection{\texorpdfstring{Proof of Theorem \protect\ref{expoprobaineq2}}
{Proof of Theorem 3.2}} \label{proofexpoprobaineq2}
Let $f\in\mathcal{B}_b (S^{3} )$ such that $(\mu, Pf)=0$.

\textit{Part} 1. Let us first deal with $\overline{M}_{\mathbb
{G}_{r}}(f)$. We have for all $\delta>0$ and $\lambda>0$,
\[
\mathbb{P} \bigl(\overline{M}_{\mathbb{G}_{r}}(f)>\delta\bigr)\leq
\exp\bigl(-
\lambda\delta|\mathbb{G}_{r}| \bigr) \mathbb{E} \biggl[\exp\biggl(\lambda
\sum_{i\in\mathbb
{G}_{r}}f(\Delta_{i}) \biggr) \biggr].
\]
Conditioning and using Bennett--Hoeffding inequality gives us
\[
\mathbb{E} \biggl[\exp\biggl(\lambda\sum_{i\in\mathbb
{G}_{r}}f(
\Delta_{i}) \biggr) \biggr] \leq\exp\bigl(2\lambda^{2}\|f
\|_{\infty}|\mathbb{G}_{r}| \bigr)\mathbb{E} \biggl[\exp\biggl(
\lambda\sum_{i\in\mathbb{G}_{r}}Pf(X_{i}) \biggr)
\biggr].
\]
Now, applying part 1 of the proof of the Theorem
\ref{expoprobaineq1} to $Pf$, we get (\ref{expoineq1}) for $f\in
\mathcal{B}_b (S^{3} )$.\vspace*{9pt}

\textit{Part} 2. Let us now treat $\overline{M}_{\mathbb
{T}_{r}}(f)$. We have for all $\delta>0$,
%
\begin{equation}
\label{chernofineq3}\quad  \mathbb{P} \bigl(\overline{M}_{\mathbb
{T}_{r}}(f)>\delta
\bigr)\leq\mathbb{P} \biggl(\overline{M}_{\mathbb{T}_{r}}(f-Pf)>
\frac{\delta
}{2} \biggr) + \mathbb{P} \biggl(\overline{M}_{\mathbb{T}_{r}}(Pf)>
\frac{\delta
}{2} \biggr).
\end{equation}
Now, since $ (M_{n}^{\Pi}(f-Pf) )_{n\geq1}$ is a
$\mathcal{H}_{n}$-martingale with bounded jumps, the Azuma inequality
\cite{Azuma} gives us for some positive constant $c'$,
\[
\mathbb{P} \biggl(\overline{M}_{\mathbb{T}_{r}}(f-Pf)>\frac{\delta
}{2} \biggr)
\leq\exp\bigl(-c'\delta^{2}|\mathbb{T}_{r}|
\bigr).
\]
For the second term on the right-hand side of (\ref{chernofineq3}),
we use inequalities (\ref{expoineq2}) with $Pf$ instead of $f$.
Gathering these inequalities, we get (\ref{expoineq2}) for all $r$
large enough.\vspace*{9pt}

\textit{Part} 3. The proof for the case $\overline{M}{}^{\Pi}_{n}(f)$
follows the same lines as the proof of part 2.

\subsection{\texorpdfstring{Proof of Proposition \protect\ref{devineqtheta}}
{Proof of Proposition 4.2}}\label{proofdevineqtheta}
We will prove the deviation inequality for
$|\hat{\alpha}_{0}^{r}-\alpha_0|$. The other deviation inequalities for
$|\hat{\beta}_{0}^{r}-\beta_0|,|\hat{\alpha}_{1}^{r}-\alpha_1|$ and
$|\hat{\beta}_{1}^{r}-\beta_1|$ may be treated in a similar way.

One easily checks that
\[
\hat{\alpha}_{0}^{r} - \alpha_{0} =
\frac{ (\overline{M}_{\T_{r}}(\mathbf{xy}) -
\overline{M}_{\T_{r}}(P(\mathbf{xy})) ) -
(\overline{M}_{\T_{r}}(\mathbf{x}) ) (\overline
{M}_{\T_{r}}(\mathbf{y})
- \overline{M}_{\T_{r}}(P(\mathbf{y})) )}{B_{r}}.
\]
We then have, for all $\delta>0$,
\begin{eqnarray*}
&&
\mathbb{P} \bigl(\bigl\llvert\hat{\alpha}_{0}^{r}-
\alpha_{0}\bigr\rrvert>\delta\bigr)\\
&&\qquad\leq\mathbb{P} \biggl(
\frac{\llvert\overline{M}_{\mathbb{T}_{r}}(\mathbf
{xy}-P(\mathbf{xy}))\rrvert}{B_{r}} >\frac{\delta}{2} \biggr)\\
&&\qquad\quad{} + \mathbb
{P} \biggl(
\frac{\llvert\overline
{M}_{\mathbb{T}_{r}}(\mathbf{x})\rrvert
\llvert\overline{M}_{\mathbb{T}_{r}}
(\mathbf{y}-P(\mathbf{y}))\rrvert}{B_{r}} >\frac{\delta}{2}
\biggr).\nonumber
\end{eqnarray*}
On one hand, for all $\gamma_{1}>0$ we have
%
\begin{eqnarray}
\label{i1}
&&\mathbb{P} \biggl(\frac{\llvert\overline{M}_{\mathbb
{T}_{r}}(\mathbf
{xy}-P(\mathbf{xy}))\rrvert}{B_{r}} >\frac{\delta}{2} \biggr)\nonumber\\[-8pt]\\[-8pt]
&&\qquad\leq
\mathbb{P} (B_{r}<\gamma_{1} ) + \P\biggl(\bigl\llvert
\overline{M}_{\mathbb{T}_{r}} \bigl(\mathbf{xy} - P(\mathbf{xy})\bigr
)\bigr\rrvert
>\frac{\delta\gamma_{1}}{2} \biggr).\nonumber
\end{eqnarray}
Now, for $b=\mu_{2}(\theta,\sigma^{2})-\mu_{1}(\theta)^{2}$, where
$\mu_1$ and $\mu_2$ are given in (\ref{mu}), we have
\begin{eqnarray*}
\mathbb{P} (B_{r}<\gamma_{1} )&\leq& \mathbb{P} \biggl(-
\overline{M}_{\mathbb{T}_{r}}\bigl(\mathbf{x}^{2}-\mu_{2}
\bigr) > \frac{b-\gamma_{1}}{3} \biggr) \\
&&{}+ \mathbb{P} \biggl(\bigl
\llvert
\overline{M}_{\mathbb{T}_{r}}(\mathbf{x} - \mu_{1})\bigr\rrvert>
\frac{\sqrt{b - \gamma_{1}}}{\sqrt{3} } \biggr)
\\
&&{} + \mathbb{P} \biggl(\overline{M}_{\mathbb{T}_{r}}(\mathbf{x}-
\mu_{1}) > \frac{b - \gamma_{1}}{6|\mu_{1}|} \biggr).
\end{eqnarray*}
We choose $\gamma_{1} < \min\{\frac{2b}{2+3\delta}, \frac{-4 +
\sqrt{48b\delta^{2}+16}}{6\delta^{2}},
\frac{b}{1+3\delta|\mu_{1}|} \}$ so that
$\frac{\delta\gamma_{1}}{2} < \max\{\frac{b-\gamma_{1}}{3},\allowbreak
\frac{\sqrt{b - \gamma_{1}}}{\sqrt{3}}, \frac{b -
\gamma_{1}}{6|\mu_{1}|} \}$. Then we have
\[
\P(B_{r} < \gamma_{1} ) \leq\P\biggl(
\overline{M}_{\T_{r}}\bigl(\mu_{2} - \mathbf{x}^{2}
\bigr) > \frac{\delta\gamma_{1}}{2} \biggr) + 2\P\biggl(\bigl\llvert
\overline{M}_{\T_{r}}(\mathbf{x} - \mu_{1})\bigr\rrvert>
\frac{\delta\gamma_{1}}{2} \biggr),
\]
and therefore we get
\begin{eqnarray*}
&&
\P\biggl(\frac{\llvert\overline{M}_{\mathbb{T}_{r}}(\mathbf{xy} -
P(\mathbf{xy}))\rrvert}{B_{r}} > \frac{\delta}{2} \biggr) \\
&&\qquad\leq2\P
\biggl(\bigl
\llvert\overline{M}_{\T
_{r}}(\mathbf{x} - \mu_{1})\bigr\rrvert
> \frac{\delta\gamma_{1}}{2} \biggr) + \P\biggl(\overline{M}_{\T
_{r}}\bigl(
\mu_{2} - \mathbf{x}^{2}\bigr) > \frac{\delta\gamma_{1}}{2} \biggr)
\\
&&\qquad\quad{} + \P\biggl(\bigl\llvert\overline{M}_{\mathbb{T}_{r}}
\bigl(
\mathbf{xy} - P(\mathbf{xy})\bigr)\bigr\rrvert>\frac{\delta\gamma
_{1}}{2} \biggr).
\end{eqnarray*}
On the other hand, we have
\begin{eqnarray*}
\P\biggl(\frac{\llvert\overline{M}_{\mathbb{T}_{r}}(\mathbf
{x})\rrvert
\llvert\overline{M}_{\mathbb{T}_{r}}(\mathbf{y}-P(\mathbf
{y}))\rrvert}{B_{r}} > \frac{\delta}{2} \biggr) &\leq&\P\biggl(
\frac{\llvert\overline{M}_{\T_{r}}(\mathbf{x} -
\mu_{1})\rrvert \llvert\overline{M}_{\T_{r}}(\mathbf{y} -
P(\mathbf{y}))\rrvert}{B_{r}} > \frac{\delta}{4} \biggr)
\\
&&{}+ \P\biggl(\frac{\llvert\overline{M}_{\T_{r}}(\mathbf{y} -
P(\mathbf{y}))\rrvert}{B_{r}} > \frac{\delta}{4|\mu_{1}|} \biggr).
\end{eqnarray*}
The last term of the previous inequality can be dealt with in the
same way as inequality (\ref{i1}), using $\gamma_{3}>0$ such that
\[
\gamma_{3} < \min\biggl\{\frac{4b|\mu_{1}|}{4|\mu_{1}|+3\delta},
\frac{2|\mu_{1}| (-4+\sqrt{{24b\delta^{2}}/{|\mu
_{1}|}+16} )}{3\delta^{2}},
\frac{2b}{2+3\delta} \biggr\}.
\]
For the second term, we have
\begin{eqnarray*}
&&
\P\biggl(\frac{\llvert\overline{M}_{\T_{r}}(\mathbf{x} -
\mu_{1})\rrvert \llvert\overline{M}_{\T_{r}}(\mathbf{y} -
P(\mathbf{y}))\rrvert}{B_{r}} > \frac{\delta}{4} \biggr) \\
&&\qquad\leq\P\biggl
(\bigl
\llvert\overline{M}_{\T_{r}}(\mathbf{x} - \mu_{1})\bigr\rrvert
> \frac{\sqrt{\delta}}{2} \biggr)
+ \P\biggl(\frac{\llvert\overline{M}_{\T_{r}}(\mathbf{y} -
P(\mathbf{y}))\rrvert}{B_{r}} > \frac{\sqrt{\delta}}{2} \biggr).
\end{eqnarray*}
Let $\gamma_{2}>0$ such that $ \gamma_{2} <
\min\{\frac{2b}{2+3\sqrt{\delta}},
\frac{-4+\sqrt{48b\delta+16}}{b\delta},
\frac{b}{1+3\sqrt{\delta}|\mu_{1}|} \}$, in such a way that we
obtain $ \frac{\gamma_{2}\sqrt{\delta}}{2} <
\max\{\frac{b-\gamma_{2}}{3},
\frac{\sqrt{b-\gamma_{2}}}{\sqrt{3}},
\frac{b-\gamma_{2}}{6|\mu_{1}|} \}. $ We thus have
\begin{eqnarray*}
&&
\P\biggl(\frac{\llvert\overline{M}_{\T_{r}}(\mathbf{x} -
\mu_{1})\rrvert \llvert\overline{M}_{\T_{r}}(\mathbf{y} -
P(\mathbf{y}))\rrvert}{B_{r}} > \frac{\delta}{4} \biggr) \\
&&\qquad\leq\P\biggl
(\bigl
\llvert\overline{M}_{\T_{r}}(\mathbf{x} - \mu_{1})\bigr\rrvert
> \frac{\sqrt{\delta}}{2} \biggr)
\\
&&\qquad\quad{}+ \P\biggl(\bigl\llvert\overline{M}_{\T_{r}}\bigl(\mathbf{x}^{2}
- \mu_{2}\bigr)\bigr\rrvert> \frac{\gamma_{2}\sqrt{\delta}}{2} \biggr
) + \P
\biggl(\bigl\llvert\overline{M}_{\T_{r}}\bigl(\mathbf{y} - P(\mathbf{y})
\bigr)\bigr\rrvert> \frac{\gamma_{2}\sqrt{\delta}}{2} \biggr)
\\
&&\qquad\quad{}+ 2\P\biggl(\bigl\llvert\overline{M}_{\T_{r}}(\mathbf{x}-
\mu_{1})\bigr\rrvert> \frac{\gamma_{2}\sqrt{\delta}}{2} \biggr).
\end{eqnarray*}
From the foregoing, we deduce that for all $\gamma>0$ such that
$\gamma< \min(\gamma_{1}, \gamma_{2}, \gamma_{3})$,
\begin{eqnarray*}
&&
\P\bigl(\bigl|\hat{\alpha}_{0}^{(r)}-\alpha_0\bigr|>
\delta\bigr) \\
&&\qquad\leq 2\P\biggl(\bigl\llvert\overline{M}_{\T_{r}}(\mathbf{x}
- \mu_{1})\bigr\rrvert> \frac{\delta\gamma}{2} \biggr) + \P\biggl(
\overline{M}_{\T_{r}}\bigl(\mu_{2} - \mathbf{x}^{2}
\bigr) > \frac{\delta\gamma}{2} \biggr)
\\
&&\qquad\quad{} + \P\biggl(\bigl\llvert\overline{M}_{\mathbb{T}_{r}} \bigl
(\mathbf{xy} -
P(\mathbf{xy})\bigr)\bigr\rrvert>\frac{\delta\gamma}{2} \biggr) + \P
\biggl(\bigl
\llvert\overline{M}_{\T_{r}}(\mathbf{x} - \mu_{1})\bigr\rrvert
> \frac{\sqrt{\delta}}{2} \biggr)
\\
&&\qquad\quad{} + \P\biggl(\bigl\llvert\overline{M}_{\T_{r}}\bigl(
\mathbf{x}^{2} - \mu_{2}\bigr)\bigr\rrvert>
\frac{\gamma\sqrt{\delta}}{2} \biggr) + \P\biggl(\bigl\llvert\overline
{M}_{\T_{r}}
\bigl(\mathbf{y} - P(\mathbf{y})\bigr)\bigr\rrvert> \frac{\gamma\sqrt
{\delta}}{2} \biggr)
\\
&&\qquad\quad{} + 2\P\biggl(\bigl\llvert\overline{M}_{\T_{r}}(\mathbf{x}-
\mu_{1})\bigr\rrvert> \frac{\gamma\sqrt{\delta}}{2} \biggr) + 2\P
\biggl(\bigl
\llvert\overline{M}_{\T_{r}}(\mathbf{x} - \mu_{1})\bigr\rrvert
> \frac{\delta\gamma}{4|\mu_{1}|} \biggr)
\\
&&\qquad\quad{} + \P\biggl(\biggl\llvert\overline{M}_{\T_{r}}\bigl(
\mu_{2} - \mathbf{x}^{2}\bigr) > \frac{\delta\gamma}{4|\mu_{1}|}\biggr
\rrvert\biggr) + \P\biggl(\bigl\llvert\overline{M}_{\mathbb{T}_{r}}
\bigl(
\mathbf{y} - P(\mathbf{y})\bigr)\bigr\rrvert>\frac{\delta\gamma}{4|\mu
_{1}|} \biggr).
\end{eqnarray*}
Now, using (\ref{4ordercontrolT}) and Markov's inequality we get
%
\begin{eqnarray*}
\mathbb{P} \biggl(\bigl\llvert\overline{M}_{\mathbb{T}_{r}} \bigl
(\mathbf{xy}-P(
\mathbf{xy})\bigr)\bigr\rrvert>\frac{\delta\gamma}{2} \biggr)&\leq&
\frac{c}{\delta^{4}\gamma^{4}} \biggl(\frac{1}{4} \biggr)^{r+1},
\\
\mathbb{P} \biggl(\bigl\llvert\overline{M}_{\mathbb{T}_{r}} \bigl
(\mathbf{y}-P(
\mathbf{y})\bigr)\bigr\rrvert>\frac{\delta
\gamma}{4|\mu_{1}|} \biggr)&\leq&
\frac{c\mu_{1}^{4}}{\delta^{4}\gamma^{4}} \biggl(\frac{1}{4} \biggr)^{r+1}
\end{eqnarray*}
and
\[
\mathbb{P} \biggl(\bigl\llvert\overline{M}_{\mathbb{T}_{r}} \bigl
(\mathbf{y}-P(
\mathbf{y})\bigr)\bigr\rrvert>\frac{\gamma\sqrt{\delta
}}{2} \biggr)\leq\frac{c}{\delta^{2}\gamma^{4}}
\biggl(\frac{1}{4} \biggr)^{r+1},
\]
where the constant $c$ can be found as in Remark \ref{remcontrolT}.

Finally, the other terms, that is, the terms related to
$\overline{M}_{\mathbb{T}_{r}}(\mathbf{x}^{2}-\mu_{2})$ and
$\overline{M}_{\mathbb{T}_{r}}(\mathbf{x}-\mu_{1})$, can be bounded
as in Corollary \ref{4ordermomentcontrol2} and this completes the proof.

\section{}\label{appendixB}
Let us gather here, for the convenience of the readers, various
theorems useful to establish LIL, ASFCLT, deviation inequalities and
MDP.

First, let us enunciate the Azuma--Bennett--Hoeffding inequality
\cite{Azuma,Bennett,Hoeffding}.

\begin{lemma}\label{lemABH}
Let $X$ be a real-valued and centered random variable such that
$ a\leq X\leq b$ a.s., with $a<b$. Then for all
$\lambda>0$, we have
\[
\E\bigl[\exp(\lambda X ) \bigr] \leq\exp\biggl(\frac{\lambda
^{2}(b-a)^{2}}{8} \biggr).
\]
\end{lemma}

\begin{lemma}\label{exponentialmapping}
Let $(E,d)$ a metric space. Let $(Z_{n})$ a sequence of random
variables values in $E$, $(v_{n})$ a rate and
$g\dvtx \mathcal{D}_{E}\subset E \rightarrow\mathbb{R}$ continuous. Let
$z\in E$ be a deterministic value:
\[
\mbox{If } Z_n\stackrel{\mathrm{superexp}}{\underset{v_n}{\Longrightarrow}}
z\qquad \mbox{then } g(Z_n) \stackrel{\mathrm{superexp}}{\underset{v_n}
{\Longrightarrow}} g(z).
\]
\end{lemma}

\begin{pf}
For all $\delta>0$, there exists (see, e.g., \cite{VanDerVaart},
proof of Theorem 2.3) $\alpha_{0}(\delta)>0$
%
\begin{equation}
\label{inevan} \mathbb{P} \bigl( \bigl|g(Z_{n})-g(z) \bigr|>
\delta\bigr)\leq\mathbb{P} \bigl(d(Z_{n},z)>\alpha_{0}(
\delta) \bigr).
\end{equation}
Indeed, since $g$ is continuous, for all $\delta>0$, there exists
$\alpha_{0}(\delta) >0$ such that
\[
\bigl|g(x)-g(z)\bigr|\leq\delta\qquad\mbox{whenever } d(x,z)\leq\alpha_{0}(\delta).
\]
We then have
\[
\bigl\{\omega\dvtx  d\bigl(Z_{n}(\omega),z\bigr)\leq\alpha_{0}(
\delta) \bigr\} \subset\bigl\{\omega\dvtx  \bigl|g\bigl(Z_{n}(\omega)
\bigr)-g(z)\bigr|\leq\delta\bigr\}
\]
and therefore inequality (\ref{inevan}). Now, the result of the lemma
follows since $Z_n\stackrel{\mathrm{superexp}}{\underset{v_n}{\Longrightarrow}} z$.
\end{pf}

Let $M=(M_{n}, \mathcal{H}_{n}, n\geq0)$ be a centered square
integrable martingale defined on a probability space $(\Omega,
\mathcal{H}, \mathbb{P})$ and $(\langle M\rangle_{n})$ its bracket.
We recall some limit theorems for martingale used intensively in
this paper.

We recall the following result due to W. F. Stout (Theorem 3 in \cite{Stout}).

\begin{theorem} \label{thmStoutLIL} Let $(M_n)$ such that $M_{0}=0$.
If $\langle M\rangle_{n} \rightarrow\infty$
a.s. and
\begin{eqnarray*}
&&\sum_{n=1}^{\infty}\frac{2\log\log\langle M\rangle
_{n}}{K_{n}^{2}\langle M\rangle_{n}}
\mathbb{E} \bigl[ (M_{n}-M_{n-1} )^{2}
\mathbf{1}_{ \{
(M_{n}-M_{n-1})^{2}>{K_{n}^{2}\langle
M_{n}\rangle}/({2\log\log\langle M\rangle_{n}}) \}} / \mathcal{H}_{n-1}
\bigr]\\
&&\qquad<\infty \qquad
\mbox{a.s.},
\end{eqnarray*}
where $K_{n}$ are $\mathcal{H}_{n-1}$ measurable and
$K_{n}\rightarrow0$ a.s., then\break  $\limsup\frac{M_{n}}{\sqrt{2\langle
M\rangle_{n}\log\log\langle M\rangle_{n}}}=1$ a.s.
\end{theorem}
We recall the following result due to Chaabane (Corollary 2.2 in \cite
{Chaabane}).

\begin{theorem}\label{thmchaabanetlcpsf} Let $(V_{n})$ be a
$(\mathcal{H}_{n})$-predictable
increasing process such that:
\begin{longlist}[H-3]
\item[H-1] $V_{n}^{-2}\langle M\rangle_{n}\underset{n\rightarrow
\infty}{\longrightarrow} 1$, a.s.;

\item[H-2] for all $\varepsilon>0$, $\sum_{n\geq
1}V_{n}^{-2}\mathbb{E} [
(M_{n}-M_{n-1} )^{2}\mathbf
{1}_{|M_{n}-M_{n-1}|>\varepsilon
V_{n}} /\mathcal{H}_{n-1} ]<\infty$, a.s.;

\item[H-3] for some $a>1$, $\sum_{n\geq 1}V_{n}^{-2a}\mathbb{E}
[(M_{n}-M_{n-1} )^{2a}\mathbf{1}_{|M_{n}-M_{n-1}|\leq V_{n}}
/\mathcal{H}_{n-1} ]<\infty$, a.s.
\end{longlist}
Then $M_n$ satisfies an ASFCLT; that is, for almost all $\omega$, the
weighted random measures
\[
W_{N}(\omega, \bullet)=\bigl(\log V_{N}^{2}
\bigr)^{-1}\sum_{n=1}^{N}
\biggl(1-\frac{V_{n}^{2}}{V_{n+1}^{2}} \biggr) \delta_{\{\psi_{n}(\omega
)\in\bullet\}}
\]
associated to the
continuous processes $\Psi_{n}(\omega)=\{\Psi_{n}(\omega,t),0\leq
t\leq1\}$ defined by
\[
\Psi_{n}(\omega,t)=V_{n}^{-1}\bigl\{
M_{k}+\bigl(V_{k+1}^{2}-V_{k}^{2}
\bigr)^{-1}\bigl(tV_{n}^{2}-V_{k}^{2}
\bigr) (M_{k+1}-M_{k})\bigr\},
\]
when $V_{k}^{2}\leq tV_{n}^{2}<V_{k+1}^{2}$, $0\leq k\leq n-1$,
weakly converge to the Wiener measure on
$\mathcal{C}([0,1],\mathbb{R})$.
\end{theorem}

Let us enunciate the following which corresponds to the
unidimensional case of Theorem 1 in \cite{Djellout}.

\begin{prop}\label{propmdp} Let $(b_{n})$ a sequence satisfying
\[
b_{n} \mbox{ is increasing},\qquad \frac{b_{n}}{\sqrt{n}}\longrightarrow+
\infty,\qquad \frac{b_{n}}{n}\longrightarrow0,
\]
such that $c(n):=n/b_{n}$ is nondecreasing, and define the
reciprocal function $c^{-1}(t)$ by
\[
c^{-1}(t):=\inf\bigl\{n\in\mathbb{N}\dvtx  c(n)\geq t\bigr\}.
\]
Under the following conditions:
\begin{enumerate}[(C3)]
\item[(C1)] there exists $Q\in\mathbb{R}_{+}^{*}$ such that
$ \frac{\langle M\rangle_{n}}{n}\underset{
{b^2_n}/{n}}{\stackrel{\mathrm{superexp}}{\longrightarrow}} Q$;
\item[(C2)] $ \limsup_{n\rightarrow
+\infty}\frac{n}{b_{n}^{2}}\log(n \mathop{\operatorname{ess}\operatorname{sup}}
_{1\leq k\leq
c^{-1}(b_{n+1})}
\mathbb{P}(|M_{k}-M_{k-1}|>b_{n}/\break \mathcal{H}_{k-1}) )=-\infty$;
\item[(C3)] for all $ a>0$ $ \frac{1}{n}\sum_{k=1}^{n}\mathbb
{E} (|M_{k}-M_{k-1}|^{2}
\mathbf{1}_{\{|M_{k}-M_{k-1}|\geq
a{n}/{b_{n}}\}}/\break \mathcal{H}_{k-1} )\underset{
{b^2_n}/{n}}{\stackrel{\mathrm{superexp}}{\longrightarrow}} 0$;
\end{enumerate}
$ (M_{n}/b_{n} )_{n\in\N}$ satisfies the MDP in $\mathbb
{R}$ with the
speed $b_{n}^{2}/n$ and the rate function $ I(x) =\frac
{x^{2}}{2Q}$.
\end{prop}
\end{appendix}

\section*{Acknowledgments}

Let us thank two anonymous referees for their very careful reading and
useful suggestions, which have clearly improved both presentation and
mathematical rigor of the present paper.



\printaddresses

\end{document}